\documentclass[10pt,reqno,final]{amsart}
\usepackage{epsfig,amssymb,amsmath,version}
\usepackage{amssymb,version,graphicx,fancybox,mathrsfs}
\usepackage[notcite,notref]{showkeys}
\usepackage{url,hyperref}
\usepackage{subfigure}
\usepackage{color}
\usepackage{stmaryrd}
\usepackage{multirow}
\usepackage{booktabs,siunitx}

\usepackage{booktabs}

\catcode`\@=11 \theoremstyle{plain}
\@addtoreset{equation}{section}   
\renewcommand{\theequation}{\arabic{section}.\arabic{equation}}
\@addtoreset{figure}{section}
\renewcommand\thefigure{\thesection.\@arabic\c@figure}
\newtheorem{thm}{\bf Theorem}

\newenvironment{theorem}{\begin{thm}} {\end{thm}}
\newtheorem{cor}{\bf Corollary}

\newtheorem{lmm}{\bf Lemma}

\theoremstyle{remark}
\newtheorem{rem}{\bf Remark}[section]

\def \ri {{\rm i}}
\def \epsilon {{\varepsilon}}

\definecolor{bgblue}{rgb}{0.04,0.39,0.54}
\definecolor{lired}{rgb}{0.3, 0.0, 0.0}
\definecolor{ligreen}{rgb}{0.0, 0.3, 0.0}
\definecolor{liblue}{rgb}{0.9, 1.0, 1.0}
\definecolor{gray}{rgb}{0.6, 0.6, 0.6}
\definecolor{sky}{rgb}{0.3, 1.0, 1.0}
\definecolor{bunhong}{rgb}{1.0, 0.3, 1.0}
\definecolor{yellow}{rgb}{0.97, 1, 0.0}
\definecolor{liyellow}{rgb}{0.9, 0.8, 0.0}
\definecolor{cengse}{rgb}{0.00,0.40,0.29}

\def \cred {\color{red}}

\newcommand{\bs}[1]{\boldsymbol{#1}}

\renewcommand \wedge \times

\begin{document}
\bibliographystyle{plain}

{\title[Perfect absorbing layer for scattering problems] {A truly exact and optimal perfect absorbing layer   for time-harmonic acoustic wave scattering problems }
\author[
	Z. Yang,\,    L. Wang\,  $\&$\,  Y. Gao
	]{
		\;\; Zhiguo Yang${}^{1,2}$,   \;\;  Li-Lian Wang${}^{1}$ \;\; and\;\; Yang Gao${}^{1}$
		}
	\thanks{${}^{1}$Division of Mathematical Sciences, School of Physical
		and Mathematical Sciences, Nanyang Technological University,
		637371, Singapore. The research of the second author is partially supported by Singapore MOE AcRF Tier 2 Grants: MOE2017-T2-2-144 and MOE2018-T2-1-059. Emails: lilian@ntu.edu.sg (L. Wang) and GAOY0032@e.ntu.edu.sg (Y. Gao).
		 \\
		\indent ${}^{2}$Current address: Department of Mathematics, Purdue University, West Lafayette, IN 47907, USA. Email: yang1508@purdue.edu (Z. Yang)}


\begin{abstract} In this   paper, we design  a truly exact and optimal perfect absorbing layer (PAL)  for
domain truncation of the two-dimensional  Helmholtz equation in an unbounded domain with bounded scatterers. This technique  is based on a complex compression coordinate transformation in polar coordinates, and
a judicious substitution of the unknown field in the artificial layer. Compared with  the widely-used
perfectly matched layer (PML) methods,  the distinctive  features of PAL 
lie in that
(i)  it is truly exact in the sense  that the PAL-solution is identical to the original solution in the bounded domain reduced by the truncation layer;
(ii) with the substitution, the PAL-equation is free of singular coefficients  and the substituted unknown field   is
 essentially non-oscillatory
in the layer; and
(iii) the construction is valid for general star-shaped domain truncation.
By formulating the  variational formulation in  Cartesian coordinates,
the implementation of this technique using standard spectral-element or finite-element methods can be made  easy  as a usual coding practice. We provide ample numerical examples to demonstrate that this method is highly accurate, parameter-free and robust for very high wave-number and thin layer. It  outperforms  the classical PML and the recently advocated PML using unbounded absorbing functions. Moreover, it  can fix some flaws of the PML approach. 
\end{abstract}
 \maketitle

\section{Introduction}

Many physical and engineering problems involving wave propagations are naturally set in  unbounded domains. Accurate simulation of such problems becomes exceedingly important  in a variety of applications.
Typically, the  first  step is to reduce  the unbounded domain to a bounded domain so that most of finite-domain solvers can be  applied.  The reduced problem should be  well-posed, and the underlying solution must be
as close as possible to the original solution in the truncated  domain.   As such,  the development of efficient and robust domain truncation techniques  has become  a research topic of longstanding interest.
Several notable  techniques 
  have been intensively  studied in literature,  which particularly   include the artificial boundary conditions (ABCs) (see, e.g., \cite{bayliss1980radiation,EM77,F83,Givoli91,Hagstrom99,Nedelec2001Book,HaSt07}), and  artificial absorbing (or sponge) layers (see, e.g., \cite{Bere94,chew19943d,Bere96a, Bere96b,collino1998perfectly,Taflove2005Book}).

   In regards to the former approach,  the local ABCs are   easy to implement, but they can only provide   low order accuracy with undesirable reflections at times.
   Alternatively,  domain truncation based on the  Dirichlet-to-Neumann (DtN) map, gives rise to an equivalent boundary-value problem (BVP),  so it is transparent
(or equivalently  non-reflecting).  However,
  such an ABC is nonlocal in both space and time, which brings about  substantial complexity  in implementation.  Moreover, it is only available for special artificial boundaries/surfaces (e.g., circle and sphere).
  For example, significant effort  is needed to seamlessly integrate DtN ABC with curvilinear spectral elements in two dimensions (cf.  \cite{Fournier06,yang2015accurate,yang2016seamless}),  but  the extension to three dimensions is  highly  non-trivial (cf. \cite{Wang2017ArXiv}). It is also noteworthy that
with a good  tradeoff between accuracy and efficiency,    the high order ABCs without high-order derivatives  become  appealing \cite{Givoli03}.

Pertinent  to the latter approach,   the  perfect matched layer (PML) first  proposed by Berenger  \cite{Bere94,Bere96a, Bere96b}  essentially
builds  upon surrounding  bounded scatterers by an  artificial layer of finite width. The artificial layer is filled with fictitious absorbing media that can attenuate the outgoing waves inside. Since this pioneering works of Berenger,  the PML technique  has become  a widespread tool  for various wave simulations;
undergone in-depth  analysis of its mathematical ground (see, e.g., \cite{Las.S98,chen2005adaptive,chen2008adaptive,chen2013adaptive} and the references cited therein); and  been populated into major softwares such as the COMSOL Multiphysics.
In the past decade, this subject area continues to inspire new developments, just to name a few:  \cite{bermudez2007exact, bermudez2007optimal,FYZWL15, Deng.L17,Sun.L18,DZR18}.
Remarkably,   the  essential idea of constructing PML \cite{Bere94}  can be interpreted as a complex coordinate stretching (or transformation) \cite{chew19943d,collino1998perfectly}.  Consider for example the circular  PML in polar coordinates for the two-dimensional  Helmholtz problem \cite{collino1998perfectly,chen2005adaptive}, where
 the domain  of interest is truncated and  surrounded by an artificial  annulus of width $d,$ with the assumption that
 the inhomogeneity of the media and scatterers  are enclosed by a circle $C_R$ of radius $R.$   The corresponding (radial) complex coordinate transformation that generates the anisotropic media inside the annular layer is of the form
\begin{equation}\label{t-transform}
\tilde r=r+ \ri   \int_R^r \sigma (s)\,  {\rm d} s,\quad R<r <R+d,
\end{equation}
where $\sigma(s)>0$ is called the {\em absorbing function} (ABF).  Note that  the Helmholtz  problem inside $C_R$ remains unchanged, i.e., $\tilde r=r.$  One typical choice of the ABF is
 \begin{equation}\label{onechoice}
 \sigma (s)=\sigma_0 (s-R)^n\big/d^n,\quad n=0,1,\cdots, 
 \end{equation}
 where  $\sigma_0>0$ is a tuning parameter. 
 The PML truncates domain at a finite distance and  attenuates the wave (i.e., the original solution) in the  annular layer by enforcing homogeneous Dirichlet boundary condition at the outer boundary $r=R+d,$ where some artificial reflections are usually induced.
In theory, the reflection can be  a less important issue based on  the fundamental analysis (see, e.g., \cite{Las.S98,collino1998perfectly,chen2005adaptive}). Lassas and Somersalo \cite{Las.S98} showed that the PML-solution converges to the exact solution exponentially when the thickness of the layer tends to infinity.  As pointed out by Collino and Monk   \cite{collino1998perfectly},  it is important to optimally choose the parameters to reduce the potentially increasing error in the discretisation.
 The use of adaptive techniques can significantly enhance the performance of PML as advocated  by  Chen and Liu \cite{chen2005adaptive}.
 It is noteworthy that the complex coordinate transformation \eqref{t-transform}-\eqref{onechoice} is also used  in developing the uniaxial PML in Cartesian coordinates along each coordinate direction \cite{singer2004perfectly,chen2013adaptive,chen2017PML}.
However, Singer and Turkel \cite{singer2004perfectly} demonstrated that such a transformation  can magnify the evanescent waves in the waveguide setting.  Recently, Zhou and Wu \cite{zhou2018adaptive} proposed to combine  the PML with few-mode DtN truncation to deal with the evanescent wave components.  It is noteworthy that according to \cite{pmlnotes,loh2009fundamental,FYZWL15}, the PML  has flaws and failures at times.    

An intriguing advancement is the ``exact''  and ``optimal" PML  using  {\em unbounded} ABFs $\sigma(\cdot)$  (see \eqref{UPMLun} and \eqref{SingleKernel} below),  developed by  Berm{\'u}dez et al. \cite{bermudez2007exact,bermudez2007optimal}  for time-harmonic acoustic wave scattering problems.  Indeed, it was shown in  \cite{Rabinovich2010C,RAM2015} through sophisticated comparison with the  classical PML  that  it  can be  ``parameter-free'' and has some other advantages.  However,  according to the error analysis in  \cite{singer2004perfectly}  (also see  
Theorem \ref{solu-error}),  this technique fails to be exact for the waveguide problem, but can improve the accuracy of the classical PML.  On the other hand, the unbounded ABFs lead to PML-equations with singular coefficients so much care is needed  to deal with the singularities, in particular,  for high wave-numbers and thin layers.


 In this paper, we propose a truly exact and optimal perfect absorbing layer  with  general star-shaped domain truncation of the exterior Helmholtz equation in  an unbounded domain.  
In spirit of our conference paper \cite{wangyangpal17},  its construction consists of  two indispensable building blocks:
\begin{itemize}
\item[(i)] Different from \eqref{t-transform}, we use a complex compression coordinate transformation of the form: $\tilde  r=\rho(r)+\ri (\rho(r)-R).$ Here $\rho(r)$ is a real mapping that compresses $\rho\in (R,\infty)$ into $r\in (R,R+d)$ along radial direction; 
\item [(ii)]  We introduce  a suitable substitution of the field in the artificial layer: $U=wV,$ where $w$   extracts the essential oscillation of $U,$ and also includes a singular factor to deal with the singular coefficients (induced by the coordinate transformation) of the resulted PAL-equation.  The field $V$ to be approximated is  well-behaved in the  layer.  We formulate the variational form of the PAL-equation (designed in polar coordinates) in Cartesian coordinates, so it is friendly for the implementation with the finite-element or  spectral-element solvers. 
\end{itemize}

We remark that the use of compression coordinate transformation is inspired by  the notion  of  ``inside-out''  invisibility cloak \cite{zharova2008inside} where   a real rational  mapping  was used to  compress an open space into a finite cloaking layer.  However, this cloaking device is far from perfect  (cf. \cite{wangyangpal17}).  To generate a perfect cloaking layer, we employ the complex compression mapping like the complex coordinator stretching in PML, in order to attenuate the compressed outgoing waves. Notably, we can show the truncation by the PAL is truly exact in the sense that the PAL-solution is identical to the original solution in the inner domain (exterior to the scatterer but inside the inner boundary of the artificial layer).  
The substitution in (ii) turns out critical for the success of  PAL technique for the reasons that this can  overcome the numerical difficulties of dealing with  singular coefficients  of the PAL-equation  and  remove the oscillations near the inner boundary of the layer.  Indeed, both the  analysis and ample numerical  evidences show that the new PAL method is highly accurate even for high wavenumber and thin layers.


The rest of the paper is organised  as follows. In Section \ref{sect::waveguide}, we present the essential idea of PAL using a waveguide problem in a semi-infinite domain. A delicate error estimate has been derived  for the PML and PAL methods,  and a comparison study has been conducted to PAL and PML with regular and unbounded  absorbing functions. In Section \ref{sect3:domain}, we start with a general set-up for the star-shaped truncated domain for exterior wave scattering problems, and provide new perspectives of the circular PAL reported in \cite{wangyangpal17}.    In Section \ref{lm:Sect4},  we provide in  detail the construction of the PAL-equation based on the complex compression coordinate transformation and the variable substitution technique used to eliminate the oscillation and singularity. In Section  \ref{sect5:numer}, ample numerical experiments are provided to demonstrate the performance of the PAL method, and show its advantages over the PML methods. 

%

\section{Waveguide problem in a semi-infinite channel}\label{sect::waveguide}

In this section, we elaborate on the essential idea of the new PAL,  and compare it with the classical  PML 
(cf.  \cite{Bere94,chew19943d,collino1998perfectly,singer2004perfectly}), and the ``exact'' and ``optimal" PML techniques  using unbounded (singular) absorbing functions (cf. Berm{\'u}dez et al.   \cite{bermudez2007exact,bermudez2007optimal})  in the waveguide setting. Indeed, such a relatively simpler context
enables us to conduct a precise error analysis, and better understand  the significant differences between two approaches.

To this end, we consider  the semi-infinite $x$-aligned waveguide from $x=0$ to $x=\infty$ (see Figure \ref{guidegeometry} (left)), governed by  the Helmholtz equation (cf.  \cite{singer2004perfectly,Rabinovich2010C}):
\begin{subequations}\label{extproblem}
\begin{align}
&\mathscr{L}[U]:=\Delta U+k^2 U=0\;\;\; {\rm in}\;\;\; \Omega_{\infty}:=\big\{ 0< x< \infty,\;  0< y < \pi \big\}; \label{ext1}\\[4pt]
&U(x,0)=U(x,\pi)=0,\;\; x\in [0,\infty);\;\;\;  U(0,y)=g(y), \;\;\;  y\in [0,\pi],  \label{extbd}
\end{align}
\end{subequations}
where the wave number $k>0,$ the field $U$ is outgoing,  and  $g\in L^2(0,\pi)$.  
We refer to Goldstein \cite{goldstein1982finite} for the outgoing radiation condition:
\begin{equation}\label{outgoing1}
 \frac{\partial U}{\partial x}- \sum_{l=1}^{\infty}  {\ri \hat k}\, \hat c_l\,  e^{\ri  \hat k_l x}  \sin(ly)=0 \;\;\;  {\rm with}\;\;\; U(x,y)=\sum_{l=1}^{\infty}  \hat c_l\,  e^{\ri  \hat k_l x}  \sin(ly),
\end{equation}
for any $x\ge x_1,$ 
where $\hat k_l:=\sqrt{k^2-l^2},$ and $\{\hat c_l\}$ can be determined by  given Dirichlet  data  at  $x=x_1$.
 Using the Fourier sine expansion in $y$-direction,  we find  readily  that   the problem
\eqref{extproblem}-\eqref{outgoing1}  admits the  series solution:
\begin{equation}\label{uxyexact}
U(x,y)=\sum_{l=1}^\infty  \hat g_l\, e^{\ri  \hat k_l x}\sin (ly) \;\;\; {\rm with}\;\;\; \hat g_l=\frac 2 \pi \int_0^\pi g(y) \sin(ly)\, {\rm d}y.
\end{equation}

\begin{figure}[!h]
\centering
\rotatebox[origin=cc]{-0}{\includegraphics[width=0.35\textwidth]{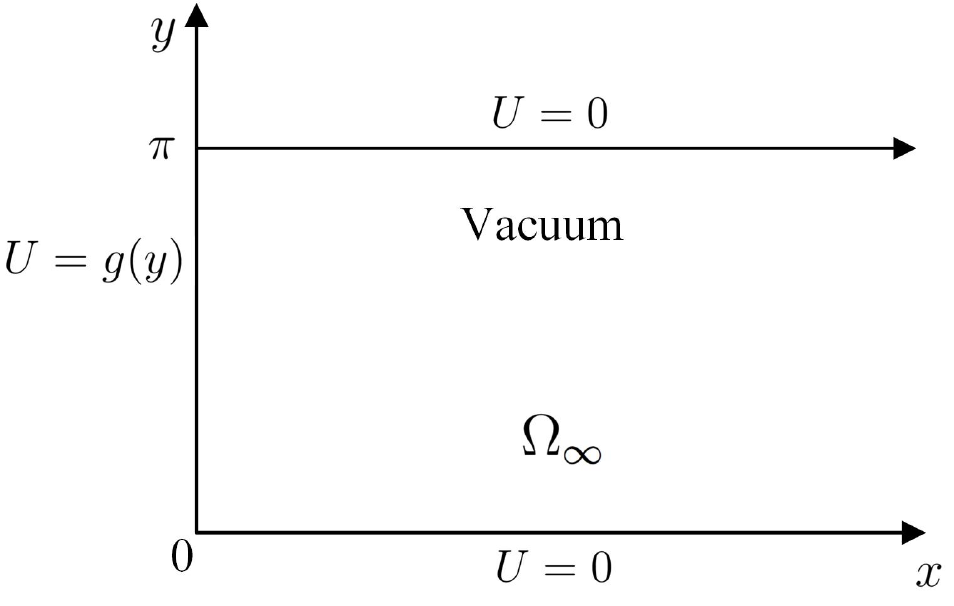}}\qquad 
\rotatebox[origin=cc]{-0}{\includegraphics[width=0.44\textwidth]{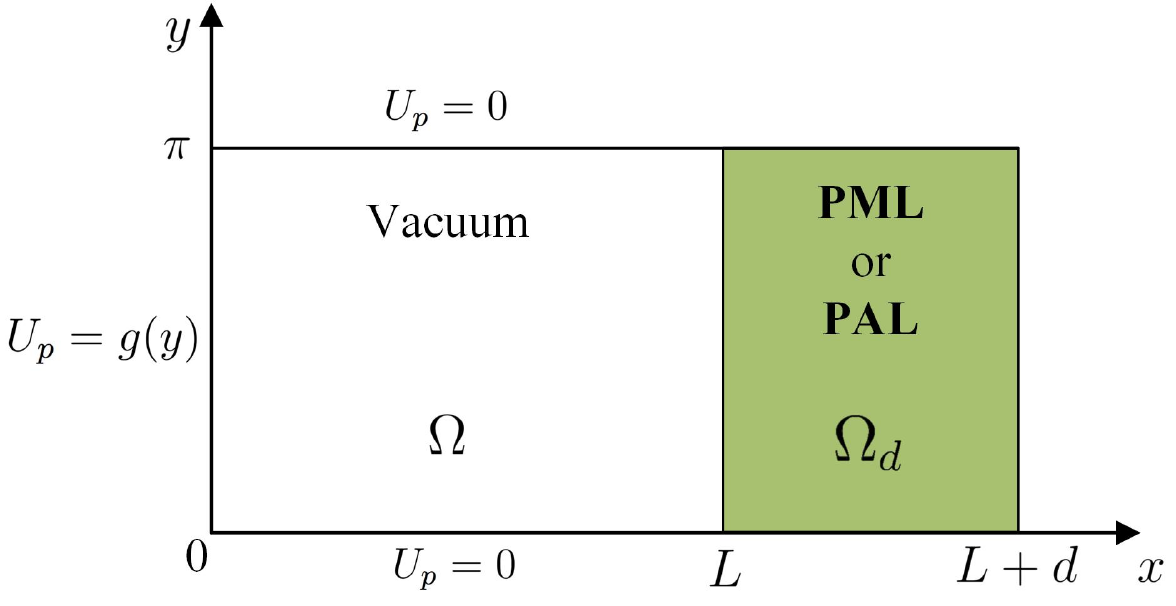}}
\caption{\small  Left: schematic diagram of  $x$-aligned semi-infinite waveguide. Right:  domain reduction by an artificial layer  using the PML or PAL technique.}
\label{guidegeometry}
 \end{figure}


\subsection{The PML technique  and its error analysis}\label{sect:waveguide}
To  solve \eqref{extproblem} numerically, one commonly-used approach is  the PML technique 
 which  reduces  the semi-infinite  strip   $\Omega_{\infty}$ in \eqref{extproblem} to
a rectangular domain $\Omega:=(0,L)\times (0,\pi)$  by  appending an artificial layer $\Omega_{d}:=(L,L+d)\times (0,\pi)$ with a finite thickness $d$  (see Figure \ref{guidegeometry} (right)). Typically, 
the  artificial layer $\Omega_d$ is  filled  with fictitious absorbing (or lossy) media that can attenuate  the  waves propagating into the layer. In practice, one would wish the layer can  diminish  the pollution (due to the reflection) of the original solution  in the ``physical domain''
$\Omega,$ and its thickness $d$ can be as small as possible to save computational cost.

The critical issue is to construct the governing equation in $\Omega_d.$ 
It is known that the PML-equation  can be obtained  from the Helmholtz equation: $(\Delta +k^2) U=0$ in $(\tilde x,\tilde y)$-coordinates through the complex coordinate stretching  (see, e.g.,  \cite{collino1998perfectly,singer2004perfectly}).  More precisely, we introduce the  complex  coordinate  transformation of the form:
  \begin{equation}\label{Scondf}
 \begin{cases}
 \tilde x=S(x),\quad \tilde y=y,  \quad \forall\, (x,y)\in \Omega\cup \Omega_d \quad \text{such that}
\\[3pt] S(x)=x, \;\;\; \forall\,  x\in (0,L);  \quad   \Re\{S'(x)\}>0,\  \Im\{S'(x)\}>0, \;\;\;  \forall\,  x\in (L,L+d),
 \end{cases}
 \end{equation}
where  $S(x)$ is a differentiable complex-valued  function.  One typically  chooses 
\begin{equation}\label{typicalcase}
S(x)=x+\frac{\ri} {k} \displaystyle \int_0^x \sigma(t)\, {\rm d}t, \;\;\; \forall\,  x\in (0,L+d),
\end{equation}
where $\sigma(\cdot)\ge 0$ is known as the {\em absorbing  function} (ABF).  
Then the  substitution
\begin{equation}\label{substituA}
\frac{\partial} {\partial x}\to \frac{\partial} {\partial \tilde x}=\frac{dx} {d\tilde x}  \frac{\partial} {\partial x}=\frac 1 {S'(x)} \frac{\partial} {\partial x} ,
\end{equation}
leads to 
 the PML-equation  (cf. \cite[(5)]{singer2004perfectly}):
\begin{equation}\label{PMLeq1}
\frac{\partial} {\partial x} \Big(\frac{1} {S'(x)} \frac{\partial U_{\rm p}}{\partial x} \Big)+
\frac{\partial} {\partial y} \Big( S'(x) \frac{\partial U_{\rm p}}{\partial y} \Big) +k^2 S'(x) U_{\rm p}=0, \quad \forall\,(x,y)\in \Omega\cup \Omega_d,
\end{equation}
which is supplemented with the transmission conditions at the interface  $x=L$:
\begin{equation}\label{tranmissionGuide}
U_{\rm p}\big|_{\Omega}= U_{\rm p}\big|_{\Omega_d},\quad   \frac{\partial U_{\rm p}}{\partial x}\Big|_{\Omega}=
 \frac{1} {S'(x)} \frac{\partial U_{\rm p}}{\partial x} \Big|_{\Omega_d},
\end{equation}
together with  the boundary conditions:
\begin{equation}\label{Upeqna}
\begin{split}
& U_{\rm p}(x,0)=U_{\rm p}(x,\pi)=0,\;\;\;\;  x\in (0,L+d); \quad\\
&  U_{\rm p}(0,y)=g(y), \;\;\;\;   U_{\rm p}(L+d,y)=0,\;\;\;\; y\in (0,\pi).
\end{split}
\end{equation}
It is noteworthy that with a suitable choice of the ABF $\sigma(t),$  the PML-solution  decays sufficiently fast  in the artificial layer,  so the homogeneous Dirichlet boundary condition is usually imposed at  $x=L+d.$

Singer and Turkel \cite{singer2004perfectly} conducted error analysis of the  PML with usual bounded ABFs.  Here, we provide a much more precise description  of the error between the original solution $U$ and the PML-solution $U_{\rm p}$ in the physical domain $\Omega$. 
\begin{thm}\label{solu-error}
 Let $U$ be the solution of \eqref{extproblem} given by \eqref{uxyexact}, and let $U_{\rm p}$ be the PML-solution of
  \eqref{PMLeq1}-\eqref{Upeqna}.
If   $k>0$ is not an  integer, then for any $(x,y)\in \Omega,$
\begin{equation}\label{erroreqnUUp}
U(x,y)=U_{\rm p}(x,y)+\sum_{l=1}^\infty \hat g_l\, R_l(x)\, e^{\ri  \hat k_l x}\sin (ly),\quad   R_l(x):= \frac{1- e^{-2\ri \hat k_{l} x} }{1-e^{- 2\ri \hat k_{l} S_d}}.
\end{equation}
If $k>0$ is an  integer,  then $R_l(x)$ for the mode $l=k$ in \eqref{erroreqnUUp} should be replaced by
\begin{equation}\label{newUp0}
 R_l(x)=-\frac{S(x)}{S_d},\quad \forall\, x\in (0,L).
\end{equation}
Moreover,  we have the following  bounds:
\begin{itemize}
\item[{\rm (i)}] for $0\le l<k,$
\begin{equation}\label{URerrors}
\frac{2 |\sin(\hat k_l x)|} {e^{2 \hat k_l  \Im\{S_d\}}+1} \le  |R_l(x)|\le \frac{2 |\sin(\hat k_l x)|} {e^{2 \hat k_l {\Im}\{ S_d\}}-1}, \quad \forall\, x\in (0,L); 
\end{equation}
\item[{\rm (ii)}] for $l>k,$
\begin{equation}\label{URerrors2}
\frac{e^{2|\hat k_l| x}-1} {e^{2 |\hat k_l|  {\Re}\{S_d\}}+1} \le  |R_l(x)|\le \frac{e^{2|\hat k_l| x}-1} {e^{2 |\hat k_l|  {\Re}\{S_d\}}-1},  \quad \forall\, x\in (0,L).
\end{equation}
\end{itemize}
In the above,  $\hat k_l:=\sqrt{k^2-l^2},\,   S_d:=S(L+d)$ and $S(x)$ is given by \eqref{Scondf}-\eqref{typicalcase}.
\end{thm}

To avoid distraction from the main results,  we sketch the proof  in   Appendix \ref{AppendixA0}.
We see that for fixed $k, l,$  the decay rate of $R_{l}(x)$ is  completely determined by the values of
${\Re}\{S_d\}$ and   ${\Im}\{ S_d\},$ i.e., the choice of ABF $\sigma(t).$
We now  apply  Theorem \ref{solu-error} to access the performance of PML with some typical ABFs:
\begin{equation}\label{PMLclassic}
\text{PML}_n:\qquad \sigma(x)=\begin{cases}
0,\quad &{\rm if}\;\; x\in (0,L),\\[4pt]
\sigma_0 \Big(\dfrac{x-L}{d} \Big)^n,\quad & {\rm if}\;\; x\in(L,L+d),
\end{cases}
\end{equation}
  and
\begin{equation}\label{UPMLun}
\text{PML}_\infty:\qquad \sigma(x)=\begin{cases}
0,\quad &{\rm if}\;\; x\in (0,L),\\[4pt]
 \dfrac{ \sigma_0\, d}{L+d-x}, \;\;\;      & {\rm if}\;\; x\in(L,L+d),
\end{cases}
\end{equation}
where $\sigma_0>0$ is a tuneable  constant. 
We remark that  $\text{PML}_n$ with integer $n\ge 0,$ is commonly-used  (see, e.g.,  \cite{collino1998perfectly,singer2004perfectly,chen2005adaptive}), while 
 $\text{PML}_\infty$ was first introduced by Berm{\'u}dez et al.   \cite{bermudez2007exact,bermudez2007optimal} (which was advocated for  its being ``optimal", ``exact" and ``parameter-free" in various settings (cf.
 \cite{bermudez2007exact,bermudez2007optimal,Rabinovich2010C,Modave2014O,Cimpeanu2015A})).
It is clear that in $\Omega_d$,  we have
\begin{equation}\label{PMLclassic20}
\text{PML}_n:\;\;\;  S(x)=
x+\frac\ri k  \dfrac{d\,\sigma_0 }{n+1}\Big( \dfrac{x-L}{d}  \Big)^{n+1},\quad {\Re}\{S_d\}=L+d,\quad
 {\Im}\{S_d\}=\frac {\sigma_0\, d }{(n+1)k},
\end{equation}
and 
\begin{equation}\label{UPMLS}
\text{PML}_\infty:\;\;\; S(x)=x+\frac{\ri\, d\, \sigma_0 } k \ln\Big(\dfrac{d} {L+d-x}\Big),
\quad {\Re}\{S_d\}=L+d,\quad
 {\Im}\{S_d\}\to \infty.
\end{equation}

Observe that in both cases,  ${\Re}\{S_d\}$ does not depend on the choice of  ABF $\sigma(t)$. Consequently,   the errors corresponding to the evanescent wave components (i.e., $l>k$)  behave like 
\begin{equation}\label{lkmwave}
|R_l(x)|\sim  e^{-2\sqrt{l^2-k^2} (L+d-x)},\quad \forall\, x\in (0,L),
\end{equation}
so the only possibility to reduce the error is to increase the width $d.$

For the oscillatory wave components (i.e., $l<k$),  the PML$_\infty$ can make the error vanish (note: $R_l(x)=0$),  but for the PML$_n,$
\begin{equation}\label{FindUfrom2}
|R_l(x)|\sim
  2 \big|\sin \big(\sqrt{k^2-l^2} x\big)\big| e^{-2\sqrt{k^2-l^2}  {\sigma_0\,d}/{((n+1)k)}},\;\;\; \forall\, x\in (0,L),
\end{equation}
so one has to enlarge  the thickness  $d$ of the layer or choose large $\sigma_0\,d$ to reduce  the error.

As an illustrative example, we consider the exact solution  \eqref{extproblem}  with
$\{\hat g_l=\ri ^l J_l(k)\}$ (in order to mimic the plane wave expansion):
\begin{equation}\label{Utest}
U(x,y)=\sum_{l=1}^{\infty}\ri ^l J_l(k) e^{\ri \sqrt{k^2-l^2} x}\sin(ly), 
\end{equation}
where $J_{l}(\cdot)$ is the Bessel function. 
Define the error function:
\begin{equation}\label{erroreqnUUp2}
{\mathcal E}_{\rm p}(x,y)=U(x,y)-U_{\rm p}(x,y)=\sum_{l=1}^\infty \ri ^l J_l(k)\, R_l(x)\, e^{\ri  \sqrt{k^2-l^2} x}\sin (ly),\quad  \forall\, (x,y)\in \Omega.
\end{equation}
Here, we take $k=9.99, L=1, d=0.1,$ and truncate the  infinite series
for  $l\le 100$ so that the truncation error is negligible.   The error curves  $|{\mathcal E}_{\rm p}(x,y)|$ with fixed  $x=L/2, L$
and $y\in (0,\pi)$ in Figure \ref{errfunction} with PML$_2$ and PML$_\infty$ in \eqref{PMLclassic}-\eqref{UPMLun}.

%

\begin{figure}[!th]
\begin{center}
\rotatebox[origin=cc]{-0}{\includegraphics[width=0.4\textwidth,height=0.28\textwidth]{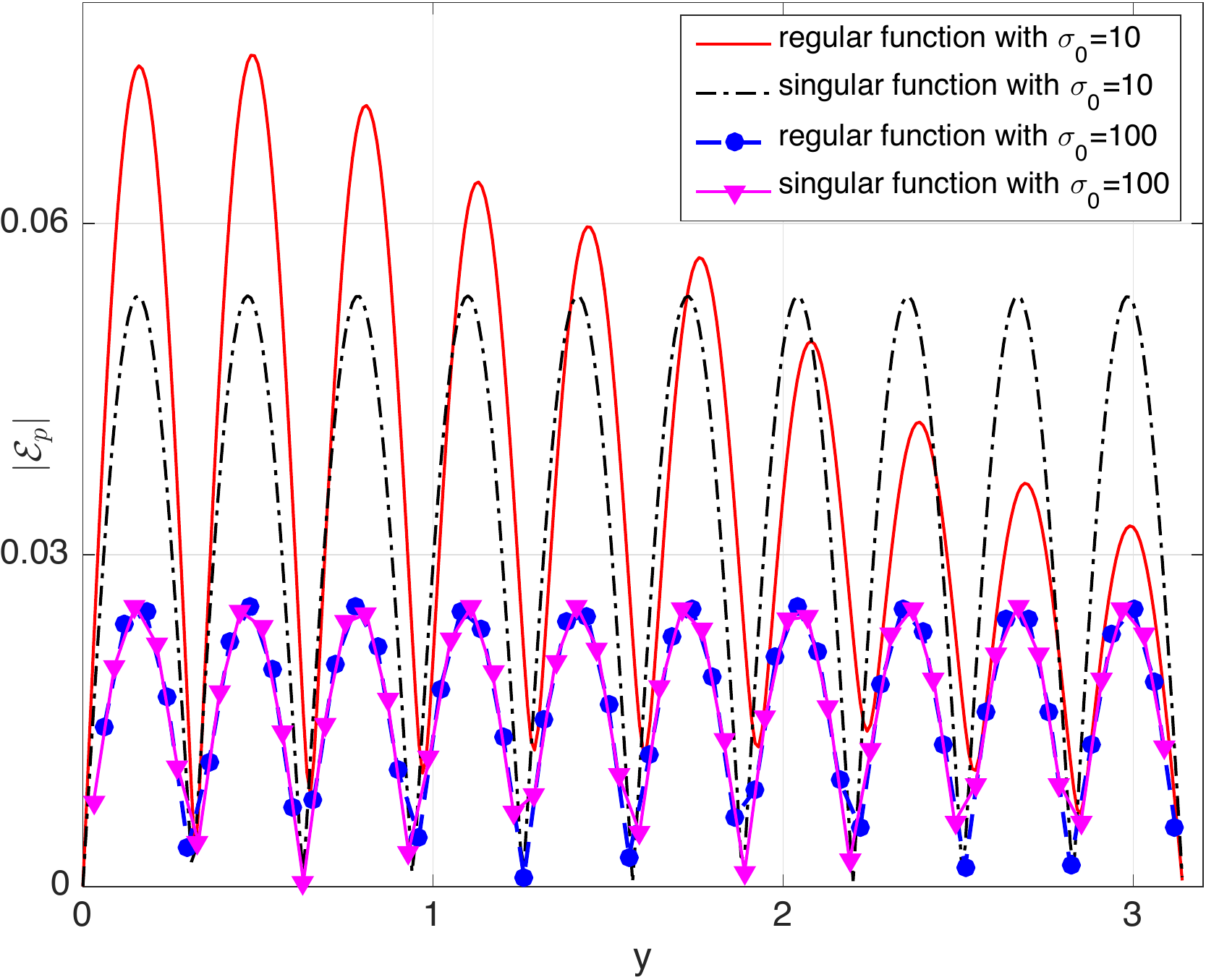}} \qquad
\rotatebox[origin=cc]{-0}{\includegraphics[width=0.4\textwidth,height=0.28\textwidth]{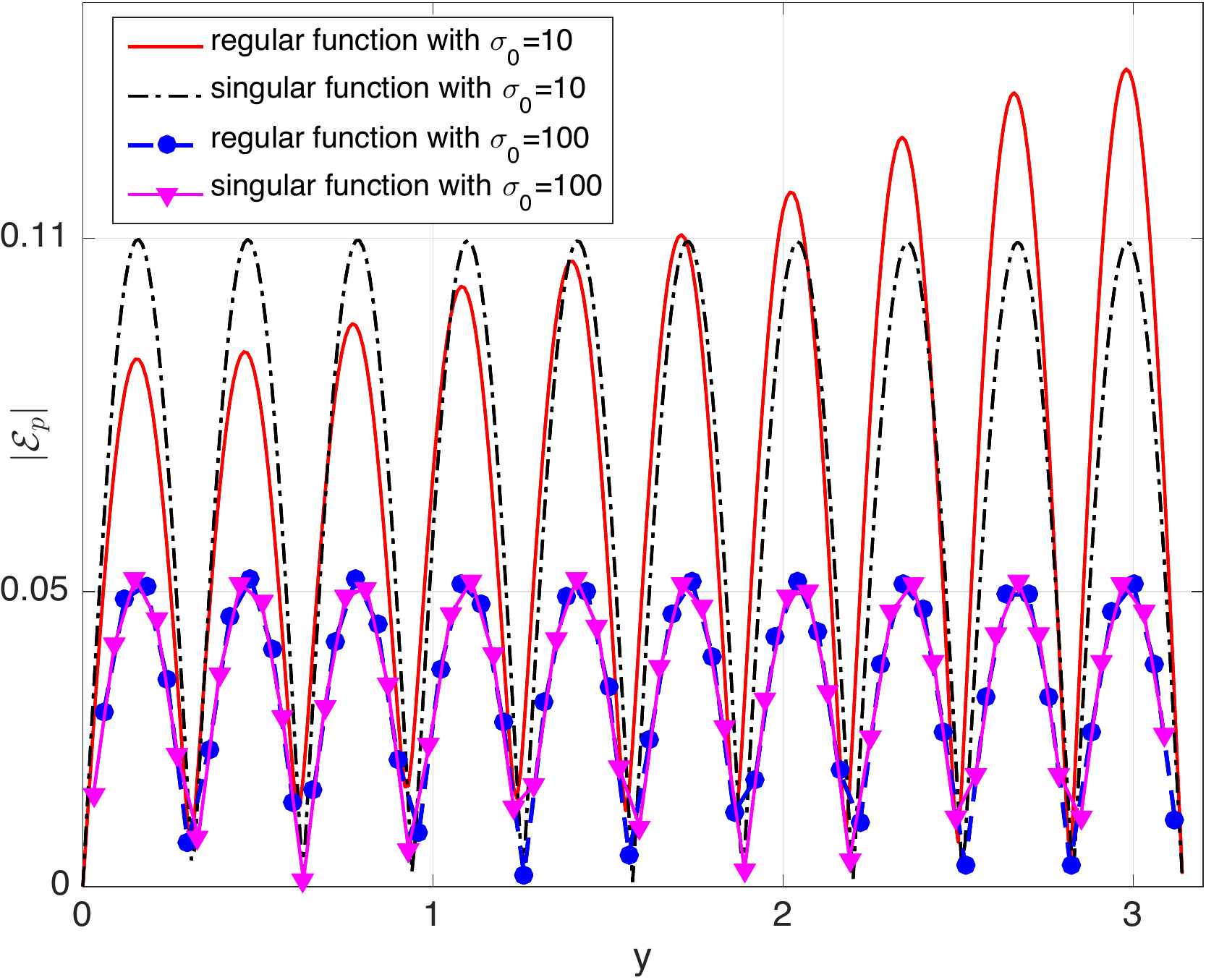}}
\caption{\small Profiles  of the error function $|\mathcal{E}_p(x,y)|$ of the solution \eqref{Utest} with $k=9.99,L=1, d=0.1,$ and $\sigma_0/k=10,100$ using   PML$_2$ and PML$_\infty$.     Left: the error curve
$|\mathcal{E}_p(L/2,y)|$ for $y\in (0,\pi).$  Right: the error curve: $|\mathcal{E}_p(L,y)|$ for  $y\in (0,\pi).$   }
\label{errfunction}
\end{center}
 \end{figure}

Observe from Figure \ref{errfunction}  that  the errors of the PML$_2$  become slightly smaller as $\sigma_0$ increases,  but they are still  large. The PML$_\infty$ using unbounded absorbing functions
 performs relatively  better, but does not  significantly improve the classical PML.
 
 \begin{rem} {\em The recent work  \cite{zhou2018adaptive} introduced the few-mode DtN technique to deal with the evanescent wave components, while the other modes were treated with the classical PML$_n$. }
 \end{rem}
 
 \begin{rem}\label{oneDversion} {\em As shown in \cite{bermudez2007exact}, the PML$_\infty$ can completely damp the planar waves of the form: $e^{\ri k_x x+\ri k_y y}$ with $k_x=k\cos\theta_0, k_y=k\sin\theta_0,$ where $\theta_0$ is the incident angle at the left boundary $x=0$ of the half-plane waveguide: $\Delta u+k^2u=0, \; x>0, \; -\infty<y<\infty.$ We also refer to \cite{Cimpeanu2015A} for some other successful scenarios.    However, it fails to be ``exact'' and ``optimal'' in this setting.  }
  \end{rem}

\begin{rem}\label{PMLremark}{\em Note from \eqref{UPMLS} that for the PML$_\infty,$  $S'(x)=1+\ri\, \sigma(x)/k \to \infty$ as $x\to L+d,$
so the coefficients in \eqref{PMLeq1} are singular at the outer boundary $x=L+d.$
In \cite{bermudez2007exact,bermudez2007optimal},  the use of e.g, Gauss quadrature to avoid  sampling the unbounded  endpoints was suggested for evaluating the matrices of the linear system in finite-element discretisation. However,  for large wavenumber $k$ and very thin layer,
 much care is needed to deal with  the singularity to achieve high order. Moreover, in more complex situations, e.g., the circular/spherical PML, $S(x)$ also appears in the PML-equation, so the logarithmic singularity poses even more challenge  in numerical discretisation. }
\end{rem}

\subsection{New PAL technique and its error analysis}\label{Subsect:22} As  reported in  \cite{wangyangpal17}, the new PAL  was inspired by the design of the ``inside-out" invisibility cloak  (cf. \cite{zharova2008inside}) using the notion of transformation optics (cf. \cite{Pendry.2006}).
 In  \cite{zharova2008inside},  the real rational transformation was introduced  to construct  the media and design the clocking layer:
\begin{equation}\label{newguids}
\rho=\rho(x)=L+\frac{d(x-L)}{L+d-x},\quad \, x \in (L, L+d),\;\;\; \rho\in (L,\infty),
\end{equation}
which compresses the outgoing waves  in  the infinite strip: $\rho>L$ into the finite layer: $(L, L+d).$
However, it is known that any attempt of using a real compression coordinate transform fails to work for Helmholtz and Maxwell's equations.  Indeed,  according to \cite{pmlnotes},  {``any real coordinate mapping from an infinite to a finite domain will result in solutions that oscillate infinitely fast as  the boundary is approached -- such fast oscillations cannot be represented by any finite-resolution grid, and will instead effectively form a reflecting hard wall.'' }


   To break the curse of infinite oscillation,  we propose  the  complex compression coordinate transformation (C$^3$T), that is,
    $S(x)=x$ for $x\in (0,L),$ and
\begin{equation}\label{SxformNew}
S(x)=S_{R}(x)+{\ri}\,  S_{I}(x), \;\;\;\;      x\in (L, L+d),
\end{equation}
where the real part:  $S_R(x)= \sigma_1 \rho(x)$ and the imaginary part: $S_I(x)=\sigma_0(\rho(x)-L)/k.$
Like \eqref{UPMLun}, the imaginary part $S_I(x)$ involves an unbounded ABF:
\begin{equation}\label{ourchoice}
   S_I(x)= 
   \frac 1 k \int_L^x \sigma(t)\, {\rm d}t, \quad \sigma(x)= \sigma_0\, \Big(\frac{d}{L+d-x}\Big)^2,\quad x\in (L, L+d).
\end{equation}
For clarity, we denote the PAL-solution by $U_{\rm a}.$ Thanks to \eqref{substituA}, we obtain the PAL-equation  as the counterpart of  \eqref{PMLeq1}-\eqref{Upeqna}:
\begin{equation}\label{PALsystem}
\begin{split}
& \frac{\partial} {\partial x} \Big(\frac{1} {S'(x)} \frac{\partial U_{\rm a}}{\partial x} \Big)+
\frac{\partial} {\partial y} \Big( S'(x) \frac{\partial U_{\rm a}}{\partial y} \Big) +k^2 S'(x) U_{\rm a}=0, \quad (x,y)\in \Omega\cup \Omega_d,\\[4pt]
& U_{\rm a}\big|_{\Omega}= U_{\rm a}\big|_{\Omega_d},\quad   \frac{\partial U_{\rm a}}{\partial x}\Big|_{\Omega}=
 \frac{1} {S'(x)} \frac{\partial U_{\rm a}}{\partial x} \Big|_{\Omega_d}\quad {\rm at}\;\;\;  x=L,\\[4pt]
& U_{\rm a}(x,0)=U_{\rm a}(x,\pi)=0; \quad U_{\rm a}(0,y)=g(y), \;\;  U_{\rm a}(L+d,y)=0.
\end{split}
\end{equation}
Importantly, we can show that the new PAL  is truly exact and non-reflecting. 
\begin{thm}\label{exact} Let $U$ be the solution of \eqref{extproblem}, and  $U_{\rm a}$ be the solution of \eqref{PALsystem}. Then we have
 \begin{equation}\label{equivla}
 U_{\rm a}(x,y)\equiv U(x,y),\quad \forall\, (x,y)\in \Omega.
 \end{equation}
\end{thm}
\begin{proof} Note that the error formula  in Theorem \ref{solu-error} is valid for   general  $S(x)$.  It is evident that by  \eqref{SxformNew},   we have ${\Re}\{S(L+d)\}=\infty$ and ${\Im}\{S(L+d)\}=\infty.$
Thus, it follows immediately from \eqref{URerrors}-\eqref{URerrors2} that
$R_l(x)=0$ uniformly for all $l$ and $x\in (0,L),$ which implies  $U_{\rm a}(x,y)\equiv U(x,y)$ in $\Omega.$
\end{proof}

Different from PML,  the truncation by the  new PAL is exact at continuous level, but the compression coordinate transformation induces singular coefficients at $x=L+d$ in the PAL-equation \eqref{PALsystem}, which causes some numerical difficulties in discretization. To  overcome this, we introduce the substitution of the unknown:
\begin{equation}\label{uxycaseBA}
U_{\rm a}(x,y)=w(x) V_{\rm a}(x,y),\quad   w(x)=
\dfrac{L+d-x}{d},\;\;\;  x\in (L,L+d),
\end{equation}
with  $w(x)=1$ for $x\in (0,L).$
In fact,  the transformed PAL-equation in the new unknown $V_a$ is free of singularity.  
To show this, it's more convenient to  work with the variational  form.  
Denote
$${\mathbb U}:=\big\{u=wv : v\in H^1(\Omega\cup \Omega_d), \; v(x,0)=v(x,\pi)=0,\;\; x\in (0,L+d) \big\},$$
and  introduce  the  sesquilinear form on ${\mathbb U}\times {\mathbb U}$:
\begin{equation}\label{caseA}
{\mathcal B}(U_{\rm a},\varPsi)=(\bs C\, \nabla U_{\rm a}, \nabla \varPsi) -k^2(S'\, U_{\rm a},\varPsi),
\end{equation}
where $\bs C={\rm diag}(1/S'(x),S'(x)),$ and $(\cdot,\cdot)$ is the inner product of $L^2(\Omega\cup\Omega_d).$

 In view of  \eqref{ourchoice},  a direct calculation leads to   $ w^2(x) S'(x)=\alpha=\sigma_1+\ri \sigma_0/k.$
Then, substituting $U_{\rm a}=wV_{\rm a}$  and $\varPsi=w \varPhi$ into \eqref{caseA}, we obtain from  direct calculation that
\begin{equation}\label{newformA}
\begin{split}
\widetilde {\mathcal B}_{_{\Omega_d}}(V_{\rm a},\varPhi)={\mathcal B}_{_{\Omega_d}}(wV_{\rm a},w \varPhi)=
(\widetilde {\bs C}\, \widetilde \nabla V_{\rm a}, \widetilde \nabla \varPhi)_{\Omega_d} -\alpha k^2( V_{\rm a},\varPhi)_{\Omega_d},
\end{split}
\end{equation}
where $(\cdot,\cdot)_{\Omega_d}$ is the inner product on the artificial layer  $\Omega_d,$ and
\begin{equation}\label{defnUaPhi}
\widetilde {\bs C}={\rm diag}\big(w^2(x)/\alpha,\; \alpha\big),\quad  \widetilde \nabla=\big(w(x)\partial_x-1/d, w(x)\partial_y\big).
\end{equation}
It is seen that the substitution can absorb the singular coefficients.  In the implementation,
 one can easily  incorporate the substitution into the basis functions and directly approximate $U_{\rm a}$. 

We conclude this section with some numerical results. 
Consider \eqref{extproblem} with $k=29.9$ and the boundary source term is prescribed as $g=\sin(5y)-\sin(30y)$ in \eqref{extbd}. The semi-infinite  strip   $\Omega_{\infty}$ in \eqref{extproblem} is reduced to a rectangular domain $\Omega:=(0,1)\times (0,\pi)$  by  appending the PAL layer $\Omega_{d}:=(1,1+d)\times (0,\pi)$ with a finite thickness $d$. Spectral element method based on the sesquilinear form \eqref{newformA} is adopted for computation. Numerical results obtained by PAL ($\sigma_0=\sigma_1=1$) are also compared with PML technique with bounded and unbounded absorbing functions, i.e., PML$_n$ ($n=1$, $\sigma_0=10$) and PML$_\infty$ ($\sigma_0=10$) in \eqref{RegularKernel}-\eqref{SingleKernel}, respectively.   

\begin{figure}[htbp]
\begin{center}
  \subfigure[$d=0.1$ ]{ \includegraphics[scale=.3]{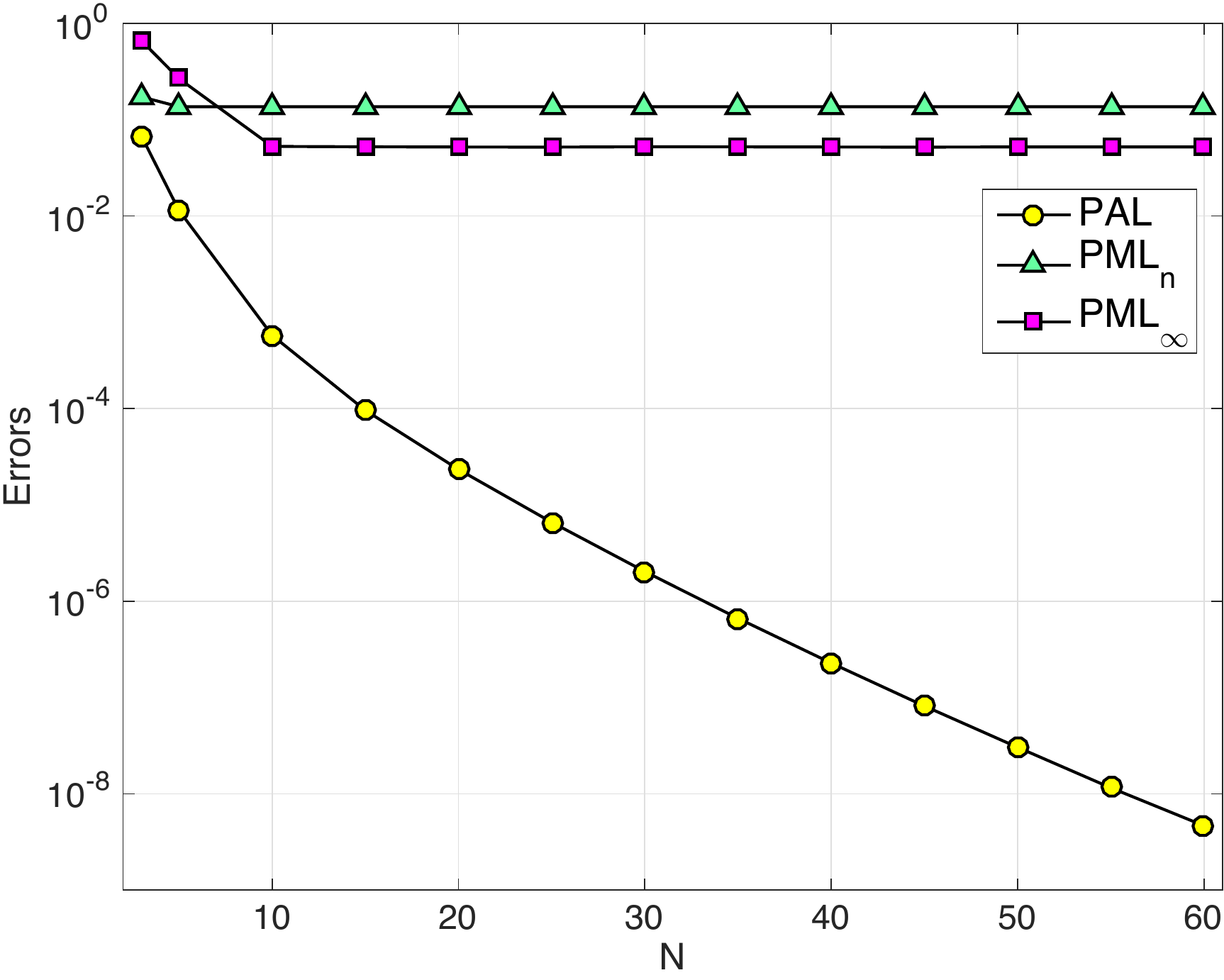}}\qquad
  \subfigure[$d=0.5$ ]{ \includegraphics[scale=.3]{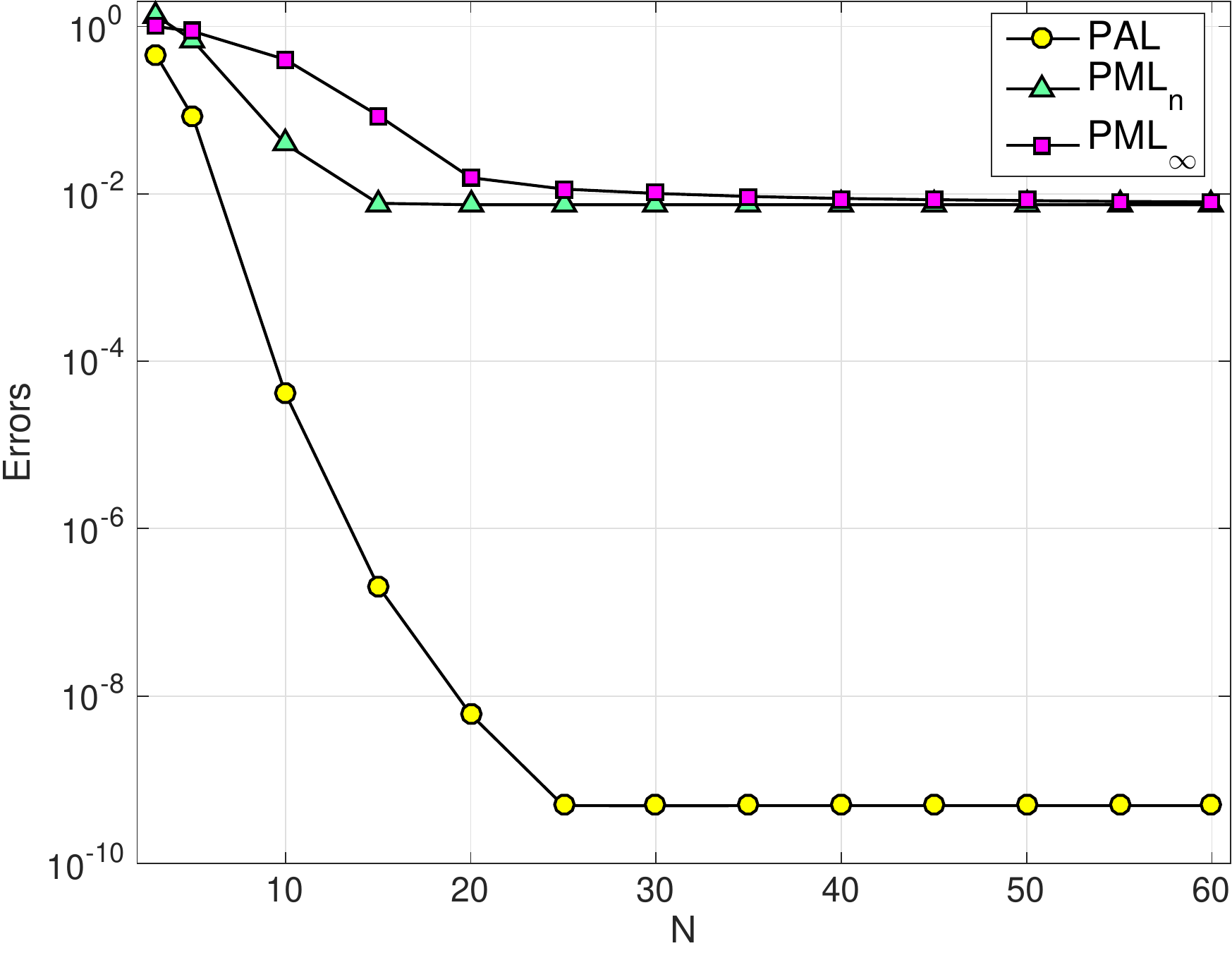}}\;\;
    \caption{\small A comparasion study for waveguide problem: error for  PAL, $\text{PML}_n$ and $\text{PML}_{\infty}$ against $N$ for $k=29.9$ with different layer thickness (a) $d=0.1$ and (b) $d=0.5$.} 
   \label{fig: comparestywaveguide}
\end{center}
\end{figure}
In Figure \ref{fig: comparestywaveguide}, we depict the $L^{\infty}$ errors for the numerical solution with these three truncation methods. It can be seen that the errors of the PAL method 
decreases exponentially to $10^{-10}$ as $N$ increases. However, due to the existence of evanescent modes, the errors saturated at around $10^{-1}$ and $10^{-2}$ for PML method with $d=0.1$ and $d=0.5$, respectively. As analysed previously, the saturation level can only be improved by an increased  layer thickness $d$,  which is prohibitive due to the increased computational cost.

\section{Star-shaped domain truncation and circular PAL}\label{sect3:domain}
\setcounter{equation}{0}
\setcounter{lmm}{0}
\setcounter{thm}{0}

One of the main purposes of this paper is to  design the PAL with a general star-shaped domain truncation for solving the two-dimensional  time-harmonic acoustic wave scattering problems.   More precisely, we consider the exterior domain:
\begin{subequations} \label{extproblem2d}
\begin{align}
& \mathscr{L}[U]:=\Delta U+k^2 U=f \quad {\rm in} \;\;\;  \Omega_e:=\mathbb{R}^2\setminus \bar D; \label{curlformA}\\
&  U=g\;\;\;  {\rm on}\;\;\partial D; 
\qquad \Big|\frac{\partial U}{\partial r} -\ri k U\Big|\le  \frac c r \quad      {\rm for}  \;\;\; r=: |\bs x| \gg 1, \label{silvmulA}
\end{align}
\end{subequations}
where   $D\subset {\mathbb R}^2$ is a bounded scatterer with  Lipschitz boundary $\Gamma_{\!D}=\partial D,$
and $g\in H^{1/2}(\Gamma_{\!D}).$ Here,  we assume that the source $f$ is compactly supported in a disk $B(\supset D).$
The far-field condition in \eqref{silvmulA} is known as the Sommerfeld radiation boundary condition. The PAL technique to be introduced  is  applicable to  the  Helmholtz problems with other types of boundary conditions such as Dirichlet or impendence boundary condition on   $\Gamma_{\!D},$ and also to solve acoustic
 wave propagations   in inhomogeneous media in  bounded domains.

We start with a general set-up for the star-shaped truncated domain, and provide new perspectives of the circular PAL reported in 
 \cite{wangyangpal17}, which shed lights on the study of  general PAL with star-shaped domain truncation  in the forthcoming section.  

\subsection{Star-shaped truncated domain}\label{mainPAL-Eqn}

As illustrated in Figure \ref{geometry2d}, we enclose $ \bar D$ by a star-shaped domain $\Omega_{\rm s}$ with respect to the origin. Assume that  the boundary  of  $\Omega_{\rm s}$ is piecewise smooth  with the parametric form in the  polar coordinates, viz.,
\begin{equation}\label{omegaS}
\Gamma_{\!R_1}:=\partial \Omega_{\rm s}=\big\{(r,\theta)\,:\,r=R_1(\theta),\;\;\theta\in [0,2\pi)\big\},
\end{equation}
or equivalently,  $\Gamma_{\!R_1}$ has the parametric form in Cartesian coordinates:
\begin{equation}\label{omegaSC}
\Gamma_{\!R_1}=\big\{\bs x=(x,y)\,:\, x=R_1(\theta)\cos \theta,\;  y=R_1(\theta)\cos \theta,\;  \theta\in [0,2\pi)\big\}.
\end{equation}
Then the  artificial layer is formed by surrounding  $\Omega_{\rm s}$ with
\begin{equation}\label{PAL-Omega}
\Omega_{\varrho}^{\rm PAL}:=\big\{(r,\theta)\,:\, R_1(\theta)<r<R_2(\theta):=\varrho R_1(\theta),\;\; \theta\in [0,2\pi)\big\},
\end{equation}
where the constant $\varrho>1$ can tune the ``thickness'' of the layer.   The layer $\Omega_{\varrho}^{\rm PAL}$  provides a  star-shaped domain truncation of the unbounded domain $\Omega_{e}$. We further denote the domain of interest   and the real computational domain, respectively, by
\begin{equation}\label{phusical}
 \Omega:=\Omega_{\rm s}\setminus \bar D,  \quad \Omega_{\rm c}:=\Omega_{\varrho}^{\rm PAL}\cup \Omega \cup \Gamma_{\!R_1},
\end{equation}
where we need to approximation the original solution in $\Omega,$ but have to  couple the original equation in $\Omega$ with the artificial equation in  $\Omega_{\varrho}^{\rm PAL}$ in real computation.

%
%

\medskip

\begin{figure}[htbp]
\begin{center}
  \includegraphics[scale=.45]{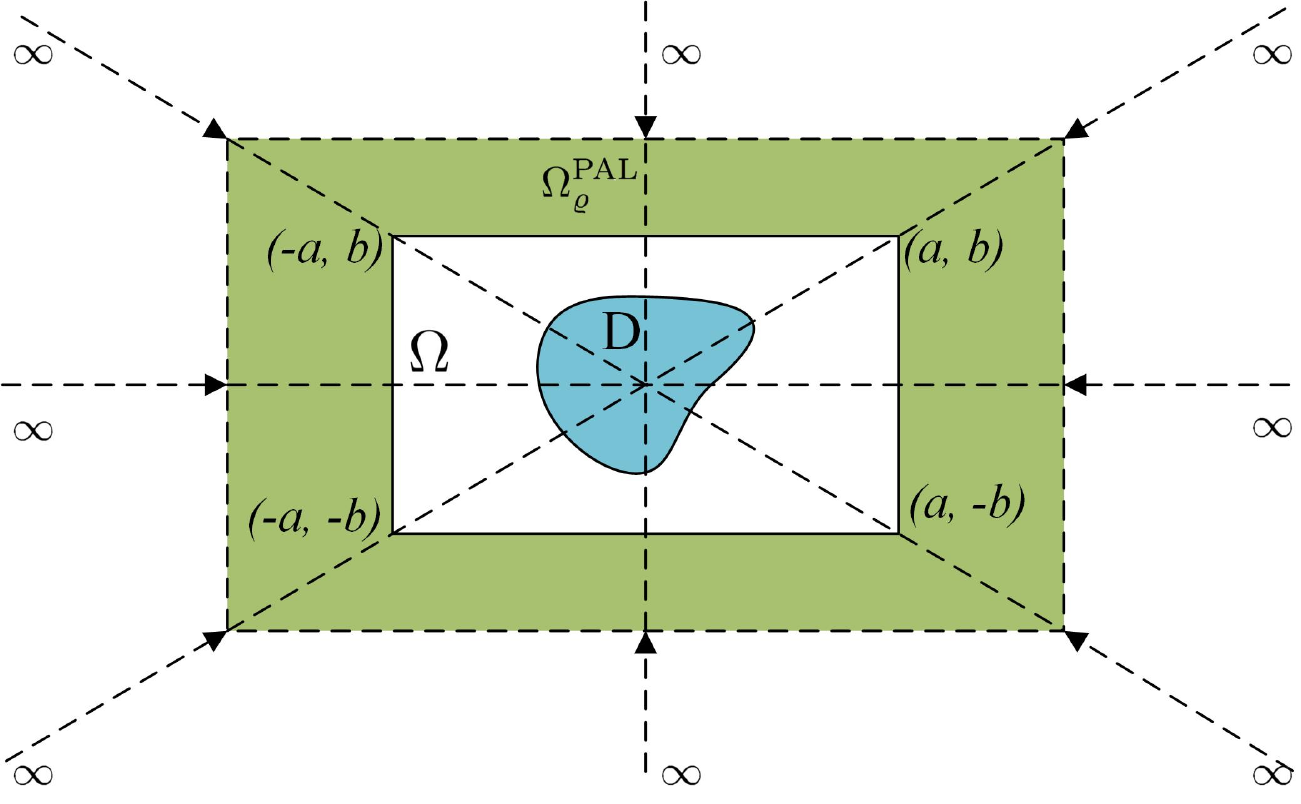} \qquad
    \includegraphics[scale=.45]{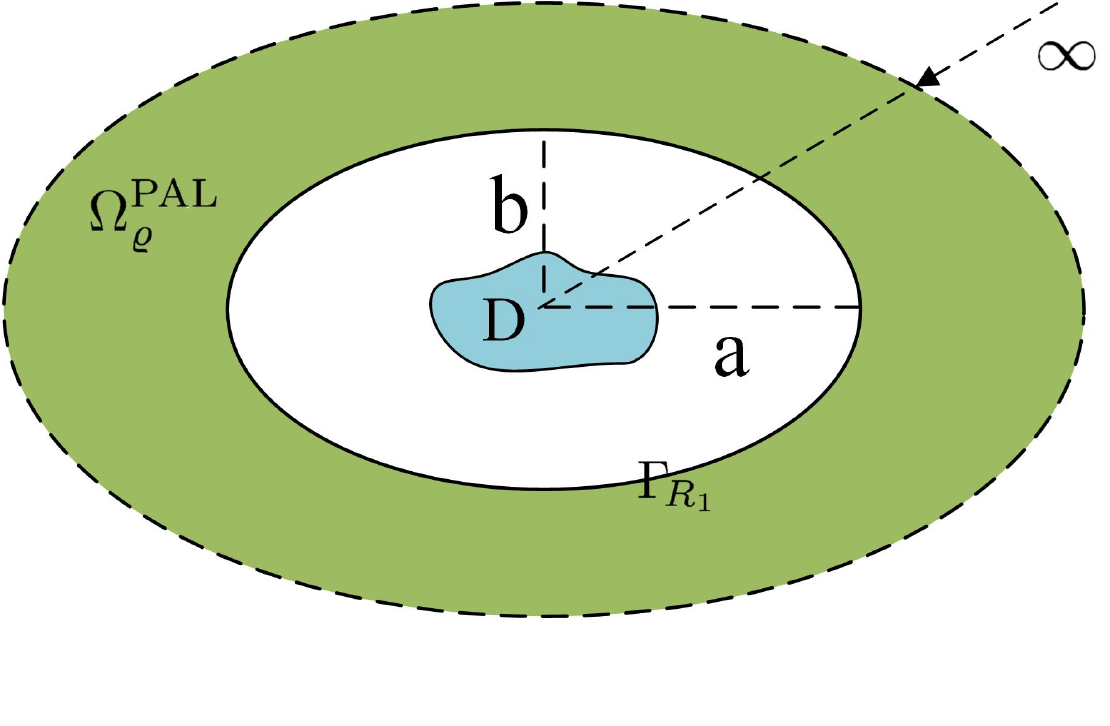}
    \caption{\small Schematic illustration of the rectangular PAL with four trapezoidal patches (left), the elliptical PAL (right), and the compression transformation from radial direction.}
   \label{geometry2d}
\end{center}
\end{figure}

Since the choice of the truncated domains is {\em a prior} arbitrary, it could be an advantageous to choose non-classical shapes to offer more flexibility to deal with non-standard geometry of the scatterer and inhomogeneity of the media.
 It is important to note that  the configuration of the artificial layer is solely determined by the parametric form of $R_1(\theta)$ and the tuning ``thickness" parameter $\varrho.$
We list below some typical  examples of such star-shaped domain truncation.
\begin{itemize}
\item[(i)]\,  In the circular case, the artificial layer  is an annulus, i.e.,
\begin{equation}\label{circularOmega}
  \Omega_{\varrho}^{\rm PAL}=\{R_1<r<R_2=\varrho R_1,\;0\le \theta<2\pi\},
\end{equation}
where $R_1$ is independent of $\theta.$  As a variant, the ``perturbed''  annular layer takes the form:
\begin{equation}\label{circularOmegaB}
  \Omega_{\varrho}^{\rm PAL}=\{R_1(\theta)=a+b\sin\theta <r<R_2(\theta)=\varrho R_1(\theta),\;0\le \theta<2\pi\},
\end{equation}
where $a,b>0$ are some given constants.


\item[(ii)]\, In the rectangular case,  we take for instance  the boundary $\Gamma_{\!R_1}$ as a square with four vertices:
$(a,b), (-a,b), (-a,-b), $  $ (a,-b) $ with  $ a,b>0$ (see Figure \ref{geometry2d} (left)).  Then we have
\begin{equation}\label{rect-Omega}
R_1(\theta)=\begin{cases}
a \sec \theta,\quad & \theta\in [0,\theta_0)\cup [2\pi-\theta_0, 2\pi),\\[2pt]
b \csc \theta,\quad & \theta\in [\theta_0,\pi-\theta_0),\\[2pt]
-a \sec \theta,\quad & \theta\in  [\pi-\theta_0, \pi+\theta_0),\\[2pt]
-b \csc \theta,\quad & \theta\in [\pi+\theta_0, 2\pi-\theta_0),
\end{cases}\qquad \theta_0=\arctan \frac b a.
\end{equation}
Similarly, we can consider a general rectangular domain truncation.

\item[(iii)]\, If we choose  $\Gamma_{\!R_1}$ to be an ellipse:
$\frac{x^2}{a^2}+\frac{y^2}{b^2}=1$ with $a>b>0$ (see Figure \ref{geometry2d} (right)),  then we have
\begin{equation}\label{ellitpitic}
R_1(\theta)=\frac{ab}{\sqrt{b^2+(a^2-b^2)\sin^2\theta}},\quad \theta\in [0,2\pi).
\end{equation}

\end{itemize}


Now, the key issue is  how to construct the governing equation in the artificial  layer. In practice,  one  wishes  (i) the solution of the resulted coupled problem in $\Omega_{\rm c}$ can  approximate the original solution $U|_{\Omega}$ as accurate as possible to avoid the  pollution  of the truncation, but (ii) the layer $\Omega_{\varrho}^{\rm PAL}$ should be thin enough to save computational cost. 
To show the essence of designing the PAL-equation for the above general star-shaped truncated domain, we first recap on the circular PAL proposed in  \cite{wangyangpal17}, but explore this technique from a very different viewpoint.

\subsection{Some new perspectives of the circular PAL} As some new insights, we next show  that the governing equation (in the annular layer: $r\in (R_1,R_2)$)  in  \cite{wangyangpal17} can  be obtained from the B\'enenger equation (in the unbounded domain $\rho>R_1$, see Collino and Monk \cite{collino1998perfectly}) by a real compression transformation.
Then  we can claim the exactness of the PAL technique --  the PAL-solution for $r<R_1$ coincides with the original solution $U|_{r<R_1}.$ This should be in contrast with the PML technique \cite{collino1998perfectly,chen2005adaptive}, where the governing equation in the layer
is obtained by naively  truncating the
 B\'enenger equation in unbounded domain at $r=R_2$,  and then impose then homogeneous Dirichlet boundary condition.

Like  \eqref{newguids},  Zharova et al. \cite{zharova2008inside} introduced the real compression transformation
\begin{equation}\label{rrhorelate00}
\rho:=\rho(r)=R_1+ (R_2-R_1) \dfrac{r-R_1}{R_2-r}\quad {\rm or} \quad  r=R_2-\frac{(R_2-R_1)^2}{\rho+R_2-2R_1},
\end{equation}
for $\rho\in (R_1,\infty)$ and  $r\in (R_1,R_2), $ to design the inside-out (or inverse) invisibility cloak and also  a  matched layer. 
 In principle, it compresses all the infinite space: $R_1<\rho<\infty$ into the finite  annulus:  $R_1<r<R_2,$ where ideally  the wave propagation is expected to be
equivalent to the wave propagation in the infinite space.
However, such a technique fails to work, as   the numerical approximation of the waves within the layer    suffers   from  {\em the curse of infinite oscillation} \cite{pmlnotes}.

Following \cite{wangyangpal17}, we propose to fill the cloaking layer with lossy media (i.e., complex material parameters),
and deal with  the singular media by using a   suitable substitution of unknowns.
More precisely, we introduce the compression complex coordinate transformation in polar coordinates:  
\begin{equation}\label{asBeqn00}
\begin{split}
& \tilde r=S(r)= \begin{cases}
r, & r<R_1,\\[2pt]
R_1+\sigma_1 T(r)+\ri\,  {\sigma_0}  \, T(r),\; &
R_1<r<R_2,
\end{cases}\quad \tilde \theta=\theta, \;\;\; \theta\in [0,2\pi),
\end{split}
\end{equation}
where $\sigma_0,\sigma_1>0$ are tuning parameters, and
\begin{equation}\label{TtransformA00}
T(r)=\dfrac{(R_2-R_1)(r-R_1)}{R_2-r}=(R_2-R_1)^2\int_{R_1}^r \frac{{\rm d}t} {(R_2-t)^2} ,\quad r\in (R_1,R_2).
\end{equation}
It is noteworthy  that $\Im\{\tilde r\}$ defines a compression mapping between $(0,\infty)$ and $(R_1,R_2);$ and
the parameters $\sigma_0,\sigma_1$  can be $k$-dependent, e.g., a constant multiple of $1/k$.

The PAL-equation can be obtained by applying the complex coordinate transformation \eqref{asBeqn00} to the Helmholtz
problem \eqref{extproblem2d} in  $(\tilde r,\tilde \theta)$-coordinates (see \cite{wangyangpal17}):
\begin{equation}\label{circularA}
\begin{cases}
\displaystyle\frac 1 r \frac{\partial}{\partial r}\Big(\frac{\beta r}{\alpha} \frac{\partial U_{\rm PAL}}{\partial r}\Big)+\frac{1}{r^2} \frac{\alpha} {\beta}
 \frac{\partial^2 U_{\rm PAL}}{\partial \theta^2}
+\alpha\beta\,k^2\, U_{\rm PAL}=f,\quad {\rm in}\;\; \Omega\cup {\Omega}_{\varrho}^{\rm PAL},\\[10pt]
 U_{\rm PAL}=g\;\;\;  {\rm on}\;\; \Gamma_{\!D},\quad  
 |U_{\rm PAL}| \;\; \text{is uniformly bounded  as $r \to R_2$},
\end{cases}
\end{equation}
together with the usual transmission conditions at $\rho=R_1.$ Here, we denoted
\begin{equation}\label{alphatalg}
\alpha=\frac{d\tilde r}{dr},\quad \beta=\frac{\tilde r} r.\quad 
\end{equation}

%
%
%

\subsubsection{B\'erenger's equations and PML techniques}
 In \cite{wangyangpal17}, we adopted the transformed Sommerfeld radiation boundary condition as $r\to R_2.$
In fact, it is only necessary to impose the uniform boundedness to guarantee the unique solvability and exactness with  $U|_{r\le R_1}=U_{\rm p}|_{r\le R_1}$ (see Theorem \ref{abstration} below).  To justify this, we next show that the PAL-equation \eqref{circularA} can be derived from the B\'erenger's equation (in unbounded domain) in \cite{collino1998perfectly}.
%
Indeed, using separation of variables,  the solution $U$ of the Helmholtz problem \eqref{extproblem2d} exterior to the circle: $r=R_1$ can be written as
 \begin{equation}\label{BUexa}
 U(\bs x)=\sum_{|m|=0}^\infty a_m\, \frac{H^{(1)}_m (kr)}{H^{(1)}_m (kR_1)} e^{\ri m\theta}, \quad  r\ge R_1,
 \end{equation}
 where $\bs x=r e^{\ri \theta},$ $H^{(1)}_m $ is the Hankel function of first kind and order $m,$ and  $\{a_m\}$ are the Fourier expansion coefficients of $U$ at $r=R_1.$ This  series  converges uniformly for $r\ge R_1$ (cf. \cite{Nedelec2001Book}). According to  \cite{collino1998perfectly}, the B\'erenger's  idea to design PML in the cylindrical coordinates can be interpreted as  stretching  the solution \eqref{BUexa} to the complex domain so that the waves become evanescent.
 Recall  the asymptotic behaviour of the Hankel function {\rm(}cf. \cite{Abr.I64}{\rm)}:
\begin{equation}\label{HapxA1}
|H_m^{(1)}(k\tilde r)|\sim \Big|\sqrt{\frac{2}{\pi k\tilde r}}\,e^{\ri(k (\Re \{\tilde r\}+\ri \Im\{\tilde r\})-\frac{1}{2}m \pi-\frac{1}{4}\pi)}\Big|=\sqrt{\frac{2}{\pi k|\tilde r|}}e^{-k \Im\{\tilde r\}},
\end{equation}
for $-\pi< {\rm arg}\{\tilde r\}<2\pi.$ This implies that  the extension should be made in the upper half-plane  such that
$\Im\{\tilde r\}\to \infty $ as $|\tilde r|\to \infty.$  In the PML technique, one uses the complex change of variables:
\begin{equation}\label{complexvariable}
\tilde r=\tilde r(\rho)=\begin{cases}
\rho,\quad & {\rm for}\;\; \rho<R_1, \\[4pt]
\rho+\ri \displaystyle\int_{R_1}^\rho \hat \sigma(t){\rm d} t,\quad & {\rm for}\;\; R_1\le \rho<\infty,
\end{cases}
\end{equation}
where in general, the absorbing function $\hat \sigma$ satisfies
\begin{equation}\label{hatsigms}
\hat\sigma\in C(\mathbb R), \quad \hat \sigma\ge 0 \quad {\rm and}\quad   \lim_{r\to\infty}\int_{R_1}^r \hat \sigma(t){\rm d} t=\infty.
\end{equation}
Denote $\hat \alpha=\tilde r'(\rho)$ and $\hat \beta=\tilde r(\rho)/\rho.$
Applying  the coordinate transformation \eqref{complexvariable} to the Helmholtz equation exterior to the circle of radius  $R_1$,  we can obtain the  B\'erenger's problem of computing the B\'erenger's solution $U_{\rm B}$ in the form (cf.  \cite{collino1998perfectly}):
\begin{equation}\label{circularA0}
\begin{cases}
 \displaystyle\frac 1 \rho  \displaystyle \frac{\partial}{\partial \rho}\Big(\frac{\hat \beta \rho}{\hat \alpha} \frac{\partial U_{\rm B}}{\partial \rho }\Big)+\frac{1}{\rho^2} \frac{\hat \alpha} {\hat \beta}
 \frac{\partial^2 U_{\rm B}}{\partial \theta^2}
+\hat\alpha\hat \beta\,k^2\, U_{\rm B}=f\quad {\rm in}\;\;  \Omega_e=\mathbb{R}^2\setminus \bar D;\\[10pt]
 U_{\rm B}=g\;\;\;  {\rm on}\;\; \Gamma_{\!D},  \\[2pt]
 |U_{\rm B}| \;\; \text{is uniformly bounded  as $\rho \to \infty$}\,,
\end{cases}
\end{equation}
together with the usual transmission conditions at $\rho=R_1.$ Note that $\hat \alpha=\hat \beta=1$ for $|\bs x|\le R_1.$
According to \cite[Theorem 1]{collino1998perfectly},
the problem \eqref{circularA0}  admits a unique solution, and  for any $|\bs x|\ge R_1,$ the B\'erenger's solution takes the form
 \begin{equation}\label{BUexa2}
 U_{\rm B}(\bs x)=\sum_{|m|=0}^\infty a_m\, \frac{H^{(1)}_m (k\tilde r (\rho))}{H^{(1)}_m (kR_1)} e^{\ri m\theta},
 \end{equation}
 where $\{a_m\}$ are the same as in \eqref{BUexa}.   In other words, {\em the B\'enenger's solution coincides with the solution of the original problem \eqref{extproblem2d} for $r<R_1.$}

 As shown in \cite{collino1998perfectly,chen2005adaptive},   the PML technique directly truncates  \eqref{circularA0} at $\rho=R_2,$ and the homogeneous boundary condition: $U_{\rm PML}=0$ at $\rho=R_2$ is then imposed.  More precisely, we have the following PML-equation:
 \begin{equation}\label{circularA01}
\begin{cases}
 \displaystyle\frac 1 r  \displaystyle \frac{\partial}{\partial r}\Big(\frac{\hat \beta r}{\hat \alpha} \frac{\partial  U_{\rm PML}}{\partial r}\Big)+\frac{1}{r^2} \frac{\hat \alpha} {\hat \beta}
 \frac{\partial^2  U_{\rm PML}}{\partial \theta^2}
+\hat\alpha\hat \beta\,k^2\,  U_{\rm PML}=f\quad {\rm in}\;\;  \Omega_e \cap \{|\bs x|< R_2\};\\[10pt]
  U_{\rm PML}\big|_{\Gamma_{\!D}}=g\,;  \quad  U_{\rm PML}\big|_{r=R_2}=0,
\end{cases}
\end{equation}
where we set $r=\rho$ as the independent variable for clarity.
 Like the waveguide setting in \eqref{PMLclassic} -\eqref{UPMLun},  the following two types of absorbing functions have been used in practice.
 \begin{itemize}
 \item[(i)]  PML$_n$ with bounded (or regular) ABFs  (see, e.g.,  \cite{collino1998perfectly,chen2005adaptive}):
\begin{equation}\label{RegularKernel}
\hat \sigma(t)=\Big(\frac{t-R_1}{R_2-R_1}\Big)^n,\quad  {\rm so}\quad  \tilde r=r+ \ri \, \sigma_0\, \frac{  R_2-R_1  }{n+1}
\Big(\frac{r-R_1}{R_2-R_1}\Big)^{n+1}, \quad  r\in (R_1,R_2),
\end{equation}
where $n$ is a positive integer.  
We refer to  \cite{chen2005adaptive} for the detailed error analysis, and also the very recent work  \cite{Li.W19} for the insightful wavenumber explicit error estimates.

\medskip 
\item[(iv)] PML$_\infty$ with   unbounded   (or singular) ABFs  (see \cite{bermudez2007exact,bermudez2007optimal}):
\begin{equation}\label{SingleKernel}
\hat \sigma(t)=\frac{ 1}{R_2-t},\quad  {\rm so}\quad  \tilde r=r+\ri\, \sigma_0   \,\ln \Big(\frac{R_2-R_1}{R_2-r}\Big), \quad  r\in (R_1,R_2).
\end{equation}
Compared with the  PML$_n$,  the PML$_\infty$  renders the solution  decay at an infinite rate near the outer boundary $r=R_2.$
It is therefore not surprising it is  parameter-free  \cite{Cimpeanu2015A}.  However, from the above analysis, we infer that the PML-equation \eqref{circularA01} with unbounded ABFs \eqref{SingleKernel} does  not really exactly solves the original problem in
$\Omega.$ In addition, the coefficients: $\hat \alpha=\tilde r'(r)$ and $\hat \beta=\tilde r(r)/r$ are singular at $r=R_2,$ which
brings about numerical difficulties  in  realisation.
 \end{itemize}

\subsubsection{Equivalence of  B\'erenger's problem and PAL equation}
 We next show that  in contrast to the PML technique, our proposed PAL-equation exactly solves the transformed problem \eqref{circularA0} by further transforming it to a bounded domain by using a real compression  mapping.
 \begin{thm}\label{abstration} Let $U_{\rm B}(\rho,\theta)$ be the  solution of the B\'enenger's problem  \eqref{circularA0} with
   $\hat \sigma=\sigma_0/\sigma_1$  in  \eqref{complexvariable}, that is, the complex coordinate transformation:
   \begin{equation}\label{simpletransform}
   \tilde r=\rho+\ri \frac {\sigma_0}{\sigma_1} (\rho-R_1),\quad {\rm for}\;\; \rho>R_1.
   \end{equation}
   Then  applying the  real compression rational mapping
   \begin{equation}\label{newcasesA}
 \rho=R_1+\sigma_1 T(r)= R_1+ \sigma_1 \dfrac{(R_2-R_1)(r-R_1)}{R_2-r},\quad
 \end{equation}
 to  \eqref{complexvariable}, we
   can derive  the circular PAL-equation \eqref{circularA}.
Moreover, the PAL-equation \eqref{circularA}   admits a unique solution, and
 \begin{equation}\label{caspA}
      U_{\rm PAL}(r,\theta)=U_{\rm B}(R_1 + \sigma_1 T(r),\theta),\quad r<R_2; \;\;\; {\rm and}\;\;\;
      U_{\rm PAL}|_{r<R_1}=U|_{r<R_1},
      \end{equation}
      where $U$ is the solution of  the original problem \eqref{extproblem2d}.
 \end{thm}
\begin{proof} It suffices to consider the transformation for $\rho>R_1.$ We find from \eqref{complexvariable} with $\hat \sigma=\sigma_0/\sigma_1$ that $\hat\alpha=1+\ri\sigma_0/\sigma_1$ and $\hat \beta=(\rho+\ri\sigma_0 (\rho-R_1)/\sigma_1)/\rho.$
Using the rational mapping: $\rho=R_1+\sigma_1 T(r),$ we obtain from \eqref{alphatalg} and the composite transformation
\eqref{simpletransform}-\eqref{newcasesA} (i.e., \eqref{asBeqn00})
that
\begin{equation}\label{alphabeta1}
\alpha=(\sigma_1+\ri\sigma_0)T'(r)=\hat\alpha\, \frac{d\rho}{dr},\quad \beta=\frac{R_1+(1+\ri \sigma_0)T(r)} r=\frac{\rho+\ri\sigma_0 (\rho-R_1)}{\rho}=\hat \beta \frac{\rho} r.
\end{equation}
By a simple substitution: $\frac{d}{d\rho}=\frac{dr}{d\rho}\frac{d}{dr},$ we can obtain  \eqref{alphabeta1} from  \eqref{circularA0} straightforwardly.

 The unique solvability of \eqref{circularA}  follows directly from that of \eqref{circularA0} (see \cite[Theorem 1]{collino1998perfectly}).
 Moreover, the PAL-solution is a compression of the B\'enenger solution: $U_{\rm p}(r,\theta)=U_{\rm B}(R_1+\sigma_1 T(r),\theta)$ for $r<R_2,$ so it is identical to the original solution for $r<R_1.$
 Thus, by \eqref{BUexa2}, we have
\begin{equation}\label{BUexa3}
 U_{\rm p}(\bs x)=\sum_{|m|=0}^\infty \hat a_m\, \frac{H^{(1)}_m (kS(r))}{H^{(1)}_m (kR_1)} e^{\ri m\theta},\quad R_1<r<R_2,
 \end{equation}
 where $\{\hat a_m\}$ are the Fourier coefficients of $U_{\rm p}(R_1,\theta).$
 \end{proof}

 With the above understanding, we can show that the PAL-solution  in the artificial layer
  decays exponentially. In fact, the bound is more precise than the estimate in \cite[Theorem 1]{wangyangpal17}.
 \begin{thm}\label{decayfuns} The solution of  the PAL-equation \eqref{circularA}
satisfies  that for all $r\in [R_1, R_2),$
\begin{equation}\label{PALdecay}
\int_{0}^{2\pi} |U_{\rm PAL}(r,\theta)|^2 d\theta\le  {\rm exp}\Big(\!\!-\sigma_0 k\frac{\tau}{1-\tau}\big(1-h(\tau) (1-\tau)^2\big)^{1/2} \Big) \int_{0}^{2\pi} |U_{\rm PAL}(R_1,\theta)|^2 d\theta,
\end{equation}
where
$$ h(\tau)=\frac{R_1^2}{(R_1(1-\tau)+\sigma_1d\,\tau)^2+\sigma_0^2d^2\,\tau^2}, \quad d=R_2-R_1,  \quad \tau=\frac{r-R_1}{R_2-R_1}\in [0,1). $$
 \end{thm}
 \begin{proof}
 We now show that the PAL-solution decays exponentially  in the PAL layer.  For this purpose, we recall the uniform estimate of Hankel functions first derived in  \cite[Lemma 2.2]{chen2005adaptive}:  {\em  For any complex $z$ with $\Re\{z\},  \Im\{z\}\ge 0,$ and for any real $\Theta$ such that $0<\Theta\le |z|,$ we have
\begin{equation}\label{Chenforma}
|H_\nu^{(1)}(z)|\le e^{-\Im\{z\}\big(1-\frac{\Theta^2}{|z|^2}\big)^{1/2}} |H_\nu^{(1)}(\Theta)|,
\end{equation}
which is valid for for any real order $\nu.$} Note that $S(R_1)=R_1,$ and  for $r\in [R_1, R_2),$
$$|S(r)|^2= (R_1+\sigma_1 T(r))^2+\sigma_0 T^2(r) \ge R_1^2\,, \quad   \Im\{S(r)\}=\sigma_0 T(r).$$
Thus, we obtain  from \eqref{Chenforma} with $z=kS(r)$ and $\Theta=kR_1$ that for $r \in [R_1,R_2),$
\begin{equation}\label{rationbnd}
\begin{split}
\max_{|m|\ge 0} \bigg|\frac{H_m^{(1)}(kS(r))}{H_m^{(1)}(kR_1)}\bigg|& \le  {\rm exp}\Big\{{-k \Im\{S(r)\} \Big(1-\frac{R_1^2}{|S(r)|^2}\Big)^{1/2}}\Big\} \\
&={\rm exp}\Big(\!\!-\sigma_0 d k\frac{\tau}{1-\tau}\big(1-h(\tau) (1-\tau)^2\big)^{1/2}\Big),
\end{split}
\end{equation}
where the last step follows from direct calculation.
\end{proof}

It is seen that the PAL-equation leads to  the exact B\'ernenger solution, which decays  exponentially  to zero  at a rate: $O(e^{-\sigma_0 d k /(1-\tau)})$ as $r\to R_2$ (i.e.,  as $\tau\to 1).$ 
 However, there are two numerical issues  to be addressed.
\begin{itemize}
\item[(i)]  The coefficients of the PAL-equation \eqref{circularA01} are singular,  which are induced by  the compression rational mapping:  $\rho=R_1+\sigma_1 T(r)$. In fact, we have
\begin{equation}\label{alphabeta-singular}
\alpha(r)\sim \frac {C_1} {(R_2-r)^2}, \quad \beta(r) \sim \frac {C_2} {R_2-r}.
\end{equation}
However, the underlying PAL-solution is not singular at $r=R_2,$ as it  decays exponentially in the artificial layer. 
\item[(ii)] Observe from \eqref{HapxA1} that the real part of the transformation: $\rho=\Re\{S(r)\}=R_1+\sigma_1 T(r)$  may increase the oscillation near the inner boundary $r=R_1$.
\end{itemize}
To resolve these two issues,  we  follow \cite{wangyangpal17} by using  a substitution of the unknown: $U_{\rm PAL}=w(r) V_{\rm p}$ with  a suitable factor $w(r),$ so that  $V_p$ to be approximated  is well-behaved. The choice of $w$ is actually spawned by  the asymptotic behaviour  due to   \eqref{HapxA1}:
\begin{equation}\label{ratioS}
\frac{H^{(1)}_m (kS(r))}{H^{(1)}_m (kR_1)}\sim \sqrt{\frac{R_1}{S(r)}}\,e^{\ri k (\Re \{S(r)\}-R_1)} e^{-k\Im\{S(r)\}} =
\sqrt{\frac{R_1}{S(r)}}\,e^{\ri k \sigma_1 T(r)} e^{-k\Im\{S(r)\}}.
\end{equation}
It implies the oscillatory part  of $U_{\rm PAL}$ can be extracted explicitly:  $U_{\rm PAL}(r,\theta)= e^{\ri k \sigma_1 T(r)} \widetilde U_{\rm PAL}(r,\theta),$ where  $\widetilde U_{\rm p}(r,\theta)$ is expected to have no essential oscillation.
Then  the second issue can be resolved effectively. In regards to  the first issue, it is seen from Theorem \ref{abstration} that the singular coefficients are induced by the real transformation: $\rho=R_1+\sigma_1 T(r).$ In fact,  similar singular mapping techniques were used to map  elliptic problems with rapid decaying solutions in unbounded domains to problems with singular coefficients in bounded domains (see, e.g.,
\cite{Boyd01,Guo.01,She.W09,Shen.TWBook01} and the references therein), so
 one can also consider a suitable variational  formulation weighted with  $\omega(r)=R_2-r$.  Unfortunately, the involved variational formulation is  non-symmetric and less efficient  in  computation.
As shown in \cite{wangyangpal17}, we can absorb the singularity and diminish the oscillation by the substitution:
\begin{equation}\label{twotechni}
U_{\rm p}(r,\theta)=w(r) V_{\rm p}(r,\theta), \quad r<R_2,
\end{equation}
where
\begin{equation}\label{wdefn}
w(r)=\begin{cases}
1,\quad &{\rm for}\;\; r<R_1,\\[2pt]
\big(R_1/(R_1+\sigma_1 T(r))\big)^{3/2} e^{\ri k \sigma_1 T(r)},   &{\rm for}\;\; R_1\le r<R_2,
\end{cases}
\end{equation}
which leads to a well-behaved and non-oscillatory field $V_{\rm p}$ in the absorbing layer.

It is important to remark that  in numerical discretisation, we  can build the substitution in the basis functions. More precisely, we  approximate $U_{\rm PAL}$ by the nonstandard basis $\{w \phi_k\}$ (where $\{\phi_k\}$ are usual spectral or finite element basis functions) to avoid transforming the PAL-equation into a much complicated problem in $V_{\rm p}.$   In the above, we just show the idea, but refer to   \cite{wangyangpal17}  and the more general case in Section \ref{lm:Sect4} for the detailed implementation.

\section{The PAL technique for star-shaped domain truncation}\label{lm:Sect4}
\setcounter{equation}{0}
\setcounter{lmm}{0}
\setcounter{thm}{0}

With the understanding of the circular case, we are now in a position  to  construct the PAL technique for the general star-shaped domain truncation with the setting 
described in Subsection \ref{mainPAL-Eqn}. We start with constructing the PAL-equation based on the complex compression coordinate transformation.  Then we show the outstanding performance of this technique. 

\subsection{Design of the PAL-equation} The first step is to extend the complex compression coordinate transformation
\eqref{asBeqn00}  to the general case: 
\begin{equation}\label{asBeqn}
\begin{split}
& \tilde r=S(r, \theta)= \begin{cases}
r, & {\rm in}\;\; \Omega,\\[2pt]
R_1(\theta)+\sigma_1\, T(r,\theta)+\ri\, \sigma_0 \, T(r,\theta),\; &
{\rm in}\;\;  \Omega_{\varrho}^{\rm PAL},
\end{cases}\quad \tilde \theta=\theta, \;\;\; \theta\in [0,2\pi),
\end{split}
\end{equation}
where
\begin{equation}\label{TtransformA}
T(r,\theta)=\dfrac{(R_2-R_1)(r-R_1)}{R_2-r}= (R_2-R_1)^2 \int_{R_1\!}^r \frac{ {\rm d}t} {(R_2-t)^2} ,\quad r\in (R_1,R_2).
\end{equation}
Different from the circular case, $R_1$ and $R_2$ are now $\theta$-dependent.
Indeed, we notice   from \eqref{omegaS} and  \eqref{PAL-Omega} that  $R_1=R_1(\theta)$  defines the inner boundary of the artificial layer $\Omega_{\varrho}^{\rm PAL},$ whose outer boundary is given by  $R_2=\varrho R_1(\theta)$.
For any fixed $\theta\in [0,2\pi),$ $\Re\{\tilde r\}=R_1(\theta)+\sigma_1\, T(r,\theta)$
 compresses the infinite ``ray":  $R_1\to \infty$ into a ``line segment":  $R_1\to R_2$ in the radial direction.
Accordingly,   for all  $\theta\in [0,2\pi),$  it compresses
the open space exterior to the star-shaped domain $\{r<R_2\}$ to the artificial layer $\Omega_{\varrho}^{\rm PAL}$
 (see  Figure \ref{geometry2d} for an illustration).

\begin{rem}\label{intpS}{\em Based on the notion of transformation optics \cite{Pendry.2006},
the use of a real singular coordinate transformation to expand the origin into a polygonal or star-shaped domain to design invisibility  cloaks, is discussed in e.g.,  \cite{yang2016seamless,Yang.LH18}. In contrast, the real part of
the transformation \eqref{asBeqn}: $R_1(\theta)+\sigma_1\, T(r,\theta)$ compresses the infinity to the finite
boundary $r=R_2(\theta)$, so the cloaking  is an inside-out or inverse cloaking as with \cite{zharova2008inside}.  However, to the best of our knowledge, this type of cloaking has not been studied in literature.}
\end{rem}



In order to derive the PAL-equation in Cartesian coordinates, it is necessary to commute between  different coordinates in the course as shown in the diagram:
\medskip
\begin{equation}\label{polartranA}
{\small \ovalbox{Cartesian:~$(\tilde x,\tilde y)$} \xleftarrow{\makebox[0.3cm]{}}\hspace*{-2pt}\xrightarrow{\makebox[0.3cm]{}} \ovalbox{Polar: $(\tilde r,\tilde \theta)$} \xrightarrow{\makebox[1.8cm]{Transform}} \ovalbox{Polar: $(r, \theta)$} \xleftarrow{\makebox[0.3cm]{}}\hspace*{-2pt}\xrightarrow{\makebox[0.3cm]{}}
\ovalbox{PAL: $(x,y)$}}
\end{equation}
\vskip 5pt
\noindent In what follows, the differential  operators ``$\nabla, \nabla\cdot$"  are in $(x,y)$-coordinates, but  the  coefficient matrix $\bs C$ and the reflective index $n$ are expressed  in $(r,\theta)$-coordinates.
For simplicity, we  denote the partial derivatives by $S_r=\partial_r S $ and $S_\theta=\partial_\theta S,$ etc..

The most important step  is to obtain   the transformed  Helmholtz operator as follows, whose derivation is given in Appendix \ref{AppendixB}.
\begin{lmm}\label{thm:PAL-eqn} Using the  transformation  \eqref{asBeqn},  the  Helmholtz  operator
\begin{equation}\label{helmhotlz}
 \mathcal{\tilde H}[\tilde U]:=\Delta \tilde U+k^2 \tilde U,
\end{equation}
in  $(\tilde x,\tilde y)$-coordinates, can be transformed into  
\begin{align}
{\mathcal  H}[U_{\rm p}]:=
\frac {1} {n}\big\{\nabla \cdot \big({\bs C}\, \nabla  U_{\rm p} \big)+k^2\, n\,   U_{\rm p}\big\},
 \label{PALeqnA}
\end{align}
where $U_{\rm p}(x,y)=\tilde U(\tilde x,\tilde y),$  ${\bs C}=(C_{ij})$ is a two-by-two symmetric matrix and $n$ is the reflective index given by
\begin{equation}\label{Cn2dpolarform}
 {\bs C}={\bs T}\,{\bs B}\,{\bs T}^t,\quad n=\frac{SS_r}{r},
\end{equation}
and
\begin{equation}\label{rotationmat}
  {\bs B}=\begin{pmatrix}
\dfrac{S}{r S_r}\Big(1+\dfrac{S_{\theta}^2}{S^2} \Big)\;\; &  -\dfrac{S_{\theta}}{S}\\[8pt]
-\dfrac{S_{\theta}}{S} & \dfrac{r S_r}{S}
\end{pmatrix},\qquad  {\bs T}=\begin{pmatrix}
\cos \theta & -\sin \theta\\[4pt]
\sin \theta  &  \cos \theta
\end{pmatrix}.
\end{equation}
\end{lmm}

\medskip

With the aid of Lemma \ref{thm:PAL-eqn},  we directly  apply  \eqref{asBeqn}-\eqref{TtransformA} to the exterior Helmholtz problem and obtain the PAL-equation for the general star-shaped truncation of the unbounded domain. 
\begin{theorem}\label{mainthmA} The PAL-equation associated with  star-shaped domain truncation
of the original Helmholtz  problem \eqref{extproblem2d} takes the form
\begin{equation}\label{PALeqn}
\begin{cases}
\nabla \cdot \big({\bs C}\, \nabla  U_{\rm PAL} \big)+k^2\,n\,  U_{\rm PAL}=f \;\;\;\;  {\rm in}\;\;  \Omega\cup\Omega_\varrho^{\rm PAL}, \\[2pt] 
U_{\rm PAL}=g\;\;\;  {\rm on}\;\; \Gamma_{\!D}, \\[2pt]  
 |U_{\rm PAL}| \;\; \text{is uniformly bounded  as $r\to R_2$}\,, 
\end{cases}
\end{equation}
where $\bs C=\bs I, n=1$ in $\Omega,$ and $\bs C, n$ in $\Omega_\varrho^{\rm PAL}$  are the same as in Lemma {\rm \ref{thm:PAL-eqn}} with the transformation $S$ given by \eqref{asBeqn}-\eqref{TtransformA}. Here,  we impose the usual transmission conditions across any interface $\Gamma_{\!R_1}: r=R_1.$
\end{theorem}

As an extension of  the circular case,  the asymptotic boundary condition at $r=R_2$ is obtained from the transformed Sommerfeld radiation condition and rapid decaying  of $U_{\rm p}$ near the outer boundary of $\Omega_{\varrho}^{\rm PAL}.$  
Indeed, for any fixed $\theta\in [0,2\pi),$ we can formally express the solution  of \eqref{extproblem2d}  as
\begin{equation}\label{BUexass}
 U(r,\theta)=\sum_{|m|=0}^\infty \tilde a_m\, \frac{H^{(1)}_m (kr)}{H^{(1)}_m (kR_1)} e^{\ri m\theta}, \quad  r> R_1,
 \end{equation}
 where  $\{\tilde a_m\}$ are determined by  $U$ at the circle $r=R_1.$ Then, we extend  \eqref{simpletransform} directly to the $\theta$-dependent situation:
  $\tilde r=\rho$ for $\rho<R_1,$ and
   \begin{equation}\label{complexvariableB}
\tilde r=\tilde r(\rho,\theta)=
\rho+\ri \frac{\sigma_1}{\sigma_0} (\rho-R_1)\quad\;\; {\rm for}\;\;\;  \rho >R_1,
\end{equation}
Then we can apply Lemma \ref{thm:PAL-eqn}  to derive the Benenger-type  problem in the unbounded domain  like \eqref{circularA0}.
Its solution for $\rho>R_1$  can be obtained from  complex stretching  of \eqref{BUexass}:
\begin{equation}\label{BUexassA}
 U(\rho,\theta)=\sum_{|m|=0}^\infty \tilde a_m\, \frac{H^{(1)}_m (k\tilde r)}{H^{(1)}_m (kR_1)} e^{\ri m\theta}, \quad \rho>R_1.
 \end{equation}
   Note that the transformation \eqref{asBeqn} is a composition of  \eqref{complexvariableB} and the real compression transformation: $\rho=R_1+\sigma_1 T(r,\theta)$.  Thus,  the PAL-equation in Theorem \ref{mainthmA} turns out to be a real compression of the Benenger-type  problem.

\subsection{Substitution and implementation}\label{mainPAL-Sub}
%
A key to  success of the PAL technique is to make a substitution of the unknown that can deal with the singular coefficients at $r=R_2$, and diminish the oscillation near $r=R_1$.  
To fix the idea,  we  assume that $g=0$, and
define the space
\begin{equation}\label{H10Gamma}
H^1_{0,\Gamma}(\Omega_{\rm c})=\big\{\phi\in H^1(\Omega_{\rm c}): \phi\,|_{\Gamma}=0\big\},\quad \Omega_{\rm c}:=\Omega\cup \Omega_{\varrho}^{\rm PAL},  \;\;\;  \Gamma:=\partial D.
\end{equation}
A weak form of \eqref{PALeqn} is to find
$u=w v$ with $v\in H^1_{0,\Gamma}(\Omega_{\rm c})$
such that
\begin{equation}\label{sesq2d0}
{\mathscr B}(u,\psi)=(\bs C\, \nabla u, \nabla \psi)_{\Omega_{\rm c}} -k^2(n\, u,\psi)_{\Omega_{\rm c}}=(f,\psi)_{\Omega_{\rm c}},
\end{equation}
for all $\psi=w \,\phi$ with $\phi\in H^1_{0,\Gamma}(\Omega_{\rm c}),$ where
\begin{equation}\label{wcase20}
w\,|_{\Omega}=1,\quad w\,|_{\Omega_{\varrho}^{\rm PAL}}= t^{3/2}\, e^{\ri k \sigma_1 T(r,\theta)},\quad t:=\frac{R_1} {\Re\{S\}}=\frac{R_1}{R_1+\sigma_1 T}.
\end{equation}
Note that in view of \eqref{ratioS} and \eqref{BUexassA}, we  extract the above  oscillatory component, together with the singular factor $t^{3/2},$
to form $w.$ 
As we shall see later on,  the power $3/2$  is the smallest to
absorb all a singular coefficients. On the other hand,
in numerical approximation, we approximate $v$ by standard spectral-element and finite element methods,
so it is necessary to compute the associated ``stiffness'' matrix: ${\mathscr B}(wv,w\phi)$  with $v,\phi$ being in the solution and test function spaces.

We next provide the detailed representation of the transformed
sesquilinear form for the convenience of both the computation and also the analysis of  the problem in $v.$   For clarity, we reformulate \eqref{sesq2d0} as:  find $v\in H^1_{0,\Gamma}(\Omega_{\rm c})$ (and set $u=w\,v$) such that
\begin{equation}\label{sesq2d3}
\breve {\mathscr B}(v,\phi):={\mathscr B}(wv,w\phi)=(\bs C\, \nabla (w v), \nabla (w\,\phi))_{\Omega_{\rm c}} -k^2(n\, |w|^2v, \phi)_{\Omega_{\rm c}}=(f, w\,\phi)_{\Omega_{\rm c}},
\end{equation}
for all $\phi\in H^1_{0,\Gamma}(\Omega_{\rm c}).$



The following formulation holds for general differentiable $w,$ which will be specified later for clarity of presentation.  
\begin{lmm}\label{thm:BVaPhi}  The sesquilinear form $\breve {\mathscr B}(v,\phi)$ in \eqref{sesq2d3} can be rewritten as
\begin{equation}\label{sesquilinearB}
\begin{split}
\breve {\mathscr B}(v,\phi)&=
\big(|w|^2\bs B \breve \nabla v, \breve \nabla \phi \big)_{\Omega_{\rm c}} + \big(w \bs B \breve \nabla w^{*}\cdot \breve \nabla  v,    \phi \big)_{\Omega_{\rm c}}  +\big(w^{*}\bs B \breve \nabla w\,  v,   \breve \nabla \phi \big)_{\Omega_{\rm c}} +\big(\breve nv, \phi\big)_{\Omega_{\rm c}}\\
&=\big(\breve{\bs B} \breve \nabla v, \breve \nabla \phi \big)_{\Omega_{\rm c}} + \big(\bs p\cdot \breve \nabla  v,    \phi \big)_{\Omega_{\rm c}}  +\big(  v,  \bs q^*\cdot  \breve \nabla \phi \big)_{\Omega_{\rm c}} +\big(\breve n\, v, \phi\big)_{\Omega_{\rm c}}, 
\end{split}
\end{equation}
where  $\breve \nabla = (\partial_r, \partial_{\theta}/r)^t, $ and
\begin{equation}\label{tilden}
\breve{\bs B}:=|w|^2 \bs B,\quad \bs p:=w \bs B \breve \nabla w^{*},\quad
\bs q:=w^{*}\bs B \breve \nabla w,\quad   \breve n:=(\breve \nabla w^*)^t  \bs B\, \breve \nabla w -k^2 \, |w|^2\, n,
\end{equation}
with $\bs B$ and $n$ defined in \eqref{rotationmat}. In \eqref{sesquilinearB}, $\bs B=\breve{\bs B}=\bs I, \bs p=\bs q=\bs 0$ and
$n=-k^2$  in  $\Omega.$
\end{lmm}
\begin{proof}  One verifies readily that
\begin{equation}\label{opsA}
\nabla =(\partial_x, \partial_y)^t=\bs T ( \partial_r,  \partial_{\theta}/r)^t=\bs T \, \breve \nabla\,,
\end{equation}
For clarity, we denote $\xi=w v$ and $\eta=(w \,\phi)^*.$ From  \eqref{opsA}, we immediately derive
\begin{equation}\label{nsigproperty}
 (\nabla \eta)^t \bs C \nabla \xi=(\bs T \breve \nabla  \eta)^t\, \bs C\, ( \bs T \breve \nabla \xi)=(\breve \nabla \eta)^t\,  (\bs T^t\bs C \bs T)\, \breve \nabla \xi=\breve \nabla^t \eta\, \bs B\, \breve \nabla \xi,
 \end{equation}
where we used the fact: $\bs B=\bs T^t\bs C \bs T,$ due to  \eqref{Cn2dpolarform} and the property:  $\bs T^{-1}=\bs T^t.$  Since
\begin{equation}\label{nsigprop2}
\breve \nabla \xi =(\breve\nabla w)v+w\breve \nabla v, \quad  \breve \nabla^t \eta =(\breve\nabla^t  w^*)\phi^*+w^*\breve \nabla^t \phi^*,
 \end{equation}
  we have
 \begin{equation}\label{neweqnB}
 \begin{split}
 \big(\bs C\, \nabla (w v), & \nabla (w\,\phi)\big)_{\Omega_{\rm c}}  =
 \big(|w|^2\bs B \breve \nabla v, \breve \nabla \phi \big)_{\Omega_{\rm c}}  +
 \big(w \bs B \breve \nabla  v,   (\breve \nabla w) \phi \big)_{\Omega_{\rm c}}  \\[4pt]
 & \quad +\big(\bs B \breve \nabla w\,  v,  w \breve \nabla \phi \big)_{\Omega_{\rm c}} +
  \big( (\breve \nabla w^*)^t  \bs B \breve \nabla w\,   v,   \phi \big)_{\Omega_{\rm c}}\,.
 \end{split}
 \end{equation}
Thus, the representation \eqref{sesquilinearB} follows from   \eqref{sesq2d3} and \eqref{neweqnB}.
\end{proof}

We next evaluate the terms involving $\bs B, n$ and $w,$ and  show that the singular coefficients
in \eqref{tilden} can be fully absorbed by $w.$  Although the derivation appears a bit lengthy and tedious, we strive to present the formulation in an accessible   manner, which depends only on the configuration of the layer: $R_1(\theta), \varrho,$ and the coordinate transformation: $T$ and $\sigma_0,\sigma_1.$ To
make the derivation  concise,  we also express them  in terms of 
\begin{equation}\label{notaA0f1}
\begin{split}
  & t=\frac{R_1}{R_1+ \sigma_1 T}=
  \frac{R_1(R_2-r)}{R_1(R_2-r)+ \sigma_1 (\varrho-1)R_1(r-R_1)}
  \in (0,1],
   \end{split}
\end{equation}
and use following regular functions in $r$:
\begin{equation}\label{notaA0f2}
\begin{split}
& \alpha=\sigma_1+\ri \sigma_0,\quad  \beta:=tS=R_1+(\alpha-\sigma_1)\tau, \quad \tau:=tT=\frac{(\varrho-1)R_1^2(r-R_1)}{R_1(R_2-r)+\sigma_1(\varrho-1)R_1(r-R_1)}, \\
  & \gamma_1:=t^2 T_r=  \frac{(\varrho-1)^2R_1^4}{(R_1(R_2-r)+\sigma_1(\varrho-1)R_1(r-R_1))^2},\\
  & \gamma_2:=t^2 T_\theta= (1-\varrho)R_1^2 R_1' \frac{  R_1(R_2-r)+r(r-R_1)   }{(R_1(R_2-r)+\sigma_1(\varrho-1)R_1(r-R_1))^2},
   \end{split}
\end{equation}
where we used
\begin{equation}\label{SSRTf1}
\begin{split}
  & T_r= \Big(\frac{(\varrho-1) R_1}{R_2-r} \Big)^2,
  \quad T_\theta= (1-\varrho) R_1' \frac{ R_1(R_2-r)+r(r-R_1)}{(R_2-r)^2}.
  \end{split}
\end{equation}


\smallskip 
With the following formulation in Cartesian coordinates  at our disposal, the implementation of the PAL technique 
using the spectral and finite elements becomes a normal coding exercise.
\begin{thm}\label{thm: implem}
The sesquilinear form $\breve{\mathscr{B}}(v, \phi)$ takes the form
\begin{equation}\label{eq:numerBlinear}
\breve{\mathscr{B}}(v, \phi)=\big(\bs T \breve{\bs B} {\bs T}^t\, \nabla v, \nabla \phi \big)_{\Omega_c}+(\bs T \bs p \cdot \nabla v, \phi)_{\Omega_c}+(v, \bs T {\bs q}^{*}\cdot \nabla \phi)_{\Omega_c}+(\breve n\, v, \phi)_{\Omega_c},
\end{equation}
where the matrix $\bs T$ is given in \eqref{rotationmat}. In $\Omega,$ we have
 $\breve{\bs B}=\bs I, \bs p=\bs q=\bs 0$ and $n=-k^2,$ while in   $\Omega_{\varrho}^{\rm PAL},$ the scalar function $\breve  n,$  and  the entries of the  matrix $\breve {\bs B}$, and the vectors
$\bs p,\bs q$  in \eqref{tilden} can be evaluated by the following expressions:
 \begin{equation}\label{eq:numern}
\breve n=k^2 \Big( \frac{\beta \sigma_1^2}{\alpha} -\alpha \beta + \frac{\sigma_1^2  R_1'^2 t^2 }{\alpha \beta} \Big)\frac{\gamma_1}{r} +\frac{9\gamma_1\sigma_1^2 t^2}{4 r R_1^2}\Big(\frac{\beta}{\alpha}+ \frac{R_1'^2}{ \alpha \beta
R_1^2}(R_1 t+\alpha \tau)^2 \Big). 
\end{equation}
\begin{equation}\label{eq:numerB}
\breve B_{11}=\frac{\beta t^2}{\alpha \gamma_1 r}\big(t^2+ (R_1' t^2+{\alpha \gamma_2})^2 {\beta^{-2}}  \big),\quad \breve B_{12}= -\big({ R_1' t^2}+{\alpha \gamma_2}\big)\frac{t^2} {\beta},\quad \breve B_{22}=\frac{\alpha \gamma_1 r}{\beta}t^2,
\end{equation}
\begin{subequations}\label{eq:numerP}
\begin{align}
&p_1=-\frac{\beta \sigma_1} {\alpha r}\Big(\frac {3t} {2R_1}+\ri k \Big)t^2 - \bigg\{\frac{3R_1' \sigma_1}{2rR_1 \alpha\beta}\Big(t+\frac{\alpha \tau}{R_1}\Big)
+ \frac{\ri k \sigma_1 R_1'}{r\alpha\beta}\bigg\} (R_1't^2+\alpha\gamma_2)t^2, \label{eq:numerp1}\\
&p_2=\frac{3R_1'\gamma_1 \sigma_1}{2R_1 \beta}\Big(t+\frac{\alpha \tau}{R_1}\Big)t^2
+ \ri k \sigma_1  \frac{R_1'\gamma_1}{\beta}t^2,\label{eq:numerp2}
\end{align}
\end{subequations}
and the elements of $\bs q$ can be obtained  by changing the signs in front of $\ri k $ in $p_1$ and $p_2$, i.e.,
$-\ri k$ in place of $\ri k$ in \eqref{eq:numerP}.
\end{thm}
\begin{proof} Using \eqref{opsA} and the property: $\bs T^{-1}=\bs T,$ we have $\breve \nabla=\bs T^t \nabla.$
Accordingly,  we can rewrite the formulation of  $\breve{\mathscr{B}}(v, \phi)$ in  Lemma \ref{thm:BVaPhi} as
\begin{equation}\label{sesquilinearB00}
\begin{split}
\breve {\mathscr B}(v,\phi)
&=\big(\breve{\bs B} \breve \nabla v, \breve \nabla \phi \big)_{\Omega_{\rm c}} + \big(\bs p\cdot \breve \nabla  v,    \phi \big)_{\Omega_{\rm c}}  +\big(  v,  \bs q^*\cdot  \breve \nabla \phi \big)_{\Omega_{\rm c}} +\big(\breve n\, v, \phi\big)_{\Omega_{\rm c}}\\
&=(\bs T \breve{\bs B} {\bs T}^t \,\nabla v, \nabla \phi )_{\Omega_c}+(\bs T \bs p \cdot \nabla v, \phi)_{\Omega_c}+(v, \bs T {\bs q}^{*}\cdot \nabla \phi)_{\Omega_c}+(\breve n\, v, \phi)_{\Omega_c}.
\end{split}
\end{equation}
Now, the main  task is to derive the representations in  \eqref{eq:numern}-\eqref{eq:numerP}.
For clarity, we deal with them separately in three cases below.

\smallskip

(i) We first derive  \eqref{eq:numern}  from the expression in \eqref{tilden}.  By  direct calculation, we find
\begin{equation}\label{eq:wbw0}
\begin{aligned}
(\breve \nabla w^*)^t  \bs B\, \breve \nabla w&=( w^{*}_r,w^{*}_\theta/r)
\begin{pmatrix}
B_{11}& B_{12}\\
B_{12}& B_{22}
\end{pmatrix}
\begin{pmatrix}
 w_r\\
w_{\theta}/r
\end{pmatrix}\\
&=B_{11}|w_r|^2+\frac{B_{22}}{r^2}|w_{\theta}|^2+\frac{B_{12}}{r} ( w_r^{*} w_{\theta}+w_{\theta}^{*}w_r)\\
&=B_{11}|w_r|^2+\frac{B_{22}}{r^2}|w_{\theta}|^2+\frac{2  B_{12}}{r}\Re\{w_r^{*} w_{\theta}\}.
\end{aligned}
\end{equation}
 Denote  $A=t^{3/2},$  so $w=A e^{\ri k \sigma_1 T}$ and $A=|w|.$  Then, we find readily that
\begin{equation}\label{newnotA}
\begin{split}
w_r=(A_r +\ri k \sigma_1  T_r A)e^{\ri k \sigma_1 T}=\Big(\frac{A_r} A +\ri k \sigma_1  T_r\Big)w,\quad
w_\theta=\Big(\frac{A_\theta} A +\ri k \sigma_1  T_\theta\Big)w,
\end{split}
\end{equation}
so we have 
\begin{equation}\label{newnotB}
\begin{split}
& |w_r|^2=\Big|\frac{A_r} A +\ri k \sigma_1  T_r\Big|^2 |w|^2=A_r^2+k^2\sigma_1^2 T_r^2A^2,\quad
|w_\theta|^2=A_\theta^2+k^2\sigma_1^2 T_\theta^2A^2,\\[4pt]
& \Re\{w_r^{*} w_{\theta}\}= A_rA_\theta+k^2\sigma_1^2 T_rT_\theta A^2, \quad\\[4pt]
& w_r^* w=A A_r-\ri k \sigma_1  T_r A^2, \quad w_\theta^* w=A A_\theta -\ri k \sigma_1  T_\theta A^2.
\end{split}
\end{equation}
Recall from  \eqref{rotationmat} that
\begin{equation}\label{BsBs}
B_{11}=\dfrac{S}{r S_r}\Big(1+\dfrac{S_{\theta}^2}{S^2} \Big),\quad B_{12}=-\frac{S_\theta} S,\quad
B_{22}=\frac{rS_r} S,\quad n=\frac{SS_r} r.
\end{equation}
Then inserting \eqref{newnotB}-\eqref{BsBs} into  \eqref{eq:wbw0}, we derive
\begin{equation}\label{SwBtW}
\begin{split}
(\breve \nabla w^*)^t  \bs B\, \breve \nabla w &=\big(A_r^2+k^2\sigma_1^2 T_r^2A^2\big)
\Big(\dfrac{S}{r S_r}+\dfrac{S_{\theta}^2}{r SS_r} \Big)+\frac {S_r} {rS}\big(A A_\theta +\ri k \sigma_1  T_\theta A^2\big)\\
&\quad -\frac{2S_\theta}{rS} \big(A_rA_\theta+k^2\sigma_1^2 T_rT_\theta A^2\big)\\
&= \dfrac{S}{r S_r} \big(A_r^2+k^2\sigma_1^2 T_r^2A^2\big)
+\dfrac{1}{r SS_r}\big(A_r^2S_\theta^2+A_\theta^2 S_r^2-2S_rS_\theta A_r A_\theta\big)\\
&\quad +\dfrac{k^2\sigma_1^2A^2}{r SS_r}\big(T_r^2S_\theta^2+T_\theta^2 S_r^2-2S_rS_\theta T_r T_\theta\big)\\
&=\dfrac{S}{r S_r} \big(A_r^2+k^2\sigma_1^2 T_r^2A^2\big)+
\dfrac{1}{r SS_r}\big(A_rS_\theta - A_\theta S_r\big)^2+\dfrac{k^2\sigma_1^2A^2}{r SS_r}
\big(T_rS_\theta - T_\theta S_r\big)^2.
\end{split}
\end{equation}
By \eqref{tilden} and \eqref{BsBs}-\eqref{SwBtW},
\begin{equation}\label{newBas}
\begin{split}
\breve n & =(\breve \nabla w^*)^t  \bs B\, \breve \nabla w -k^2 \, |w|^2\, n=
A_r^2 \dfrac{S}{r S_r}+k^2A^2\Big(\sigma_1^2 T_r^2 \frac{S}{rS_r} -\frac{SS_r} r\Big)\\
&\quad +\dfrac{1}{r SS_r}\big(A_rS_\theta - A_\theta S_r\big)^2+k^2\sigma_1^2\dfrac{A^2}{r SS_r}
\big(T_rS_\theta - T_\theta S_r\big)^2.
\end{split}
\end{equation}

We have from \eqref{asBeqn}-\eqref{TtransformA}  that in $\Omega^{\rm PAL}_{\varrho},$
$S=R_1+\alpha T,$ $R_2=\varrho R_1,$ and
\begin{equation}\label{SSRTf0}
\begin{split}
  & T=\dfrac{(R_2-R_1)(r-R_1)}{R_2-r},\quad S_r=\alpha T_r, \quad S_\theta=R_1'+\alpha T_\theta.
  \end{split}
\end{equation}
Moreover, from the expression  of $t$ in \eqref{notaA0f1}, we find immediately that
\begin{equation}\label{ArBr}
\begin{split}
 & A_r=\frac 3 2 t^{\frac 1 2}t_r=-\frac{3\sigma_1}{2R_1} t^{\frac 5 2} T_r, 
  \quad A_\theta=\frac 3 2 t^{\frac 1 2}t_\theta=\frac{3\sigma_1 t^{\frac 5 2}}{2R_1} \Big( \frac{R_1'}{R_1}T-T_\theta\Big).
 \end{split}
\end{equation}
Then by \eqref{SSRTf0}-\eqref{ArBr},
\begin{equation}\label{TsTsf1}
\begin{split}
A_rS_\theta - A_\theta S_r & =-\frac{3\sigma_1}{2R_1} t^{\frac 5 2} T_r (R_1'+\alpha T_\theta)-
\frac{3\sigma_1 t^{\frac 5 2}}{2R_1} \Big( \frac{R_1'}{R_1}T-T_\theta\Big)\alpha T_r =
-\frac{3R_1' \sigma_1}{2R_1} t^{\frac 5 2} T_r \Big(1+ \frac{\alpha T}{R_1}  \Big),
\end{split}
\end{equation}
and
\begin{equation}\label{TsTs}
\begin{split}
T_rS_\theta - T_\theta S_r=T_r (R_1'+\alpha T_\theta)-T_\theta (\alpha T_r)=R_1'T_r.
\end{split}
\end{equation}
We now express $\breve n$ in terms of the regular functions $t, \beta, \tau, \gamma_1,\gamma_2$ in \eqref{notaA0f1}-\eqref{notaA0f2} as follows:
\begin{equation}\label{newBasf2}
\begin{split}
\breve n & =A_r^2 \dfrac{S}{r S_r}+k^2A^2\Big(\sigma_1^2 T_r^2 \frac{S}{rS_r} -\frac{SS_r} r\Big)+\dfrac{1}{r SS_r}\frac{9R_1'^2 \sigma_1^2}{4R_1^2} t^{5 } T_r^2 \Big(1+\alpha \frac{T}{R_1}  \Big)^2+
k^2\sigma_1^2\dfrac{A^2}{r SS_r} R_1'^2 T_r^2\\
&=\frac{9\gamma_1\sigma_1^2 t^2}{4 r R_1^2} \frac{\beta}{\alpha}+\frac{k^2\gamma_1} r\Big(\sigma_1^2 \frac{\beta} {\alpha} -\alpha\beta\Big)+
\frac{9\gamma_1\sigma_1^2 R_1'^2 t^2}{4 r R_1^2  \alpha \beta
}\Big(t+\frac{\alpha \tau} {R_1} \Big)^2 +   k^2 \sigma_1^2 \frac{  R_1'^2\gamma_1 t^2 }{r \alpha \beta}.
\end{split}
\end{equation}
We obtain the identity  \eqref{eq:numern}   immediately by regrouping the terms.

\medskip

(ii) Now, we calculate the elements of the vectors $\bs p$ and $\bs q$ in   \eqref{sesquilinearB}, that is,
\begin{equation}\label{eq:qqq}
\begin{split}
\bs p=(p_1,p_2)^t & :=w \bs B \breve \nabla w^{*}= \bs B (\breve \nabla w^{*} w) 
=  \big(B_{11}w_r^* w+B_{12}w_\theta^* w/r, B_{12}w_r^* w +B_{22}w_\theta^* w/r\big)^t.
  \end{split}
\end{equation}
Then by  \eqref{notaA0f1}-\eqref{notaA0f2},  \eqref{newnotB}-\eqref{BsBs} and  \eqref{TsTsf1} -\eqref{TsTs},
\begin{equation}\label{eq:ppp1}
\begin{split}
p_1&=B_{11}w_r^* w+\frac{B_{12}} r w_\theta^* w=\dfrac{S}{r S_r}\Big(1+\dfrac{S_{\theta}^2}{S^2} \Big) \big(A A_r-\ri k \sigma_1  T_r A^2\big)-\frac{S_\theta}{rS}\big(A A_\theta -\ri k \sigma_1  T_\theta A^2\big)\\
&=\dfrac{S}{r S_r}  \big(A A_r-\ri k \sigma_1  T_r A^2\big)+\dfrac{S_\theta A}{r SS_r}\big(A_rS_\theta - A_\theta S_r\big)
-\ri k\sigma_1 \dfrac{S_\theta A^2}{r SS_r}
(T_rS_\theta - T_\theta S_r\big)\\
&= -\frac{\beta} {\alpha r}\Big(\frac {3\sigma_1 t} {2R_1}+\ri k \sigma_1\Big)t^2 -\frac{3R_1' \sigma_1}{2rR_1 \alpha\beta}\Big(t+\frac{\alpha \tau}{R_1}\Big)
(R_1't^2+\alpha \gamma_2)t^2- \frac{\ri k \sigma_1 R_1'}{r\alpha\beta}(R_1't^2+\alpha\gamma_2)t^2\\
&=-\frac{\beta} {\alpha r}\Big(\frac {3\sigma_1 t} {2R_1}+\ri k \sigma_1\Big)t^2 - \Big(\frac{3R_1' \sigma_1}{2rR_1 \alpha\beta}\Big(t+\frac{\alpha \tau}{R_1}\Big)
+ \frac{\ri k \sigma_1 R_1'}{r\alpha\beta}\Big) (R_1't^2+\alpha\gamma_2)t^2,
  \end{split}
\end{equation}
and
\begin{equation}\label{eq:ppp2}
\begin{split}
p_2&=B_{12}w_r w^* +\frac{B_{22}} rw_\theta w^*=
-\dfrac{S_\theta}{ S} \big(A A_r-\ri k \sigma_1  T_r A^2\big)+\frac{ S_r}{S}\big(A A_\theta -\ri k \sigma_1  T_\theta A^2\big)\\
&=\dfrac{A}{S} \big( A_\theta S_r-A_rS_\theta\big)
+\ri k\sigma_1 \dfrac{A^2}{S}
(T_rS_\theta - T_\theta S_r\big)\\
&= \frac{3R_1'\gamma_1 \sigma_1}{2R_1 \beta}\Big(t+\frac{\alpha \tau}{R_1}\Big)t^2
+ \ri k \sigma_1  \frac{R_1'\gamma_1}{\beta}t^2.
  \end{split}
\end{equation}
Thus, we obtain  \eqref{eq:numerP}.

We now turn to the elements of the vector  $\bs q.$ Note that
\begin{equation}\label{eq:ppp}
\begin{split}
\bs q=(q_1,q_2)^t & :=w^* \bs B \breve \nabla w 
= \big(B_{11}w_r w^*+B_{12}w_\theta w^*/r, B_{12}w_r w^* +B_{22}w_\theta w^*/r\big)^t.
  \end{split}
\end{equation}
Following the same lines as in  \eqref{eq:ppp1}-\eqref{eq:ppp2}, we find that $q_1,q_2$ are identical to
$p_1,p_2$ with a change of $\ri k$ therein to $-\ri k.$
\medskip

(iii) Finally, we deal with the matrix $\breve {\bs B},$ that is,
\begin{equation}\label{bsB}
\breve {\bs B}=|w|^2 \bs B =\begin{pmatrix}
\dfrac{t^3 S}{r S_r}\Big(1+\dfrac{S_{\theta}^2}{S^2} \Big)\;\; &  -\dfrac{t^3 S_{\theta}}{S}\\[8pt]
-\dfrac{t^3 S_{\theta}}{S} & \dfrac{r t^3 S_r}{S}
\end{pmatrix}:=\begin{pmatrix}
\breve {B}_{11}\;\; & \breve {B}_{12}\\[2pt]
\breve {B}_{12} & \breve {B}_{22}
\end{pmatrix}.
\end{equation}
Using the notation in \eqref{notaA0f2} and  the properties in \eqref{SSRTf0},  we can derive  the entries of
$\breve {\bs B}$ in \eqref{eq:numerB}
straightforwardly.
\end{proof}

\begin{rem}\label{circular} {\em 
In the circular case, the artificial layer  is an annulus, i.e.,
\begin{equation}\label{circularOmega2}
  \Omega_{\varrho}^{\rm PAL}=\{R_1<r<R_2=\varrho R_1,\;0\le \theta<2\pi\},
\end{equation}
where $R_1$ is independent of $\theta.$ Thus, $R_1'$ and $\gamma_2$ in \eqref{notaA0f2} reduce to $0.$ Consequently, several terms in \eqref{eq:numern}-\eqref{eq:numerP} vanish and the variables involved in the sesquilinear form $\breve{\mathscr{B}}(v, \phi)$ in \eqref{eq:numerBlinear} are reduced to
\begin{equation}\label{circweakform1}
\begin{split}
& \breve B_{11}=\frac{\beta}{\alpha \gamma_1 r}t^4,\quad \breve B_{12}=0,\quad \breve B_{22}=\frac{\alpha \gamma_1 r}{\beta}t^2,\quad \breve n=k^2 \Big( \frac{\beta \sigma_1^2}{\alpha} -\alpha \beta \Big)\frac{\gamma_1}{r} +\frac{9 \beta }{4 \alpha } \frac{\gamma_1 \sigma_1^2}{r R_1^2}t^2,\\
&  p_1=-\frac{\beta \sigma_1}{\alpha r}\Big( \frac{3t}{2R_1}+\ri k   \Big)t^2,\quad p_2=0,\quad q_1=-\frac{\beta \sigma_1}{\alpha r}\Big( \frac{3t}{2R_1}-\ri k  \Big)t^2,\quad q_2=0.
\end{split}
\end{equation} }
\end{rem}

\section{Numerical results and comparisons}\label{sect5:numer}

Consider the exterior wave scattering problem \eqref{extproblem2d} with the domain truncated with a general shar-shaped PAL layer. Assume that in all the numerical tests, a plane wave $e^{\ri k r\cos(\theta-\theta_0)}$ is incident onto the scatterer $D$ with incident angle $\theta_0$. Correspondingly, $g$ in \eqref{extproblem2d} takes the form
\begin{equation}\label{eq: sourceg}
g=-\exp\big(\ri k R_0(\theta)\cos(\theta-\theta_0)\big),
\end{equation}
given that the boundary of the scatterer $D$ is parameterized by $\partial D=\{(r,\theta): r=R_0(\theta), \theta=[0,2\pi) \}.$ 
\subsection{Circular PAL layer}
To investigate the performance of the proposed method, we start by solving \eqref{extproblem2d} with a circular scatterer $D$, so the exact solution is available as a series expansion
\begin{equation}
U(\bs x)=-\sum_{|m|=0}^\infty    \frac{ \ri^m  J_m(kR_0)}{H^{(1)}_m (kR_0)}H^{(1)}_m (kr) e^{\ri m(\theta-\theta_0)},\quad r>R_0.
\end{equation}
The domain is truncated via an annular PAL layer. The implementation is  based on Theorem \ref{thm: implem} with coefficients give by   \eqref{circweakform1}. The parameters  are set to be $\sigma_0=\sigma_1=1.$    We also compare it with PML with bounded and unbounded absorbing functions, i.e., PML$_n$ and PML$_\infty$ in \eqref{RegularKernel}-\eqref{SingleKernel}, respectively.  

Here, we use Fourier expansion approximation in $\theta$ direction, and spectral-element method in radial direction \cite{Shen.TWBook01}.  In the test, we fix $(R_0, R_1, R_2)=(1, 2, 2.2)$ and the incident angle $\theta_0=0.$  Let $M$ be the cut-off number of the Fourier modes, and $\bs N=(N_1,N)$ be the highest polynomial degrees in $r$-direction of two layers, respectively. We measure the maximum errors in $\Omega.$ We fix $N_1=300,$ $M=kR_1$ and vary $N$ so that the waves in the interior layer can be well-resolved, and the error should be dominated by the approximation in the outer annulus.  In Figure \ref{fig: comparesty}, we compare the accuracy of the solver with PAL, PML$_{n}$ (with $n=1$, $\sigma_0=(5.16, 2.78,1.89, 1.43, 1.15, 1.01)$ for $k=(50, 100, 150, 200, 250, 300)$, respectively: optimal value based on the rule in \cite{chen2005adaptive}),  and PML$_\infty$ ($\sigma_0=1/k$, as suggested in \cite{bermudez2007optimal}) for various $k$. It can be seen from Figure \ref{fig: comparesty}(a) that when the wavenumber $k$ is relatively small, the error history lines of these three methods intertwine with each other for polynomial degree $N<20$ and the error obtained by PML$_{\infty}$ is slightly smaller than that obtained by the other two methods. As $N$ increases, the convergence error for PML$_{\infty}$ becomes much larger than PAL and PML$_n$ due to large roundoff errors induced by the Gauss quadrature of singular functions. As depicted in Figure \ref{fig: comparesty} (b)-(f), when $k$ increases, the PAL obviously outperforms its rivals. 
For instance, when $k=150$ and $N=20$, the errors for PAL is around $10^{-8}$ while that for PML$_n$ and PML$_{\infty}$ are around $10^{-3}\sim 10^{-4}. $ These comparisons show that the PAL method is clearly advantageous, especially when the  wavenumber is large. 

\begin{figure}[htbp]
\begin{center}
  \subfigure[$k=50$ ]{ \includegraphics[scale=.27]{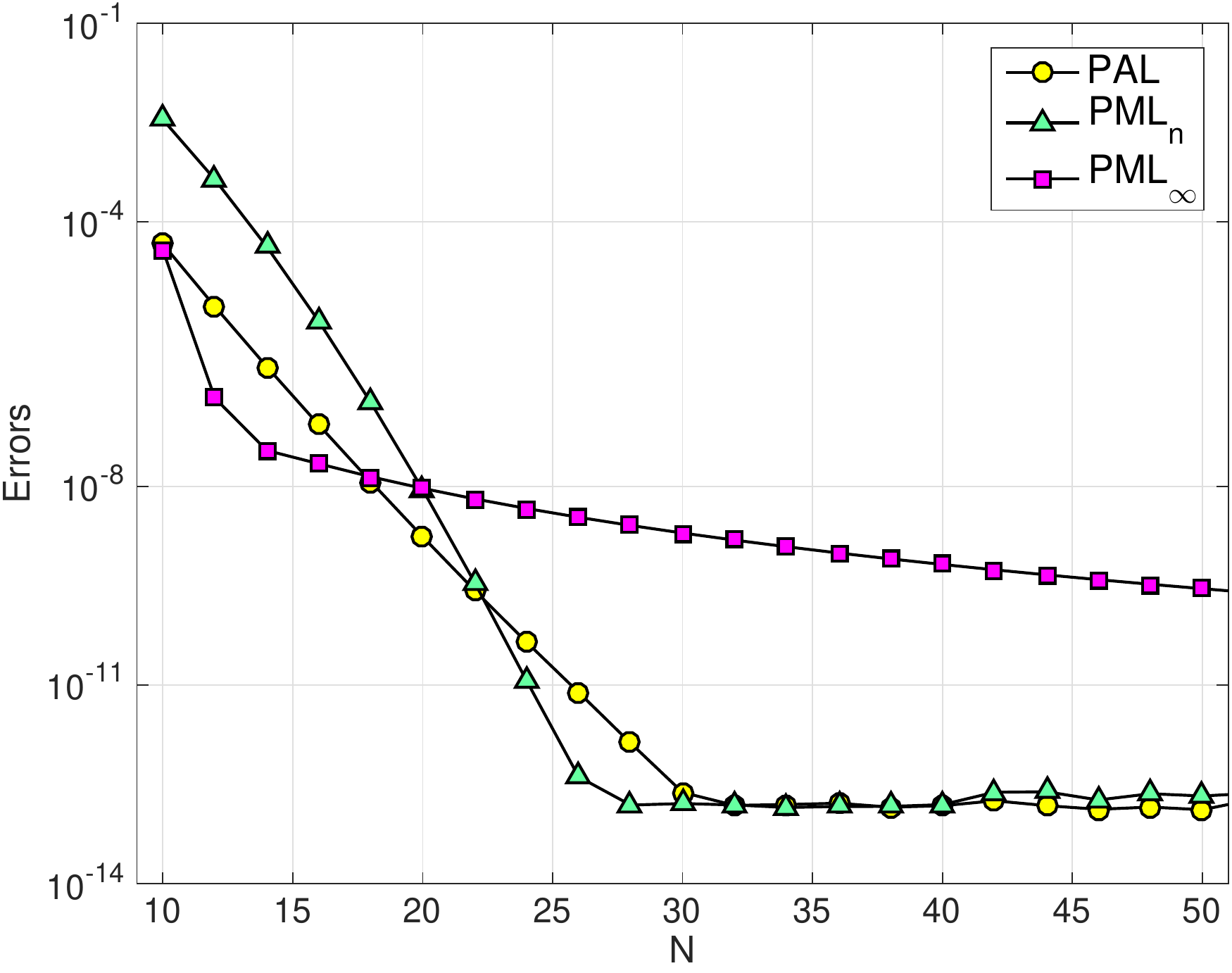}}\;\;
  \subfigure[$k=100$ ]{ \includegraphics[scale=.27]{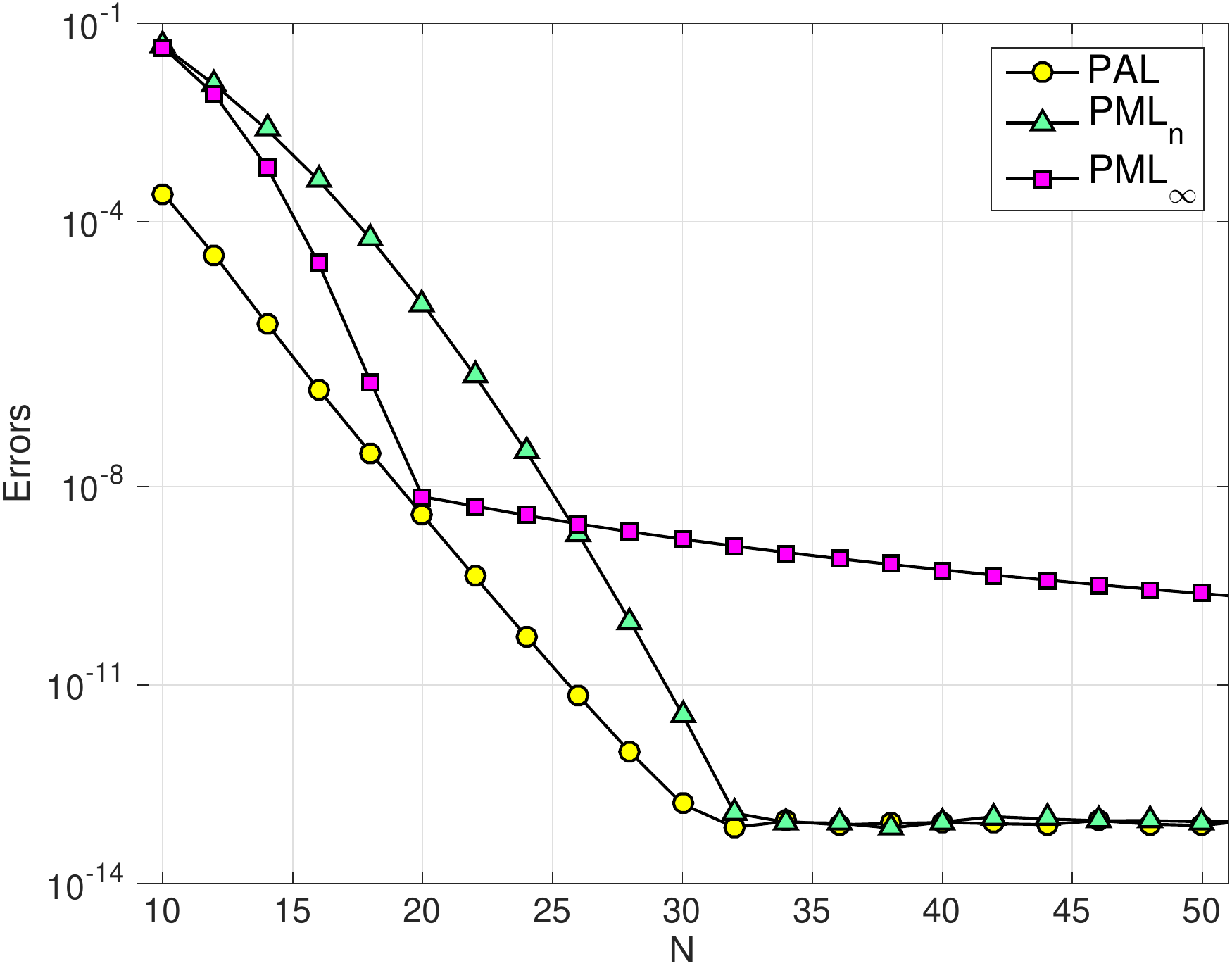}}\;\;
 \subfigure[ $k=150$]{ \includegraphics[scale=.27]{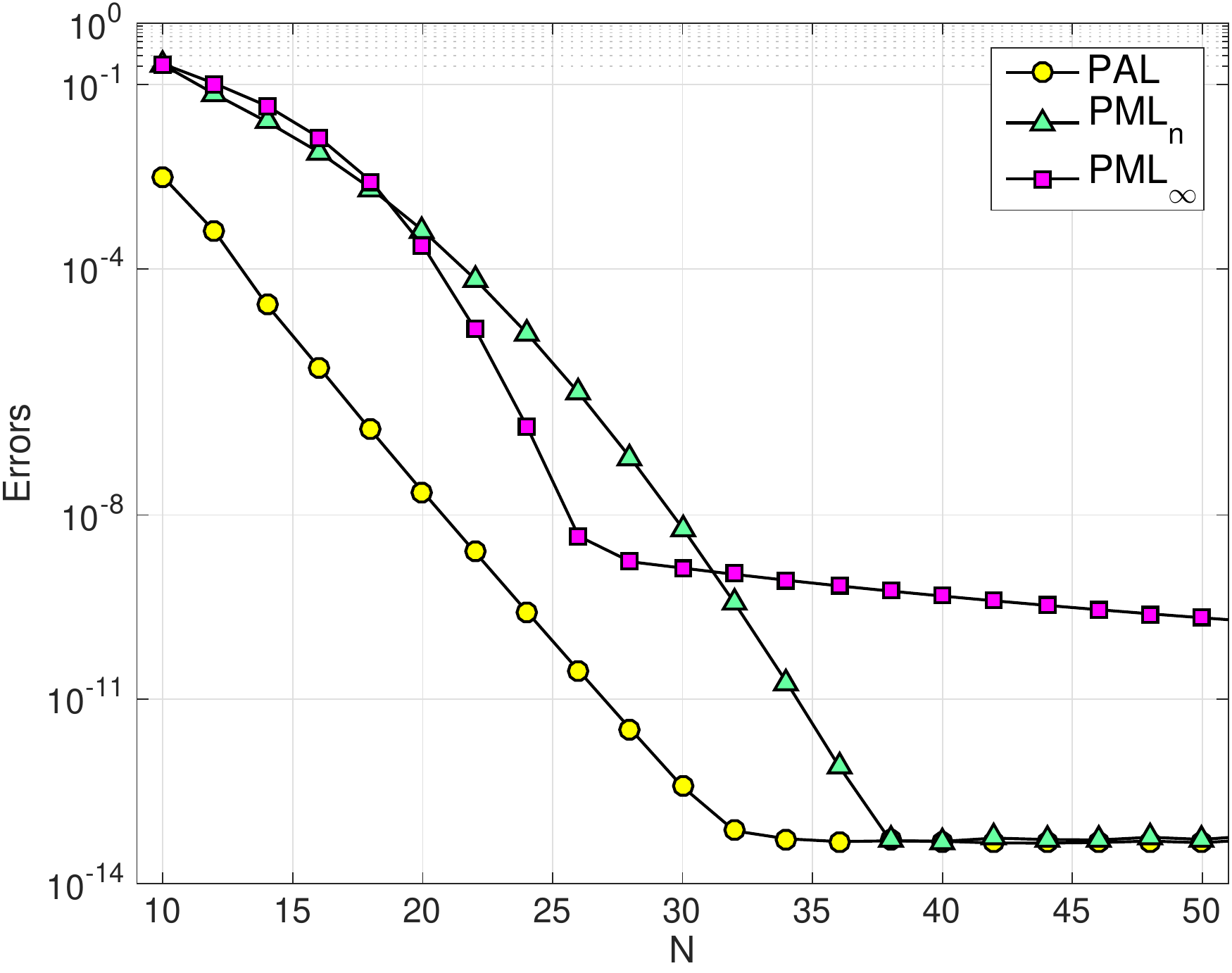}} \\
  \subfigure[ $k=200$]{ \includegraphics[scale=.27]{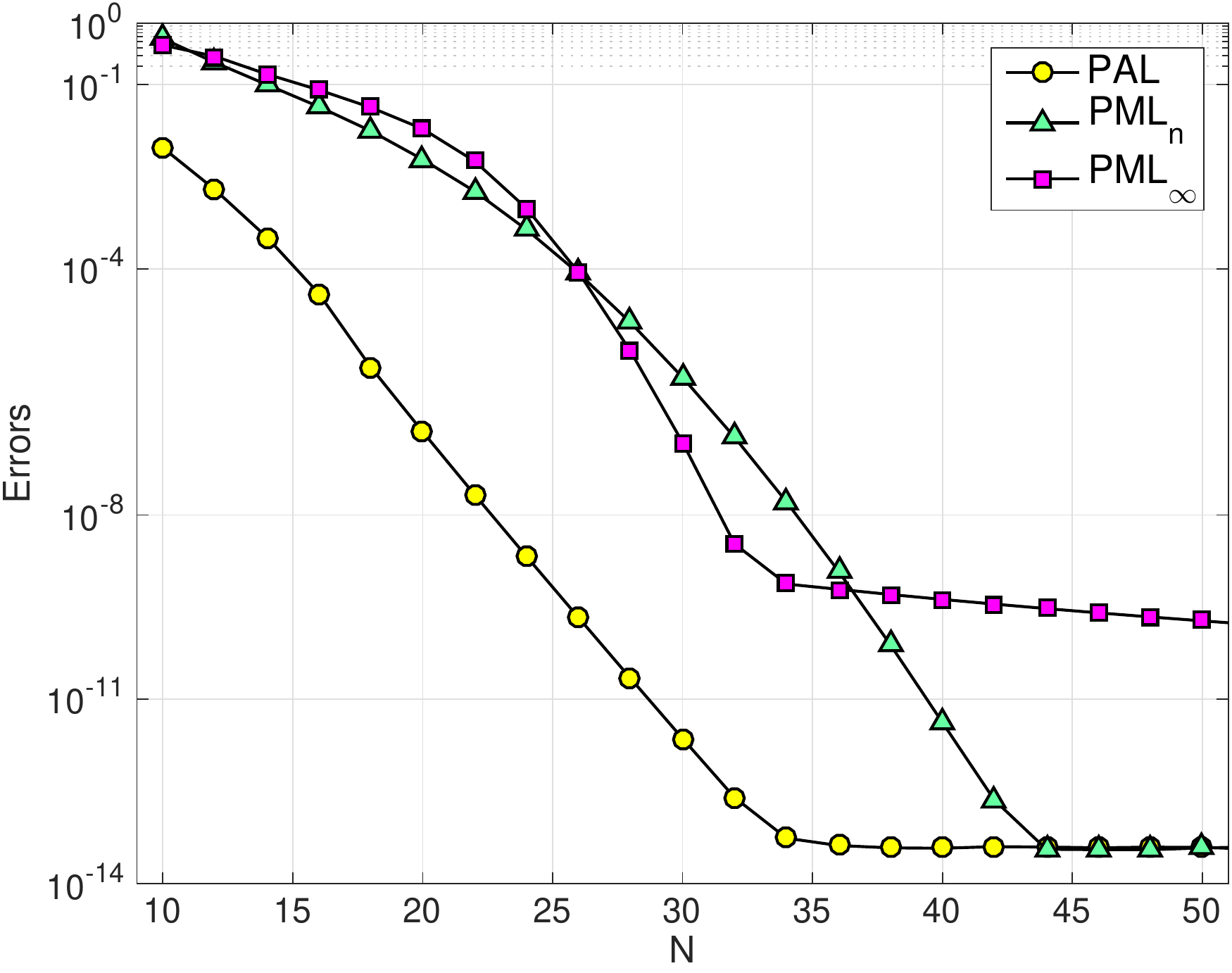}}\;\;
  \subfigure[ $k=250$]{ \includegraphics[scale=.27]{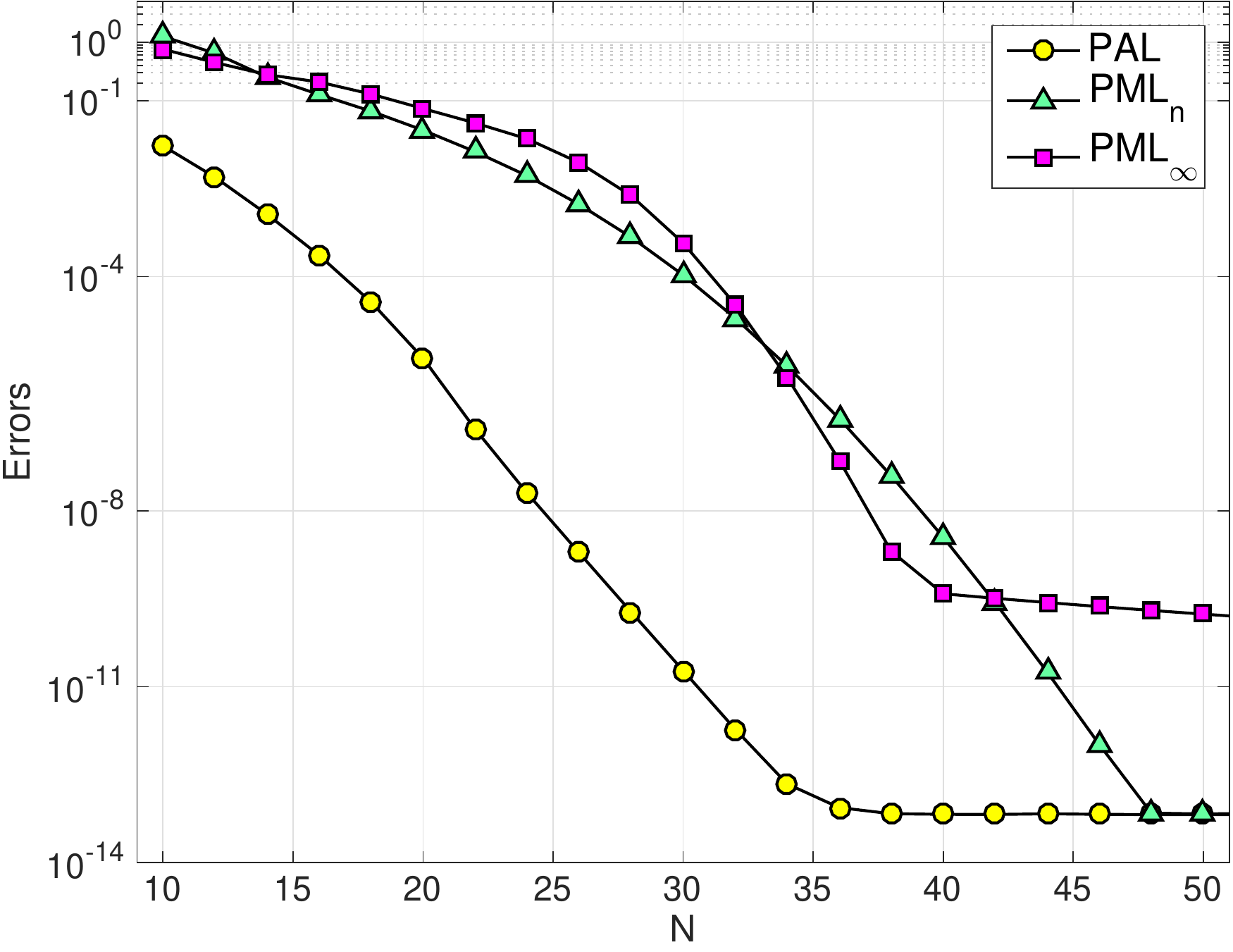}}\;\;
 \subfigure[$k=300$ ]{ \includegraphics[scale=.27]{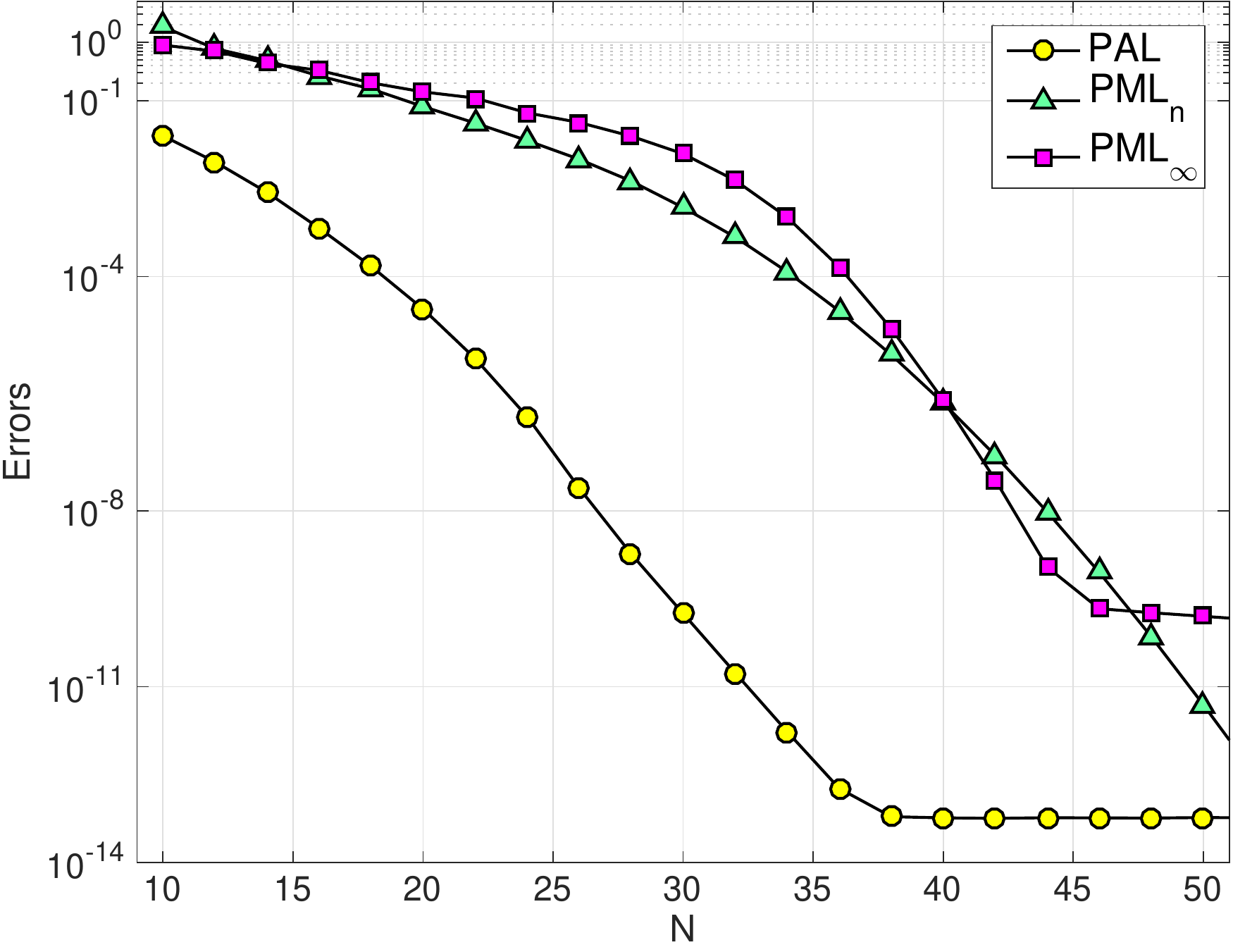}} \\
    \caption{\small A comparation study: PAL vs $\text{PML}_n$ vs $\text{PML}_{\infty}$ for various $k=50,100,150,200,250,300.$} 
   \label{fig: comparesty}
\end{center}
\end{figure}

In order to study the influence of the thickness of the artificial layer, we keep $R_1=2$, vary $R_2$ and tabulate in Table \ref{tab: thickness} the numerical errors  for a fixed polynomial degree $(N_1, N)=(300,30)$ for a large range of $k=(10, 50, 100, 200, 300, 500).$ It demonstrates that the value of the thickness $d=R_2-R_1$ has major impact on the accuracy. To optimize the PAL method and achieve high accuracy,   a good  choice of the thickness turns to be   $d=\frac{10}{kR_1}.$ We also find the PAL is much less dependent on the choice of $\sigma_0$ and  $\sigma_1,$ and it is safe to choose   $\sigma_0=\sigma_1=1,$ which is largely due to the compression of the transformation. 

\begin{table}
  \centering
    \caption{Errors vs thickness $d$ for PAL with $(N_1,N)=(300, 30)$}
    \vspace*{-5pt}
  \begin{tabular}{@{}rrrrrr@{}}
    \midrule
 $k$   & \multicolumn{1}{c}{$d=1$} & \multicolumn{1}{c}{$d=0.5$} & \multicolumn{1}{c}{$d=0.1$} & \multicolumn{1}{c}{$d=0.01$}& \multicolumn{1}{c}{$d=0.001\!\!$} \\
    \cmidrule{2-6}
    10 & 1.11E-12   &1.60E-12   & 1.26E-8&  1.39E-4 &1.69E-2 \\
    50 & 9.92E-11 & 4.67E-13 & 3.27E-12 &2.82E-7 &5.69E-4 \\
    100 &2.65E-6 &3.79E-11  &2.08E-13  &1.34E-8  &8.65E-5\\
    200 &5.72E-4  &1.91E-6  &9.60E-14  &3.12E-10  &7.73E-6 \\
    300 &  3.60E-3& 7.51E-5 &2.65E-13  &3.03E-11  &1.30E-6 \\
    500 &1.57E-2 & 1.40E-3 &1.33E-11  &1.90E-12  &1.68E-7 \\
        \midrule
  \end{tabular}
     \label{tab: thickness}
\end{table}

\begin{figure}[htbp]
\begin{center}
 \subfigure[illustration of mesh grids]{ \includegraphics[scale=.36]{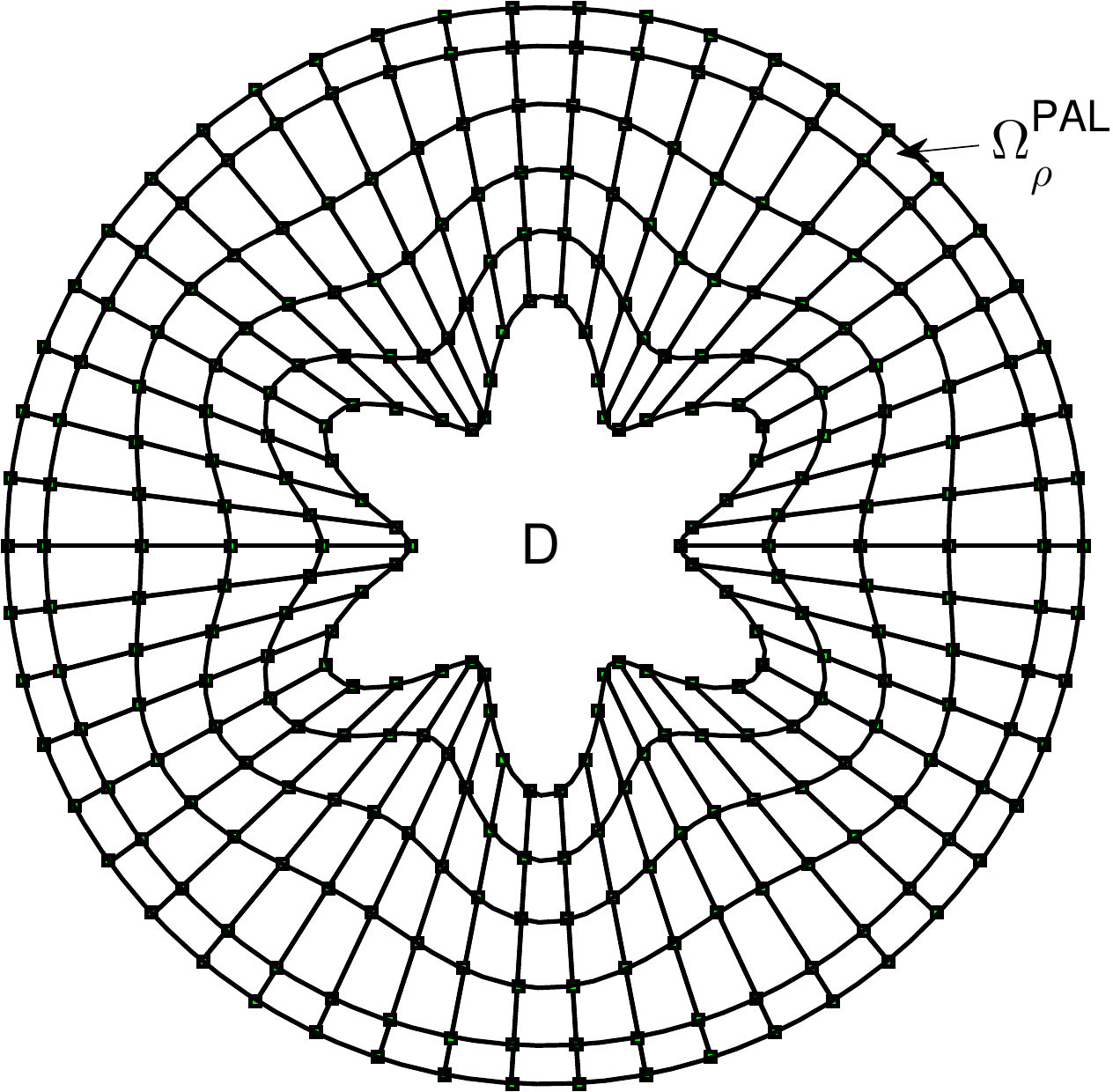}} \qquad 
  \subfigure[${\rm Re}(u)$ and ${\rm Re}(v)$]{ \includegraphics[scale=.36]{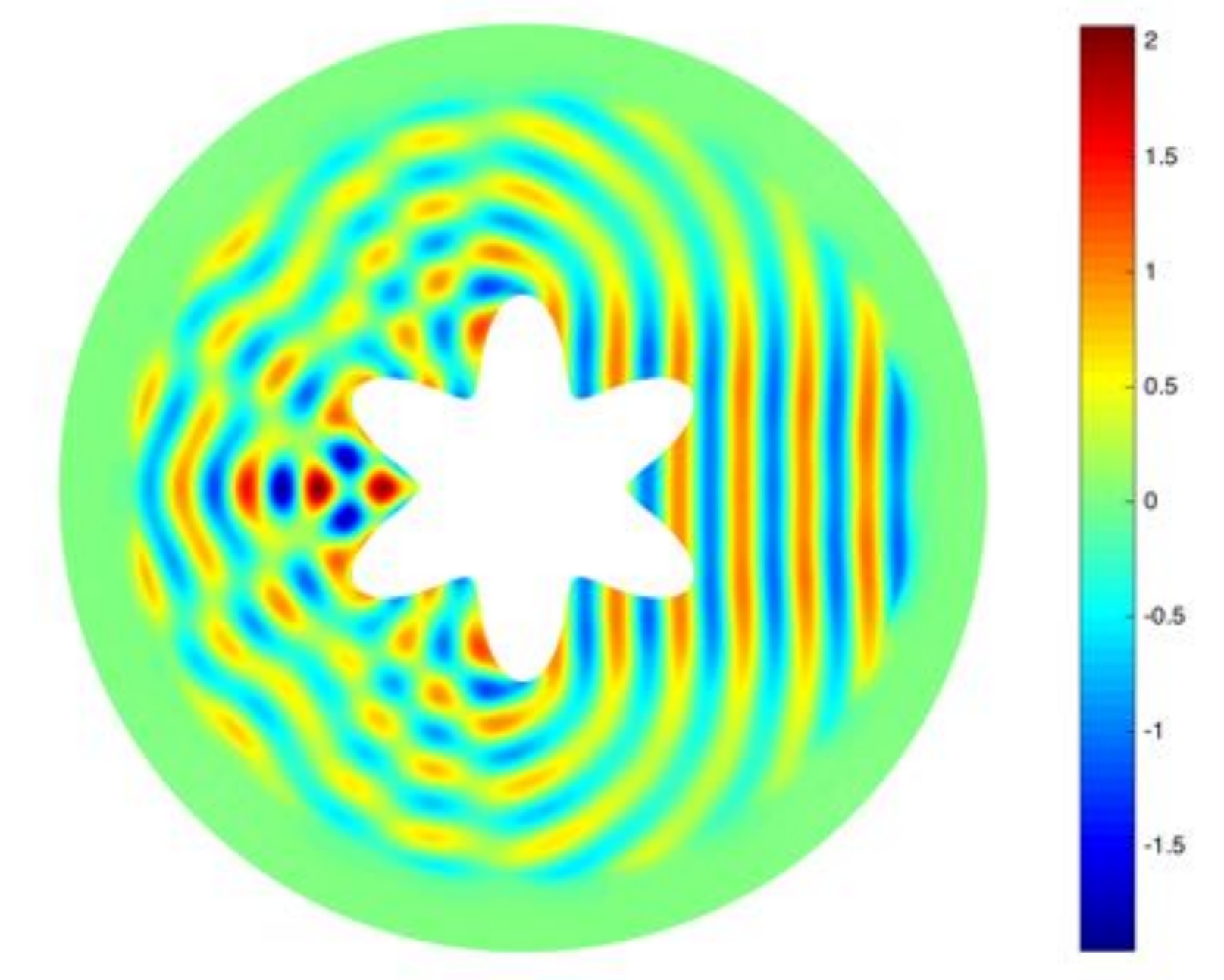}}\\
 \subfigure[errors of (b) against $N$ ]{ \includegraphics[scale=.25]{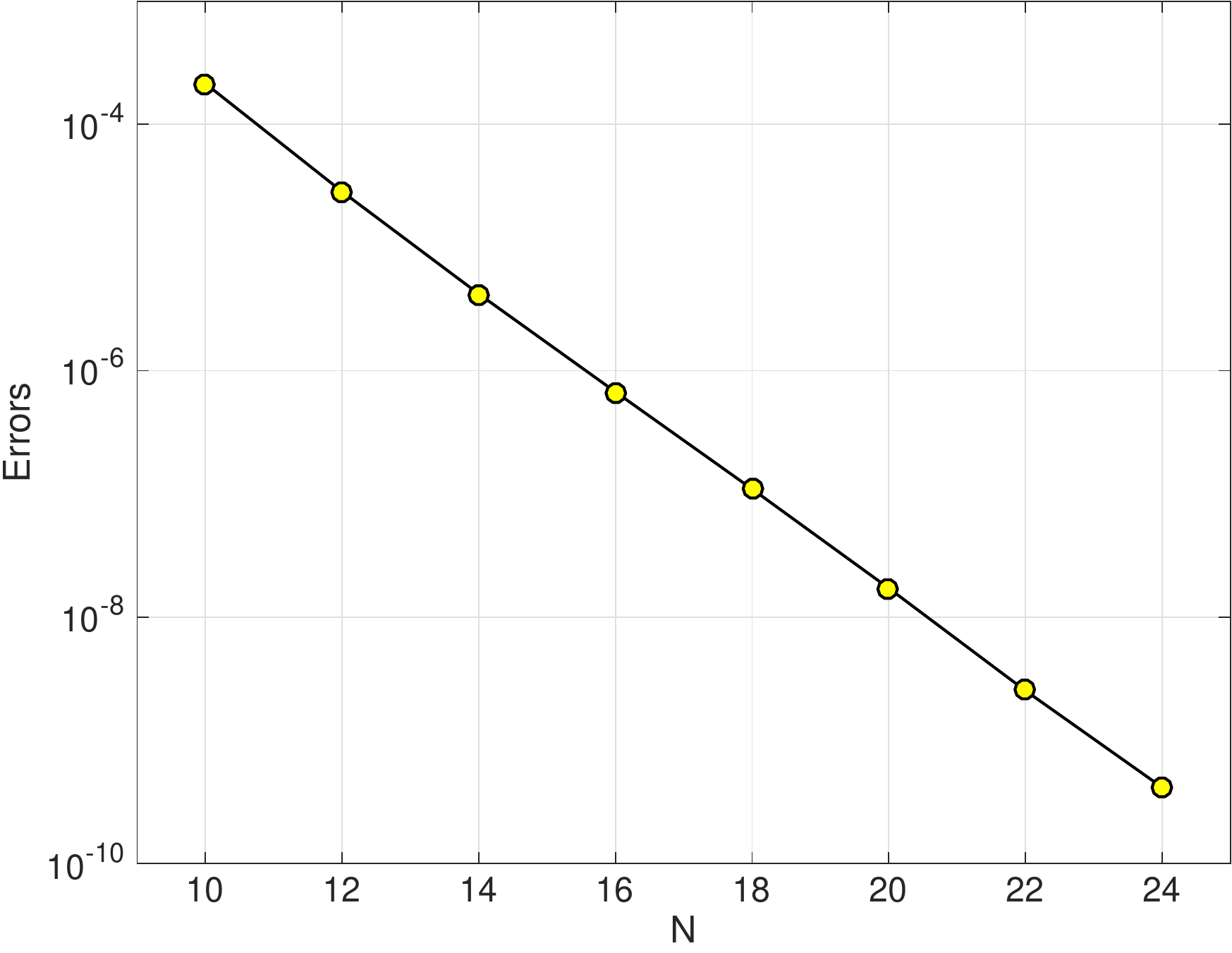}}\quad
 \subfigure[profile along x-axis]{ \includegraphics[scale=.25]{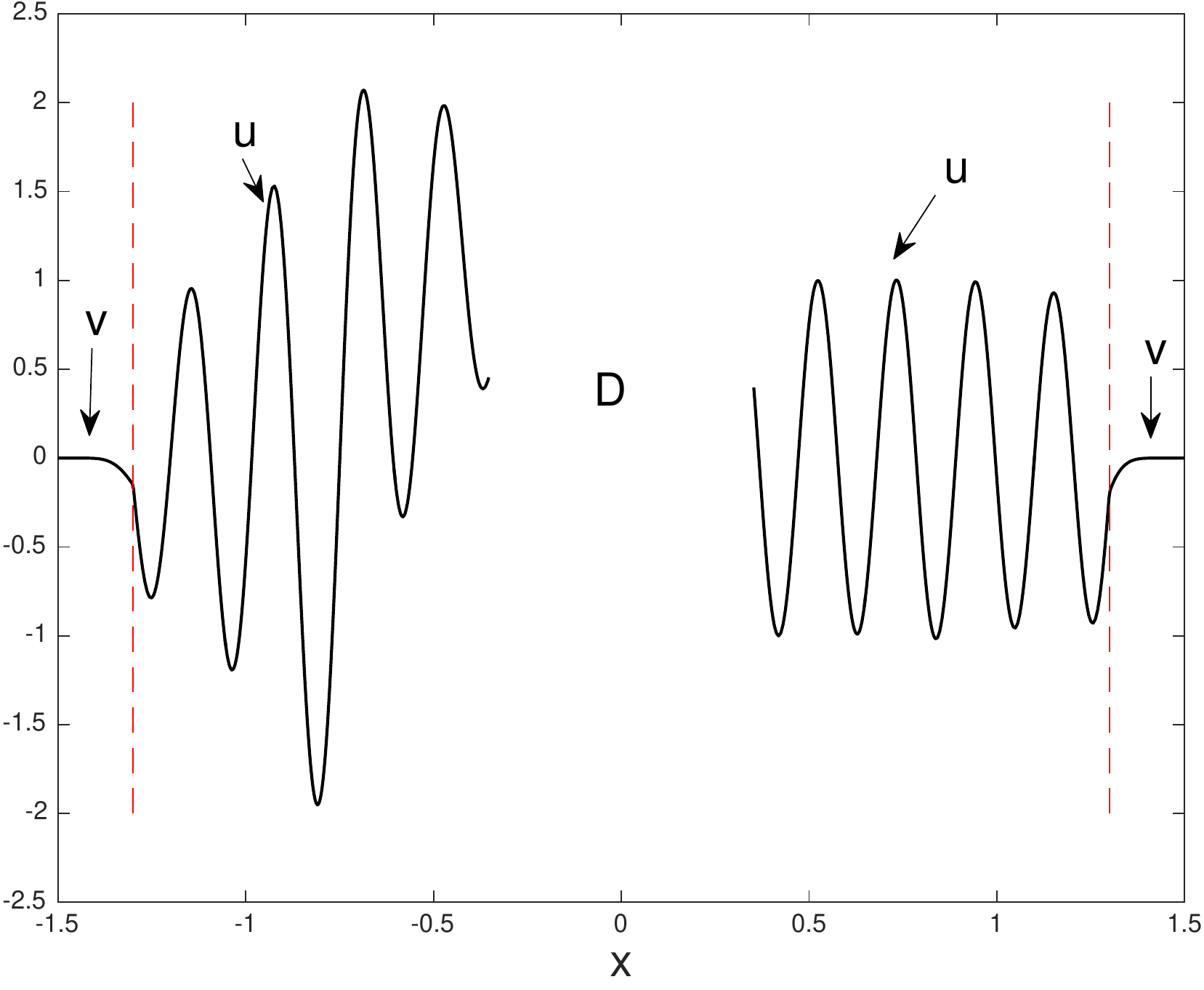}} \quad 
  \subfigure[profile along y-axis]{ \includegraphics[scale=.25]{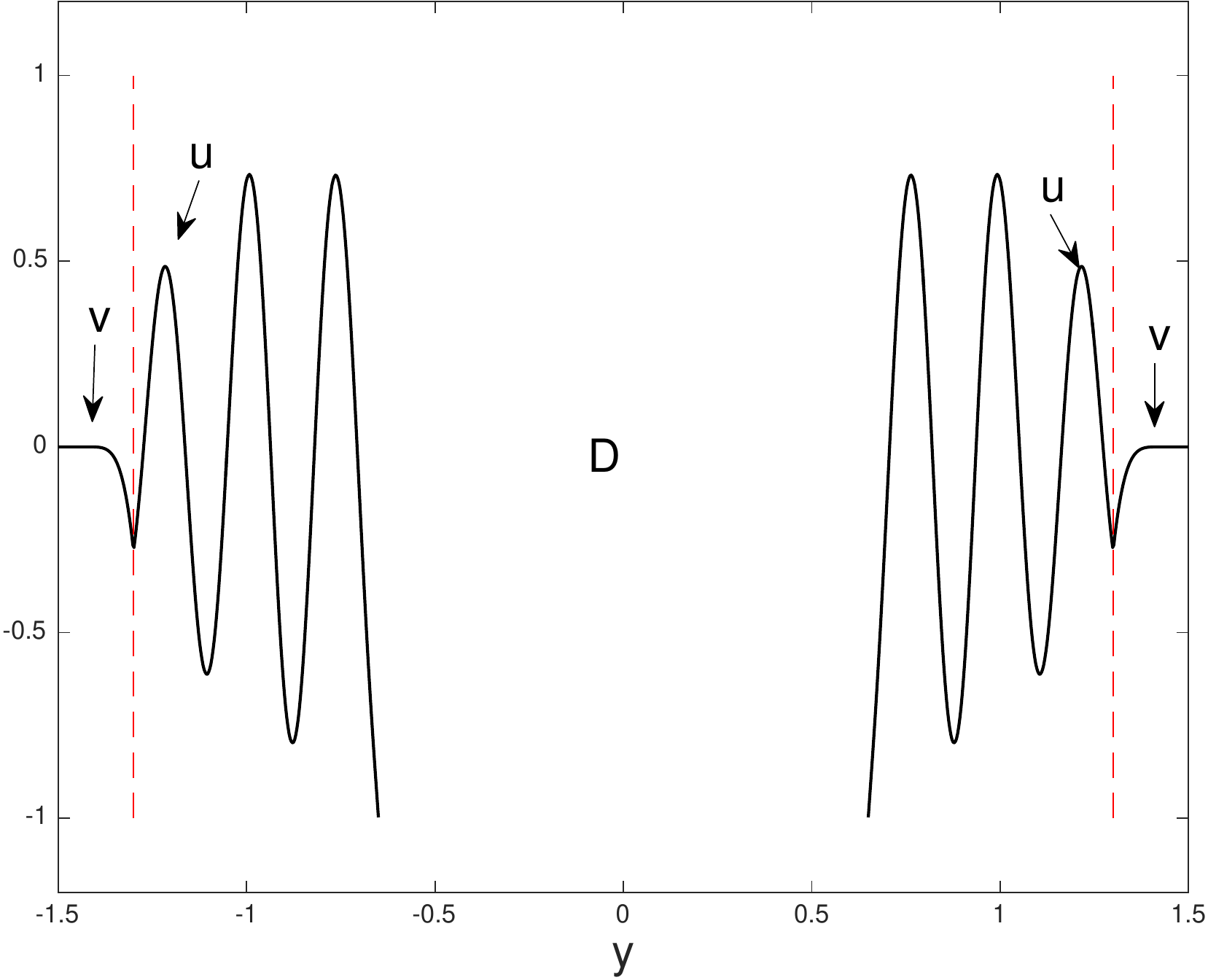}}\\
    \caption{Hexagonal star-shaped scatterer with an circular PAL layer. The simulation results are obtained with $k=30,$ $\theta_0=0,$ $\sigma_0=\sigma_1=1,$ $(R_1, R_2)=(1.3,1.55)$, $N_1=35$ and various $N$.  }
   \label{figs:PALtest1}
\end{center}
\end{figure}

Next, we consider a hexagonal star-shaped scatterer with its boundary radius parameterized by
\begin{equation}\label{eq: R0test1}
R_0(\theta)=0.5+0.15\sin(6(\theta+{\pi}/{4})), \quad \theta=[0,2\pi).
\end{equation} 
The exterior domain is surrounded with an annulus PAL layer with $(R_1, R_2)=(1.3,1.55).$ To numerically stimulate this problem, we discretize the computational domain $\Omega \cup \Omega_{\varrho}^{\rm PAL}$ with 250 non-overlapping quadrilateral elements $\Omega=\{ \Omega^{(i)}\}_{i=1}^{200} $ and $\Omega_{\varrho}^{\rm PAL}=\{ \Omega_{\varrho}^{(i)}\}_{i=1}^{50}$, as shown in Figure \ref{figs:PALtest1} (a). Once again, the spectral-element scheme  is implemented based on Theorem \ref{eq:numerBlinear}. Using the Gordon-Hall elemental  transformation $\{ T^{i}, T_{\varrho}^i\}:  [-1, 1]^2 \mapsto \{\Omega^{(i)},  \Omega_{\varrho}^{(i)}   \},$ we define the approximation space
\begin{equation}\label{finitespace}
u_{\bs N}\in V_{\bs N}=\big\{u\in H^1(\Omega):u |_{\Omega^{(i)}} \circ T^{i} \in \mathbb{P}_{N_1} \times \mathbb{P}_{N_1},\;u|_{\Omega_{\varrho}^{(i)}}=v_{\bs N}w,\; v_{\bs N}|_{\Omega_{\varrho}^{(i)}} \circ T_{\varrho}^{i} \in \mathbb{P}_{N_1} \times \mathbb{P}_{N}\big \}.
\end{equation}
We set $\theta_0=0,$ $k=30,$  and $\sigma_0=\sigma_1=1$.   $N_1=35$ is employed  in all tests, which is large enough  to guarantee  the numerical errors are mainly induced by the approximation error in the PAL layer.  Since the exact solution with irregular scatterer is not available, we adopt the numerical solution obtained by $(N_1,N)=(35,35)$ as a reference solution and the numerical errors are obtained by comparing the numerical solution with this reference solution.  In Figure \ref{figs:PALtest1} (b), we plot  $ \Re\{u_{\bs N}\}|_{\Omega}$ and $\Re\{v_{\bs N}\}|_{\Omega_{\varrho}^{\rm PAL}}$ with $(N_1,N)=(35,15)$.  We plot  the maximum errors  in $\Omega$ against $N$ in Figure \ref{figs:PALtest1} (c). 
Observe that the errors decay exponentially as  $N$ increases, and the approximation in the layer has no oscillation and is well-behaved, as shown by  the profiles along $x$-and $y$-axis of the numerical solution in Figure \ref{figs:PALtest1} (d)-(e).

\subsection{Hexagonal star-shaped layer}
\begin{figure}[t!]
\begin{center}
 \subfigure[illustration of mesh grids]{ \includegraphics[scale=.36]{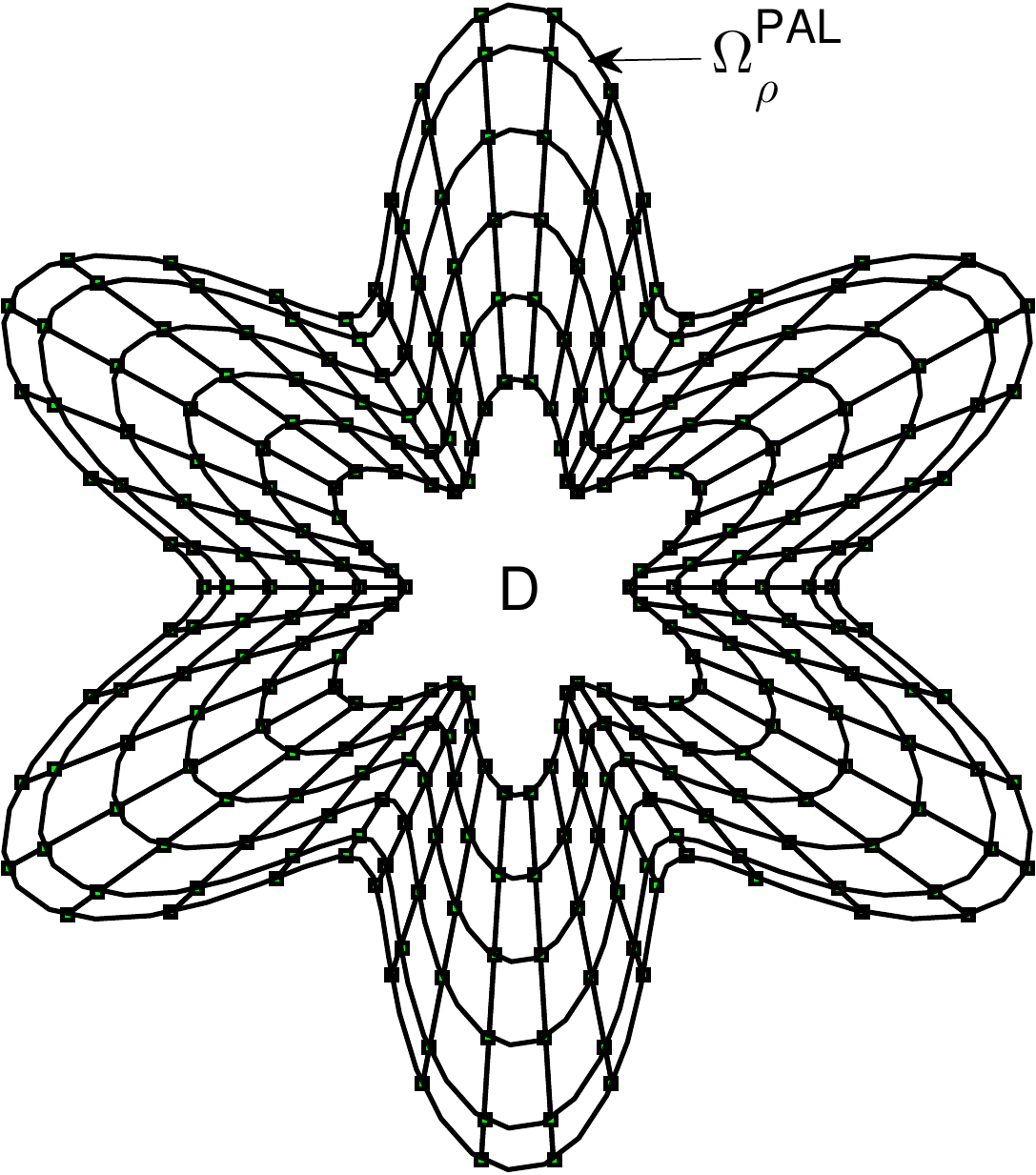}} \qquad 
  \subfigure[${\rm Re}(u)$ and ${\rm Re}(v)$]{ \includegraphics[scale=.36]{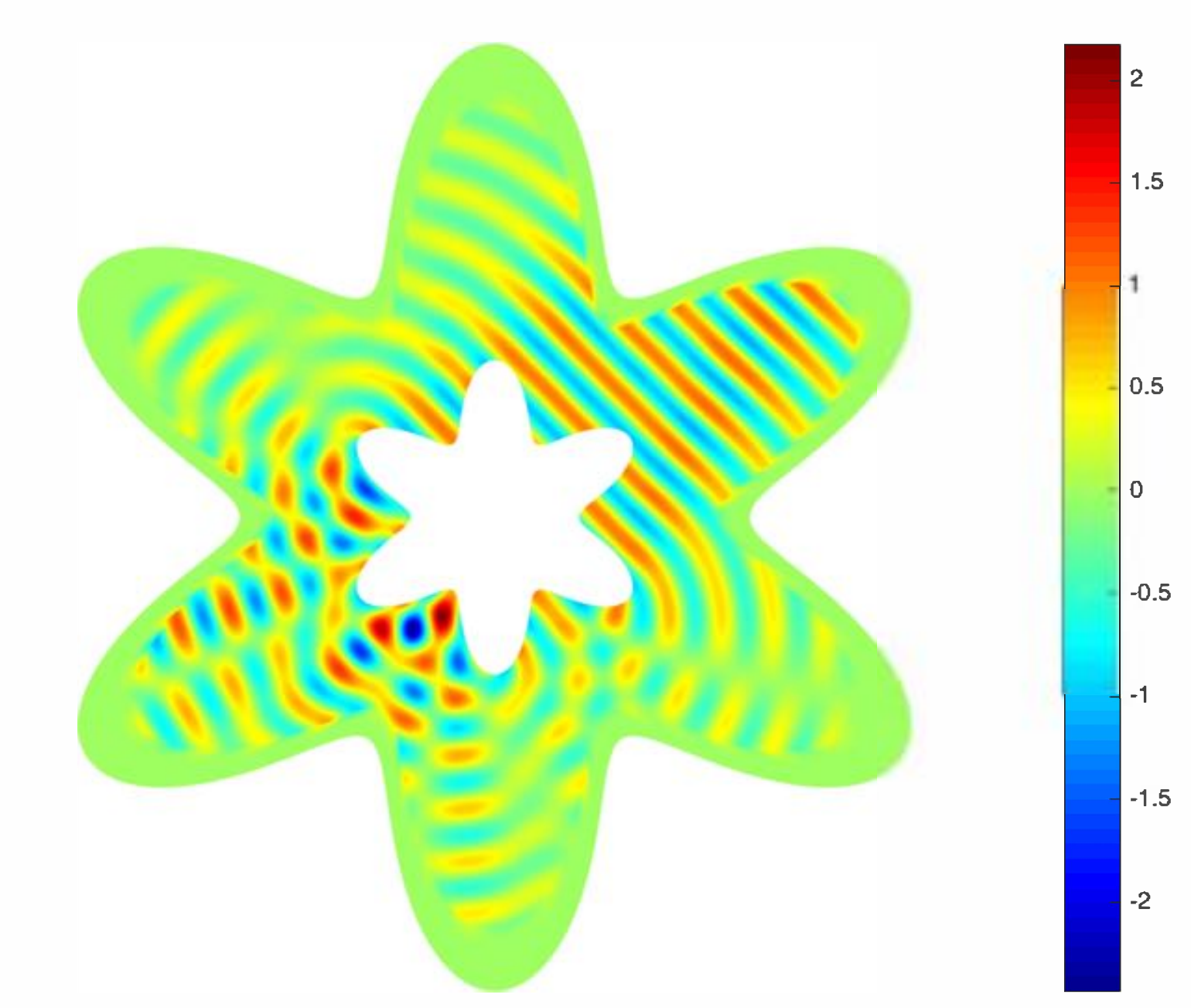}}\\
 \subfigure[errors of (b) against $N$ ]{ \includegraphics[scale=.25]{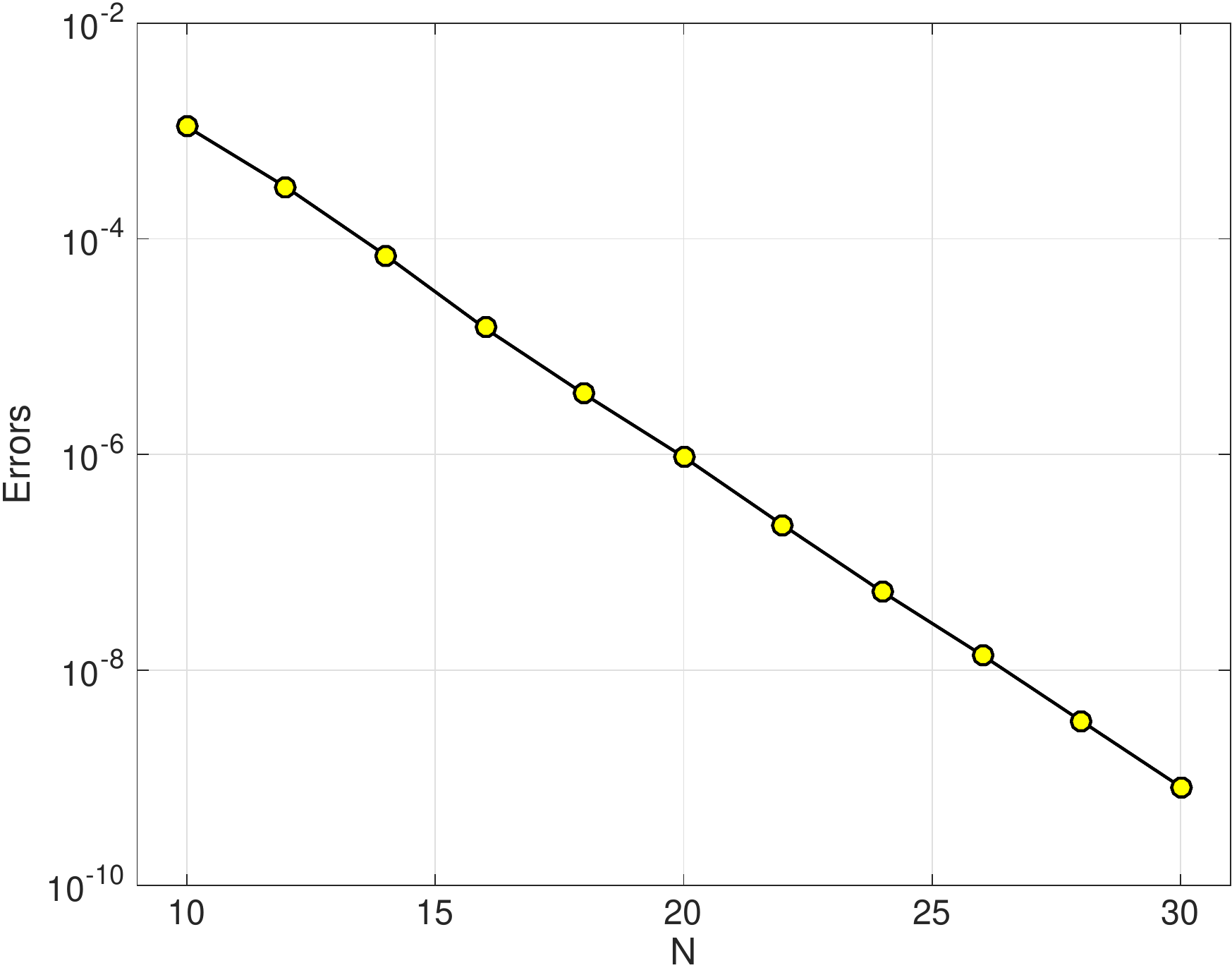}}\quad
 \subfigure[profile along x-axis]{ \includegraphics[scale=.25]{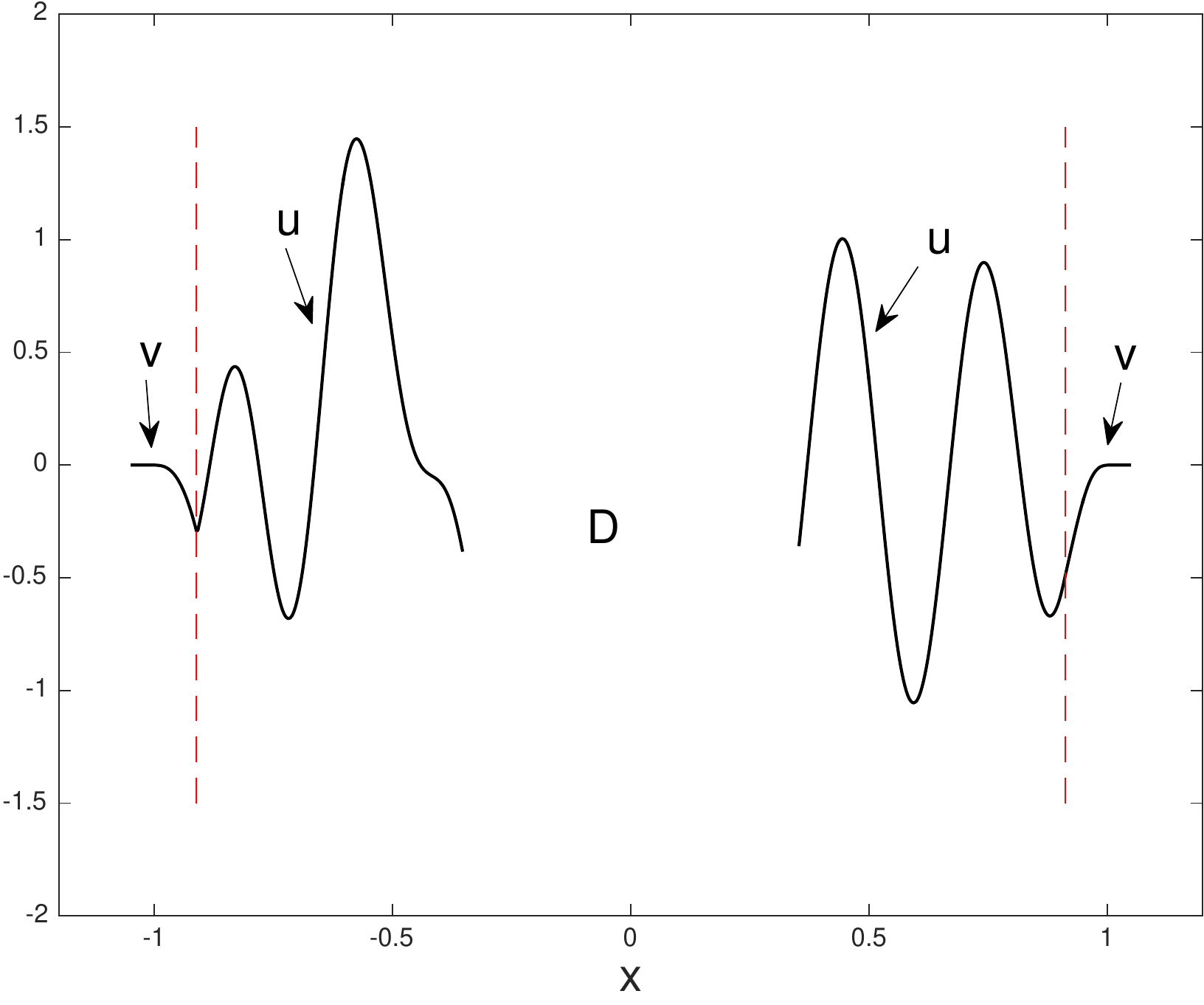}} \quad 
  \subfigure[profile along y-axis]{ \includegraphics[scale=.25]{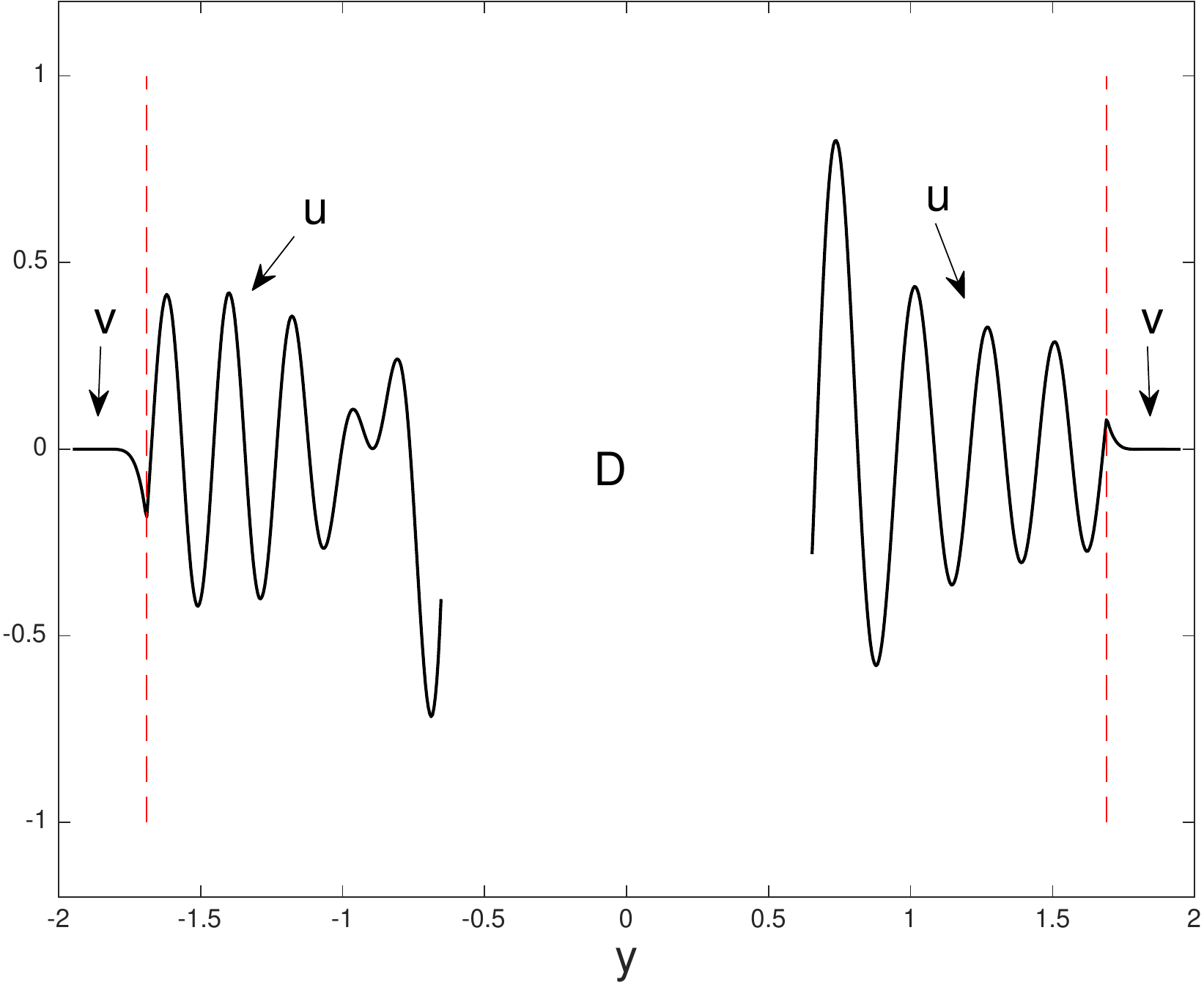}}\\
    \caption{Hexagonal star-shaped scatterer and PAL layer. The simulation results are obtained with $k=30,$ $\theta_0=\pi/4,$ $\sigma_0=\sigma_1=1,$ $R_1(\theta)=2.6R_0(\theta)$ and $R_2(\theta)=3R_0(\theta)$ with $R_0(\theta)$ defined in \eqref{eq: R0test1}, $N_1=35$ and various $N$. }
   \label{figs:PALtest2}
\end{center}
\end{figure}
Next, we surround the same scatterer in the previous example with  an hexagonal star-shaped layer, i.e., the parameterized form for the inner and outer radius for the layer are $R_1(\theta)=2.6R_0(\theta)$ and $R_2(\theta)=3R_0(\theta)$ with $R_0(\theta)$ defined in \eqref{eq: R0test1}. We set $k=30$ and the incident angle $\theta_0=\pi/4.$ We partition the computational domain into 250 quadrilateral curvilinear spectral elements, as illustrated in Figure \ref{figs:PALtest2}. We depict  $\Re\{u_{\bs N}\}|_{\Omega}$ and $ \Re\{v_{\bs N}\}|_{\Omega_{\varrho}^{\rm PAL}}$ with $(N_1,N)=(35,15)$ in Figure \ref{figs:PALtest2} (b). The maximum error in $\Omega$ against $N$ is shown in Figure \ref{figs:PALtest2} (c). It is evident that the errors decrease exponentially with increased polynomial degree $N$, displaying an exponential convergence rate.  Figure \ref{figs:PALtest2} (d)-(e) show the profiles of the numerical solution in Figure \ref{figs:PALtest2}(b) along $x$-and $y$-axis. We observe that the   solution profile in the PAL layer smoothly decreases to zero without any oscillation. 

\subsection{Elliptical layer} 
\begin{figure}[t!]
\begin{center}
 \subfigure[illustration of mesh grids]{ \includegraphics[scale=.28]{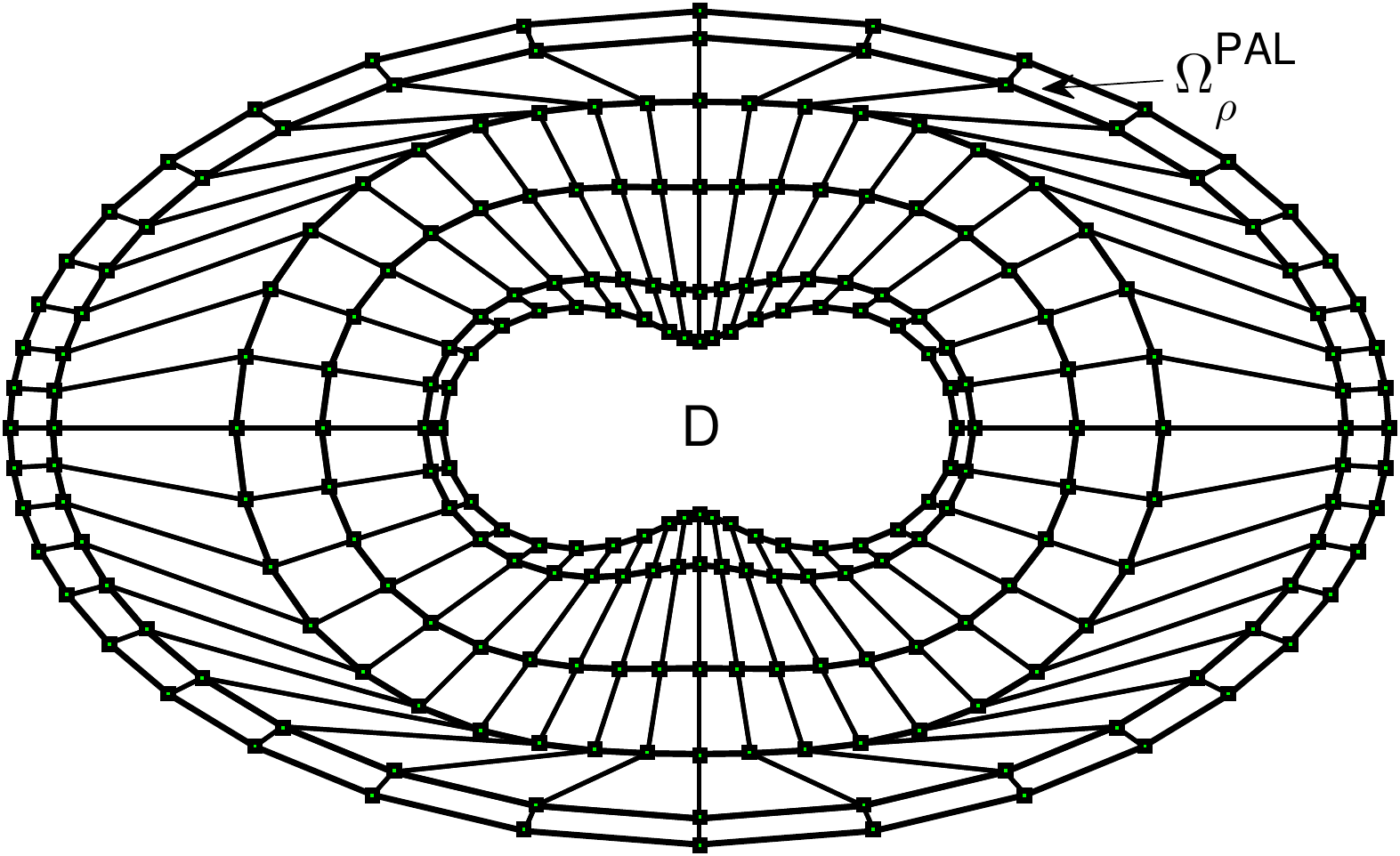}} \qquad 
  \subfigure[${\rm Re}(u)$ and ${\rm Re}(v)$]{ \includegraphics[scale=.32]{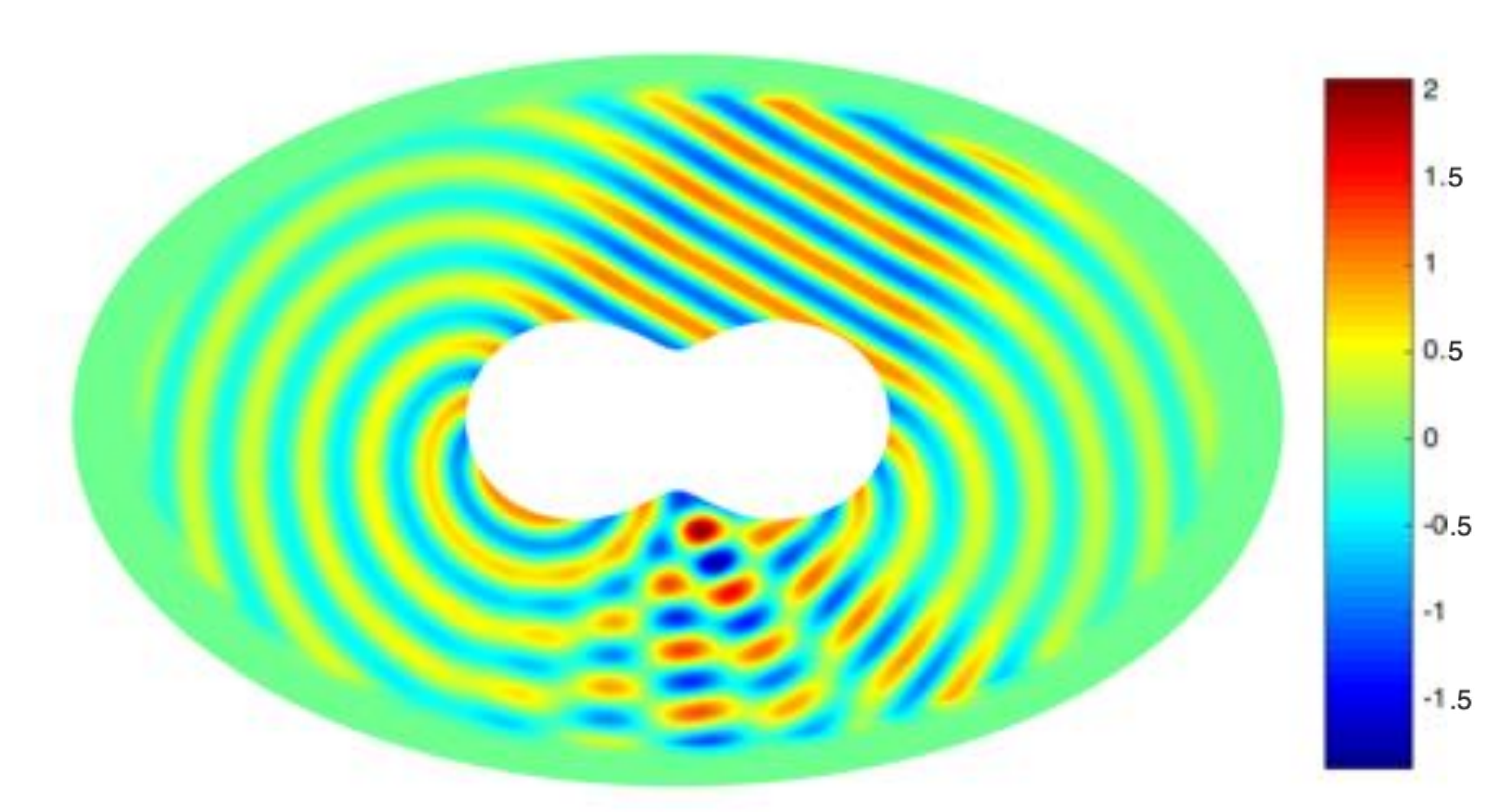}}\\
 \subfigure[errors of (b) against $N$  ]{ \includegraphics[scale=.23]{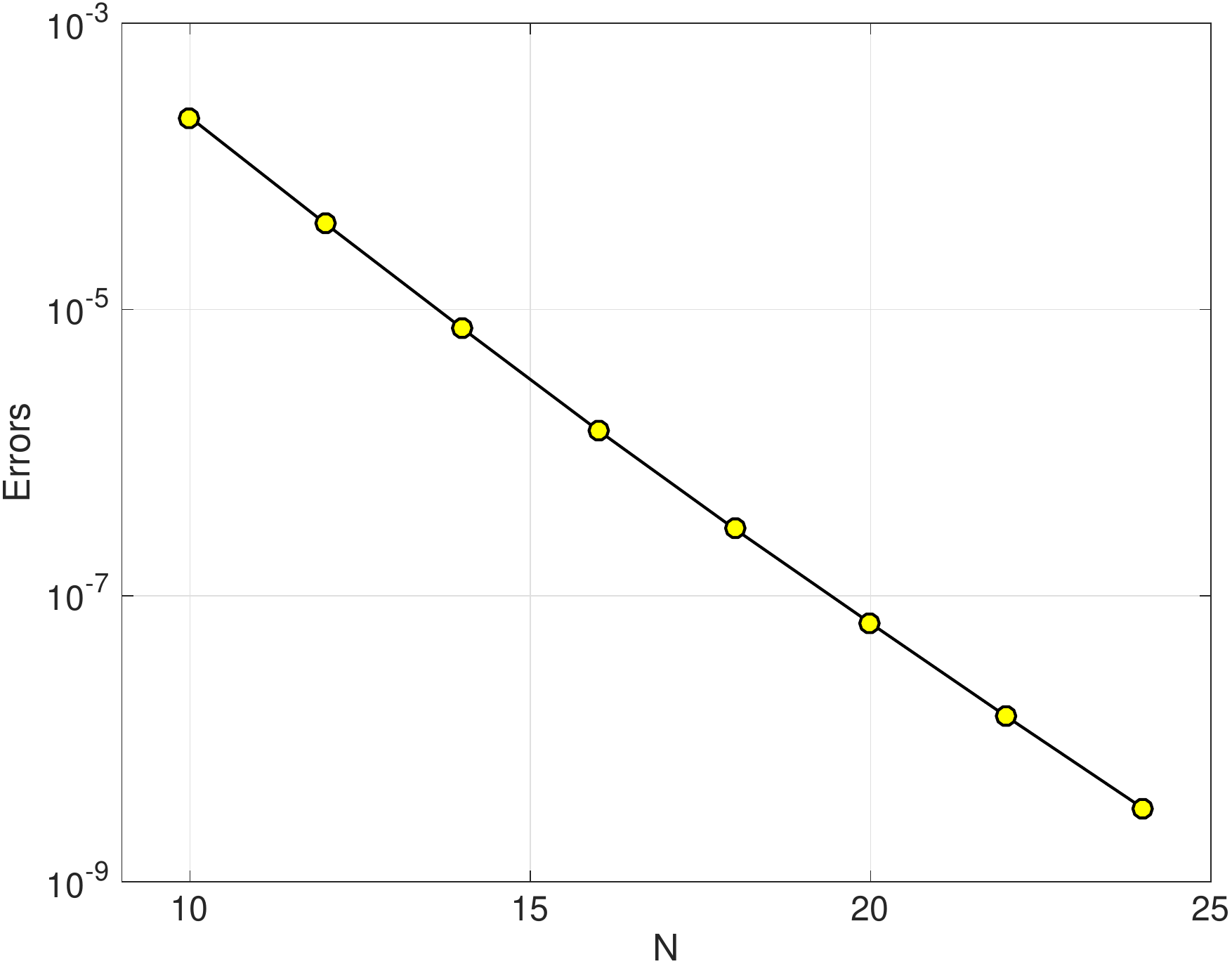}}\quad
 \subfigure[profile along x-axis]{ \includegraphics[scale=.23]{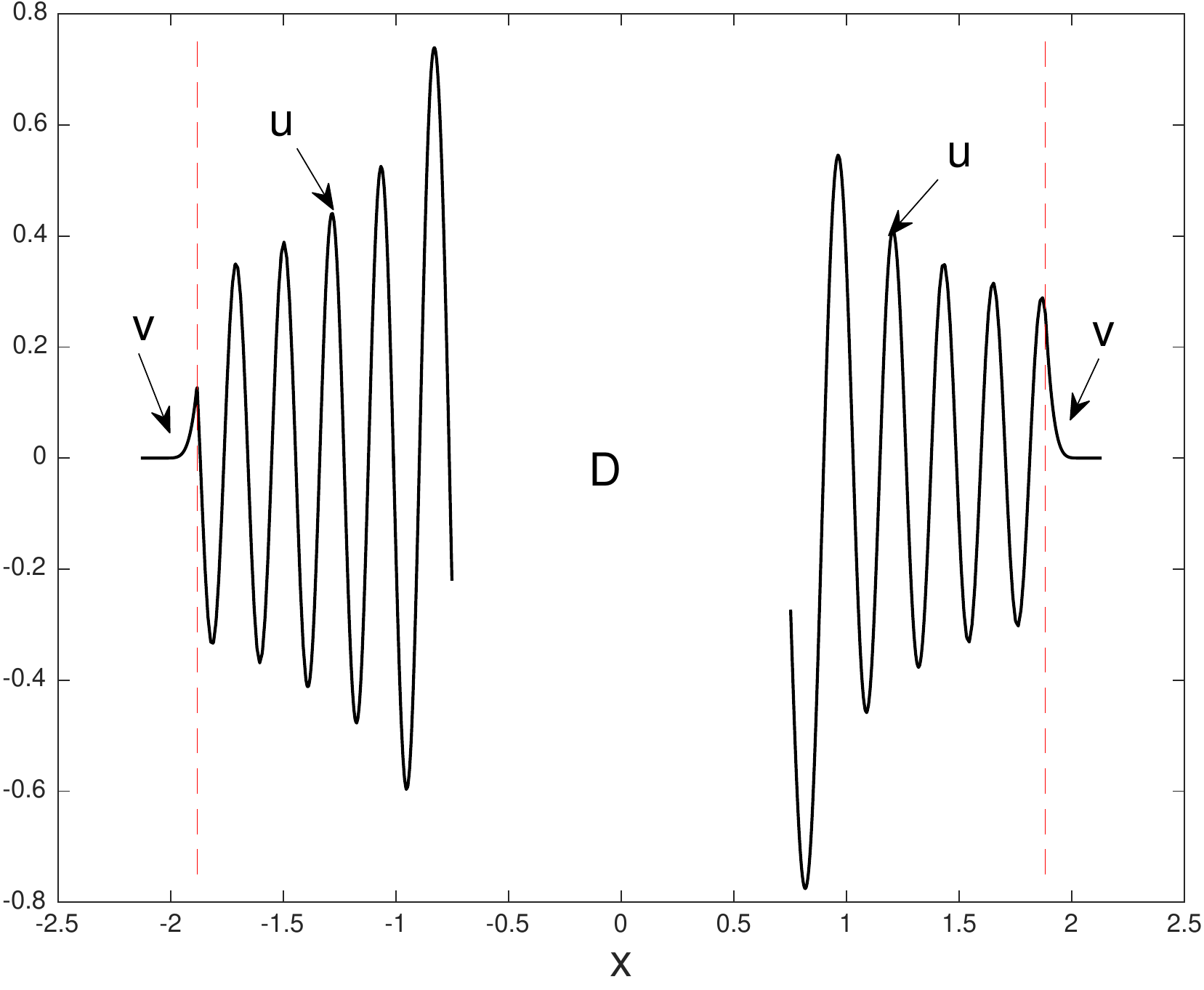}} \quad 
  \subfigure[profile along y-axis]{ \includegraphics[scale=.23]{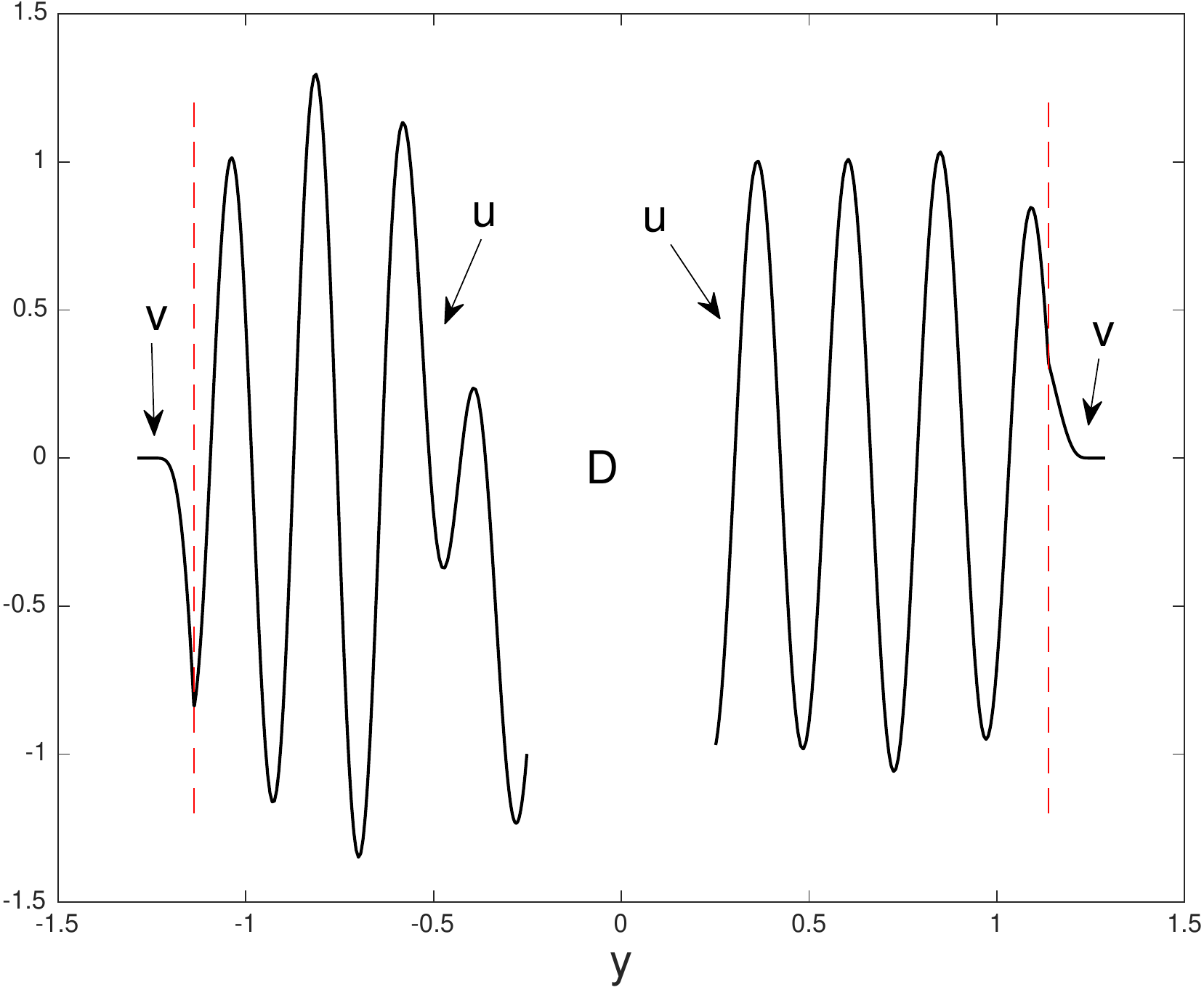}}\\
    \caption{Peanut-shaped scatterer with elliptical PAL layer. The simulation results are obtained with $k=30,$ $\theta_0=\pi/3,$ $\sigma_0=\sigma_1=1,$ $R_0(\theta)$ and $R_1(\theta)$ are computed based on \eqref{eq: R0test3} and \eqref{ellitpitic} and $R_2(\theta)=\frac{17}{15}R_1(\theta)$, $N_1=35$ and various $N$.}
   \label{figs:PALtest3}
\end{center}
\end{figure}
The geometries of the PAL layer and the scatterer can be rather general.  Consider a peanut-shaped scatterer with its boundary radius parameterized by
\begin{equation}\label{eq: R0test3}
R_0(\theta)=0.5+0.25\sin(2(\theta+{\pi}/{4})), \quad \theta=[0,2\pi).
\end{equation} 
The exterior domain is truncated with an elliptical PAL layer with $\Gamma_{\!R_1}$ to be an ellipse: $\frac{x^2}{a^2}+\frac{y^2}{b^2}=1$ with $a>b>0,$  and $a$ and $b$ take the form
\begin{equation*}
a=1.5\cosh(0.7),\quad b=1.5\sinh(0.7).
\end{equation*} 
Then $R_1(\theta)$ can be computed based on \eqref{ellitpitic} and we let $R_2(\theta)=\frac{17}{15}R_1(\theta)$. We set $k=30$ and $\theta_0=\pi/3.$ The numerical scheme are implemented based on Theorem \ref{thm: implem}. The computational domain is partitioned by 200 spectral elements shown in Figure \ref{figs:PALtest3} (a). $\Re\{u_{\bs N}\}|_{\Omega}$ and $\Re\{v_{\bs N}\}|_{\Omega_{\varrho}^{\rm PAL}}$ obtained with $(N_1,N)=(35,15)$ are plotted in Figure \ref{figs:PALtest3}. Maximum error against $N$ and the profiles of the numerical solution along $x$-and $y$-axis are depicted, respectively in Figure \ref{figs:PALtest3} (c)-(f). Due to the well-behaved and non-oscillatory nature of the solution in the PAL layer, the error history exhibits an exponential convergence rate.

\subsection{Rectangular layer}Next, we surround the scatterer in the previous example with a rectangular PAL layer with the boundary $\Gamma_{\!R_1}$ as a square with four vertices:
$(a,b), (-a,b), (-a,-b), $  $ (a,-b) $ with  $ a=1.5$ and $b=0.75.$ Thus, $R_1(\theta)$ can be computed by \eqref{rect-Omega} and we let $R_2(\theta)=\frac{17}{15}R_1(\theta).$ In this simulation, the wavenumber and incident angle are set to be $k=30$ and $\theta_0=\pi/3$. The domain of interest is discretized into 200 spectral elements, as depicted in Figure \ref{figs:PALtest4} (a). We plot $\Re\{u_{\bs N}\}|_{\Omega}$ and $\Re\{v_{\bs N}\}|_{\Omega_{\varrho}^{\rm PAL}}$ obtained with $(N_1,N)=(35,15),$ the maximum error compared with the reference solution obtained with $N_1=N=35$, and the profiles of the numerical solution along $x$- and $y$-axis in Figure \ref{figs:PALtest4} (b)-(f). And it can be observed that the error decreases exponentially as $N$ increases and the solution in the rectangular PAL layer are well-behaved and smoothly decreases to zero with non-oscillatory profiles.  

It is also possible to simulate the exterior scattering problem with locally inhomogeneous medium. All the numerical settings are the same except the refraction index $n(\bs x)$ in \eqref{thm: implem} in $\Omega$ are replaced by a shifted Guassian function 
\begin{equation}\label{eq:inhommedium}
n(\bs x)=1+c_0 \exp\Big(-\frac{(x-x_0)^2+(y-y_0)^2}{2c_1^2} \Big).
\end{equation}
The inhomogeneous refraction index is depicted in Figure \ref{figs:PALtest5} (a). Similarly, we demonstrate the real part of the solution in Figure \ref{figs:PALtest5} (b). Observe that compared with Figure \ref{figs:PALtest4} (b), the oscillation of the solution field above the upper-middle region of the peanut scatterer increases, due to the influence of the inhomogeneity therein. 
The maximum error and the profiles of the numerical solution along $x$-and $y$- axis are depicted in Figure \ref{figs:PALtest5} (c)-(f), respectively. We conclude that the proposed PAL technique is accurate and robust for various scatterers and with locally inhomogeneous medium.

\begin{figure}[htbp]
\begin{center}
 \subfigure[illustration of mesh grids]{ \includegraphics[scale=.32]{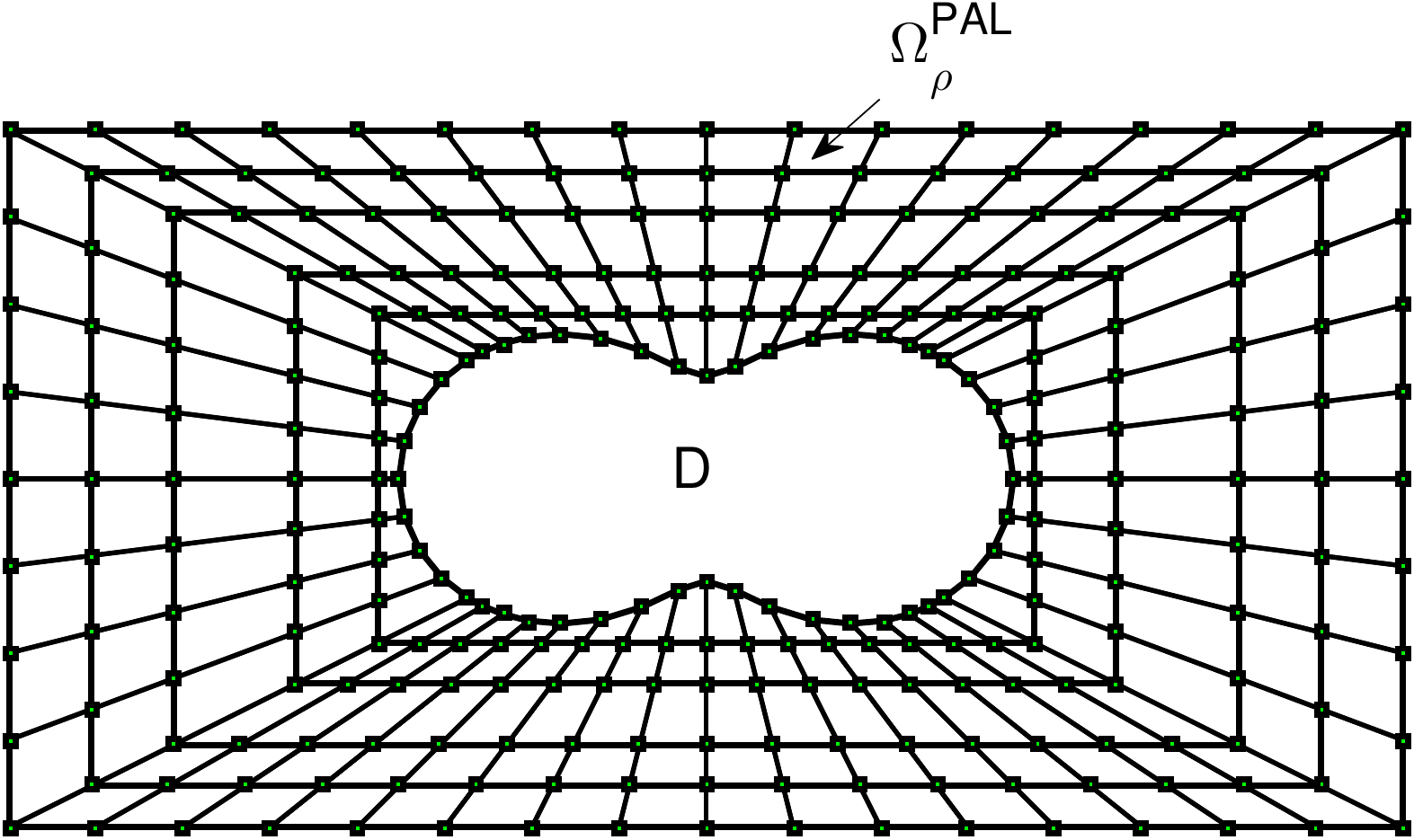}} \qquad 
  \subfigure[${\rm Re}(u)$ and ${\rm Re}(v)$]{ \includegraphics[scale=.35]{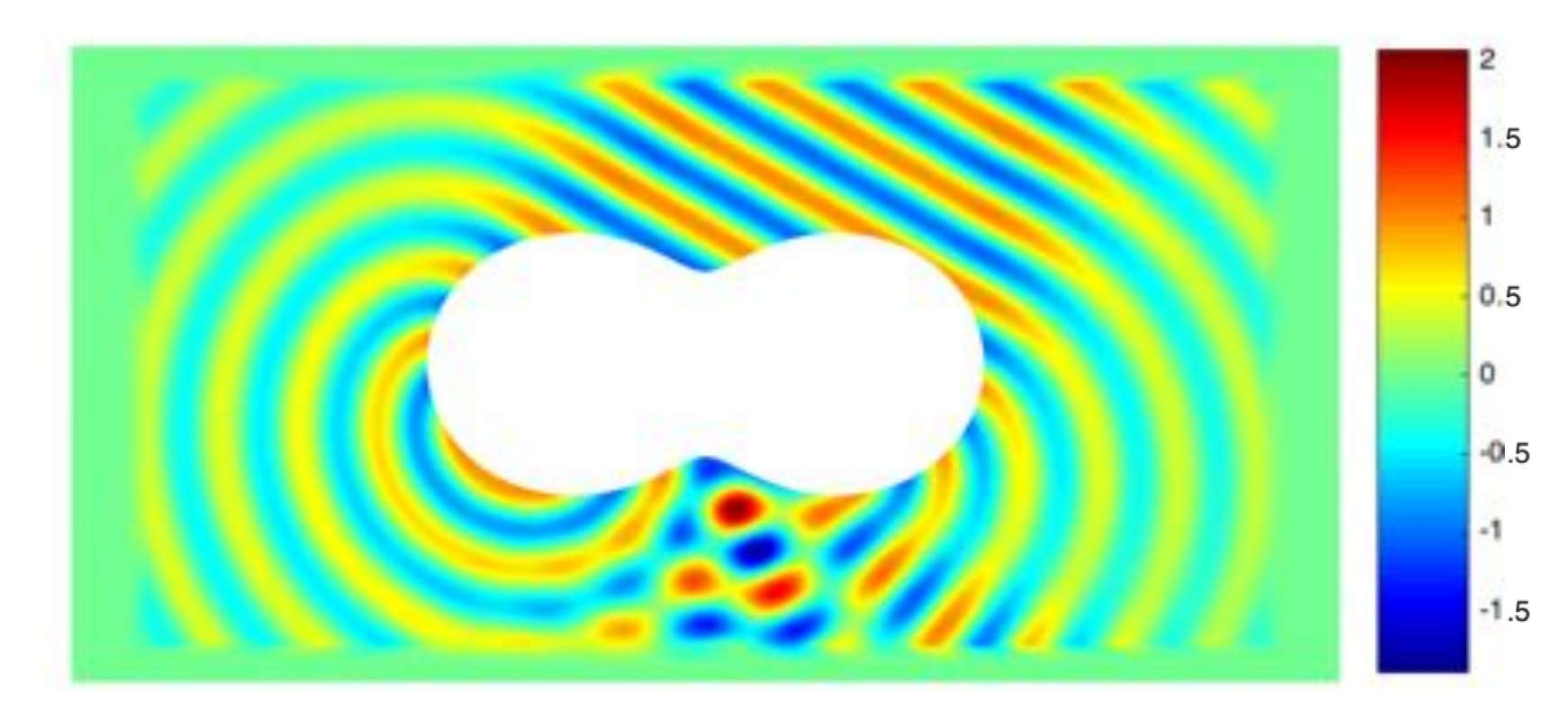}}\\
 \subfigure[ errors of (b) against $N$  ]{ \includegraphics[scale=.25]{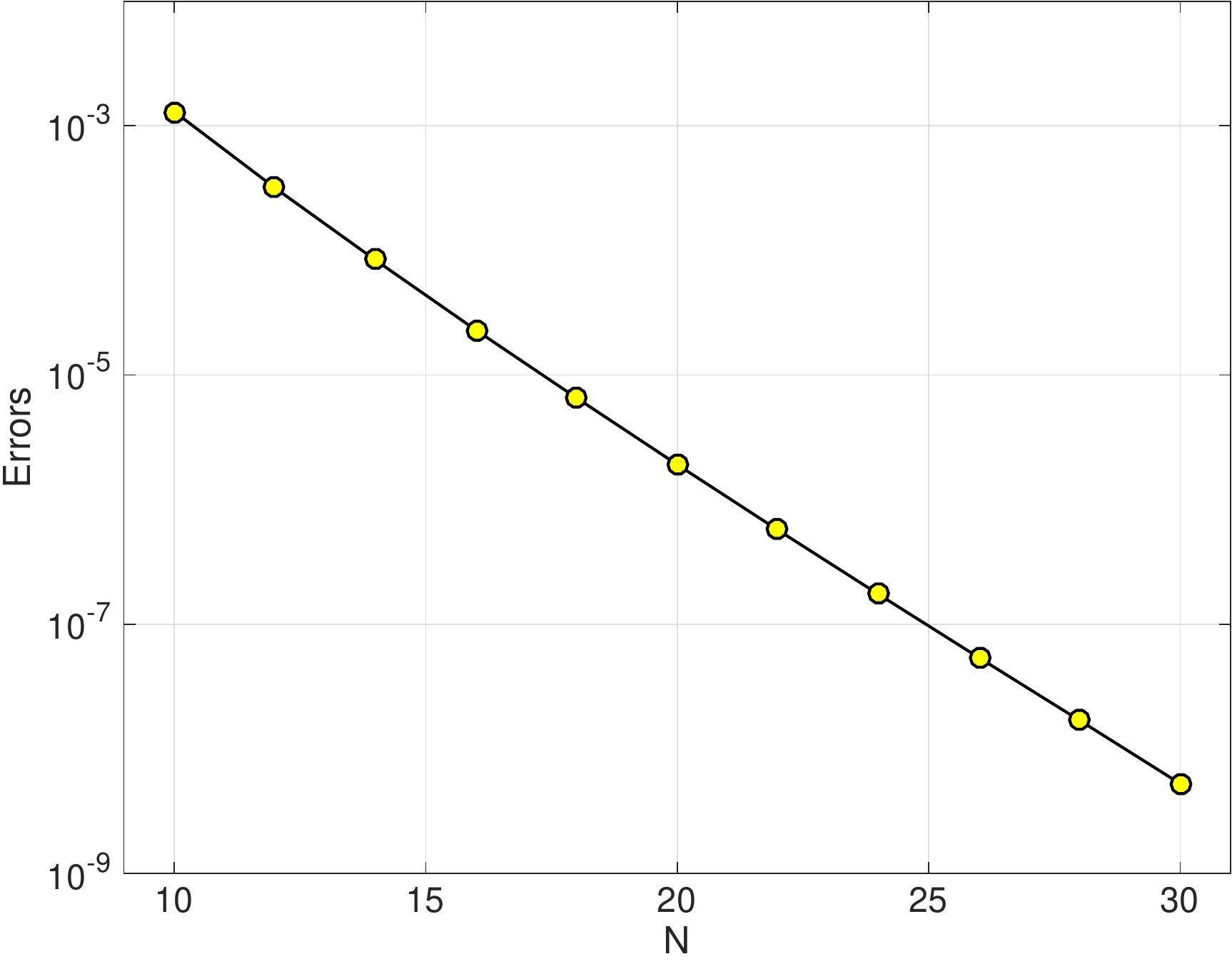}}\quad
 \subfigure[profile along x-axis]{ \includegraphics[scale=.25]{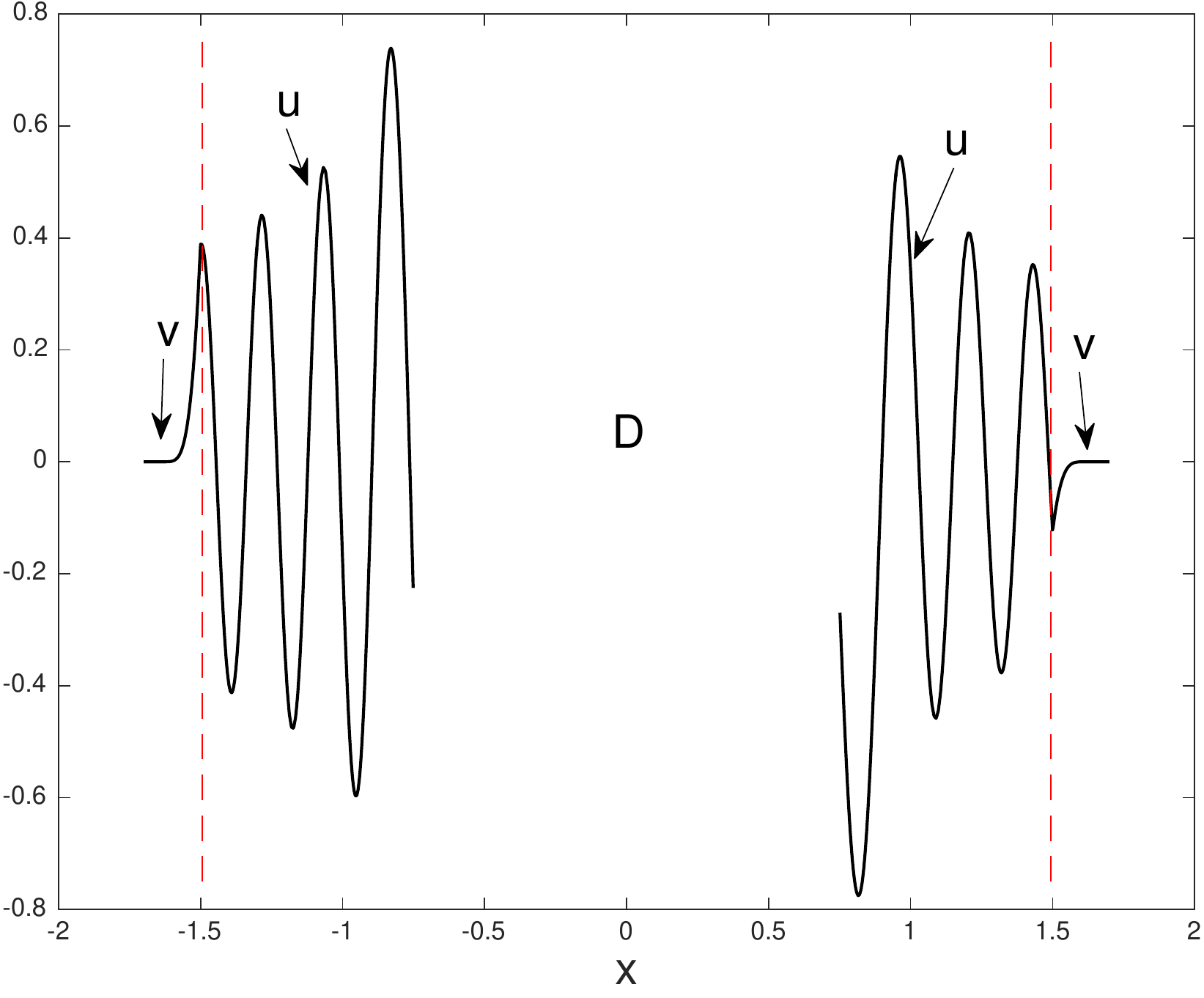}} \quad 
  \subfigure[profile along y-axis]{ \includegraphics[scale=.25]{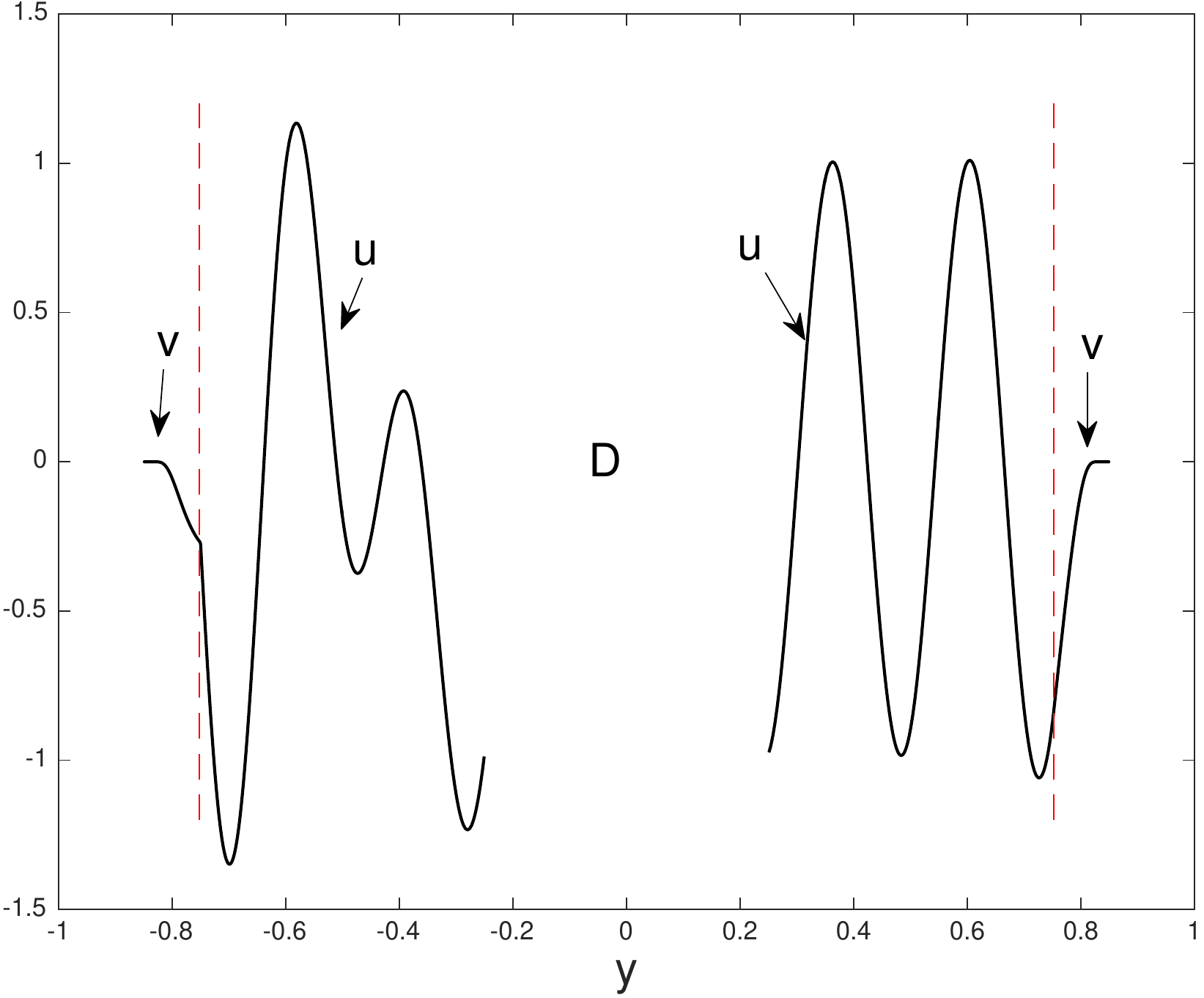}}\\
    \caption{Peanut-shaped scatterer with a rectangluar PAL layer. The simulation results are obtained with $k=30,$ $\theta_0=\pi/3,$ $\sigma_0=\sigma_1=1,$ $R_0(\theta)$ and $R_1(\theta)$ are computed based on \eqref{eq: R0test3} and \eqref{rect-Omega} and $R_2(\theta)=\frac{17}{15}R_1(\theta)$, $N_1=35$ and various $N$.}
   \label{figs:PALtest4}
\end{center}
\end{figure}

\begin{figure}[htbp]
\begin{center}
 \subfigure[refraction index $n(\bs x)$]{ \includegraphics[scale=.36]{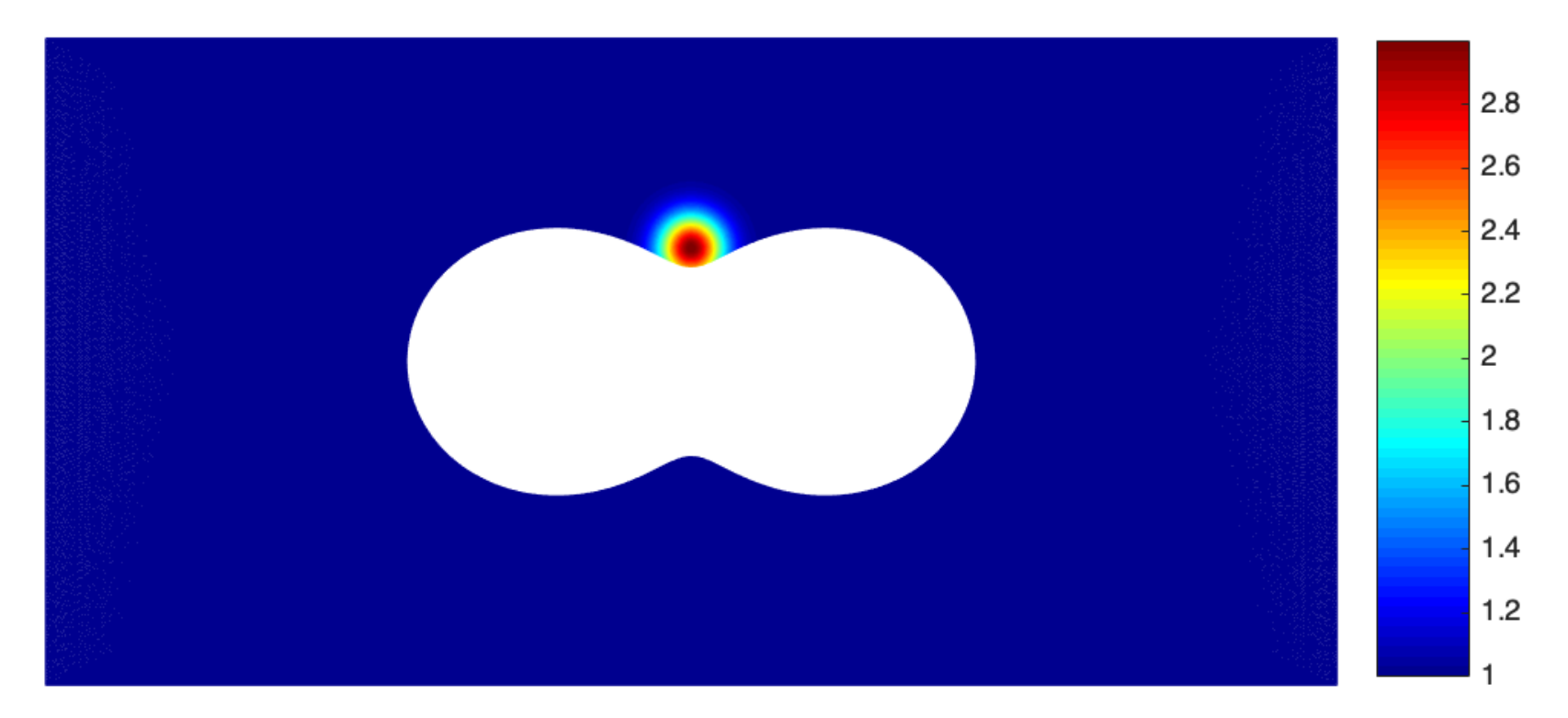}} \qquad 
  \subfigure[${\rm Re}(u)$ and ${\rm Re}(v)$]{ \includegraphics[scale=.36]{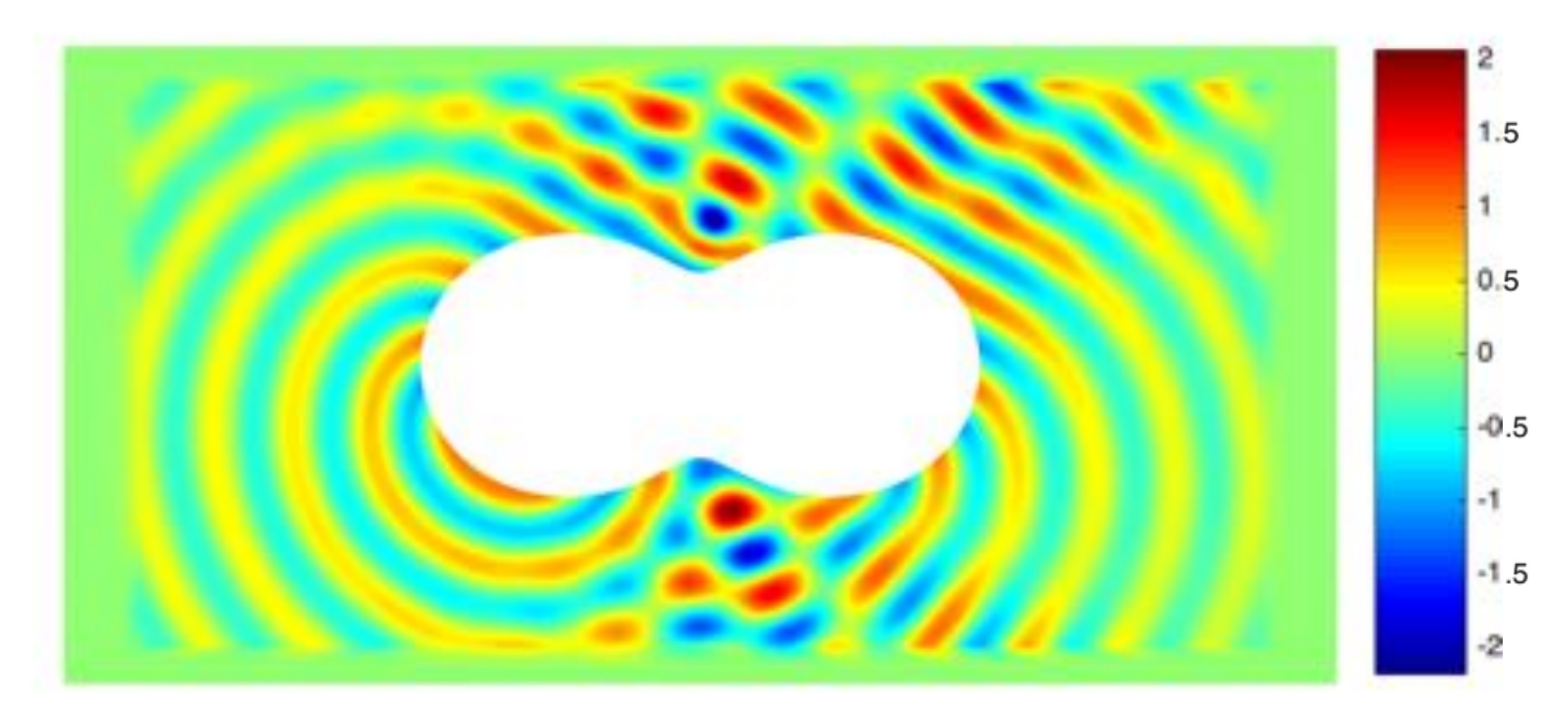}}\\
 \subfigure[errors of (b) against $N$  ]{ \includegraphics[scale=.25]{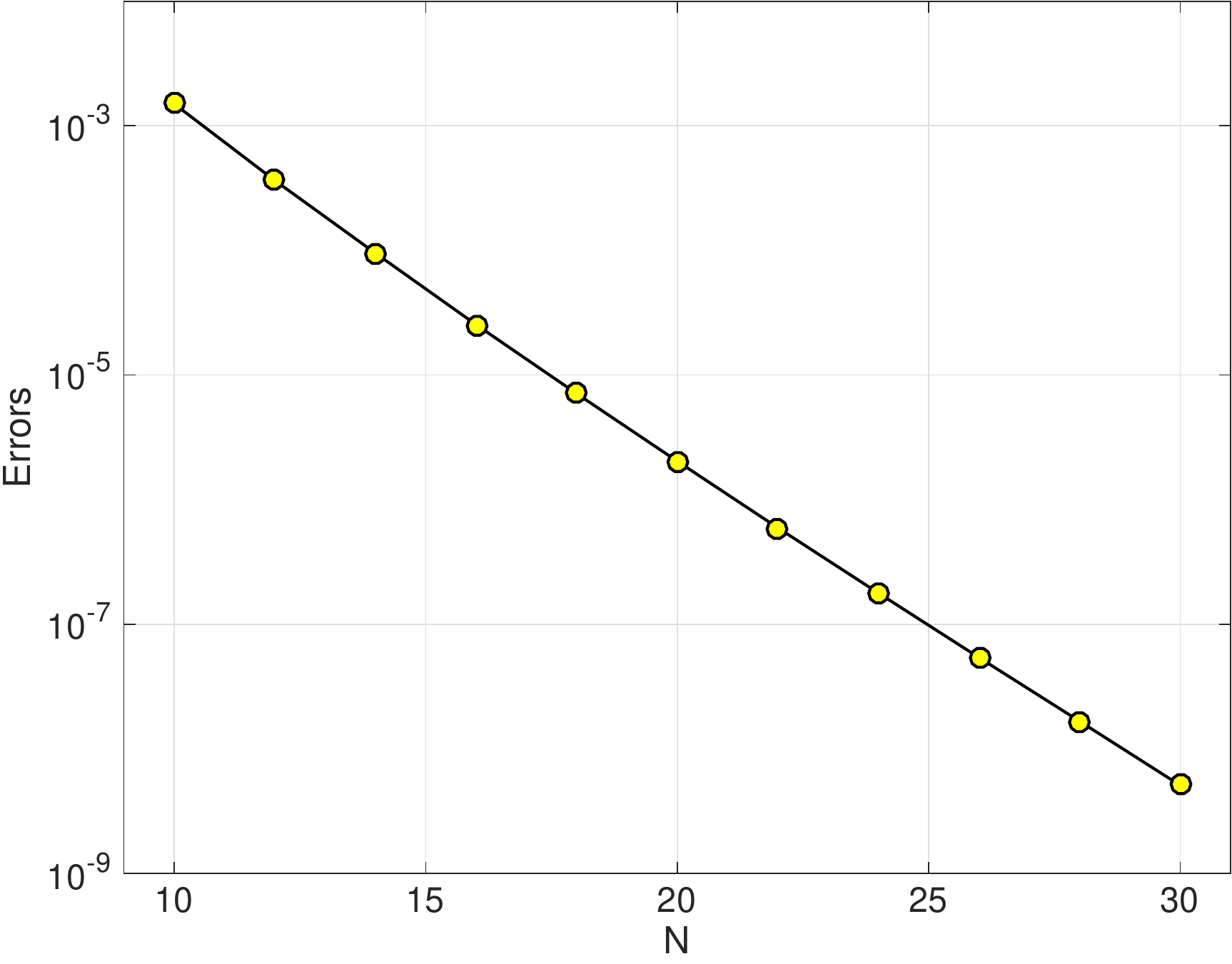}}\quad
 \subfigure[profile along x-axis]{ \includegraphics[scale=.25]{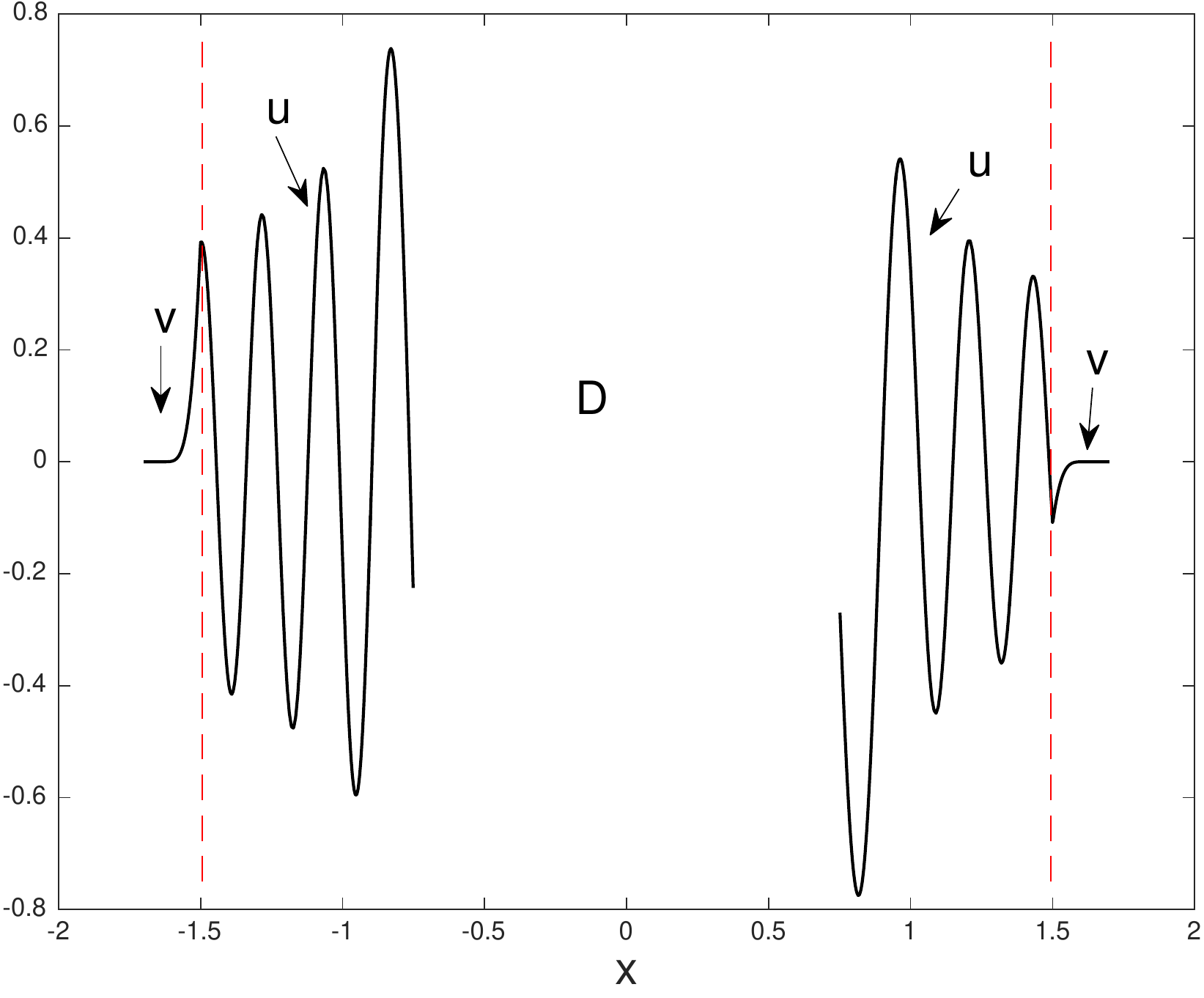}} \quad 
  \subfigure[profile along y-axis]{ \includegraphics[scale=.25]{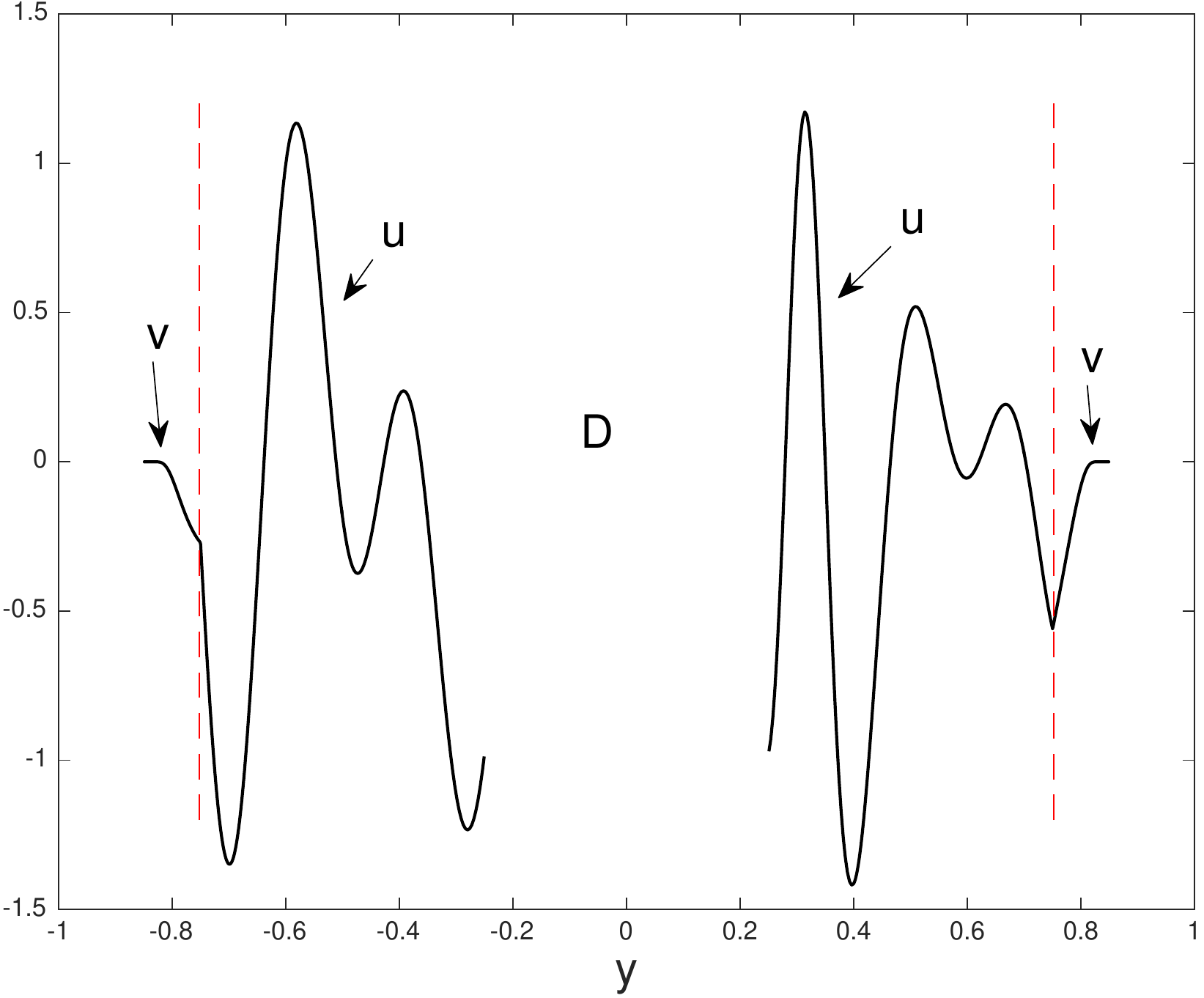}}\\
    \caption{Scattering problem with a locally inhomogeneous medium with a rectangular PAL. The simulation results are obtained with $k=30,$ $\theta_0=\pi/3,$ $\sigma_0=\sigma_1=1,$ $R_0(\theta)$ and $R_1(\theta)$ are computed based on \eqref{eq: R0test3} and \eqref{rect-Omega} and $R_2(\theta)=\frac{17}{15}R_1(\theta)$, $N_1=35$ and various $N$.}
   \label{figs:PALtest5}
\end{center}
\end{figure}

\vskip 10pt

\begin{appendix}

\section{Proof of Theorem \ref{solu-error} }\label{AppendixA0}
\renewcommand{\theequation}{A.\arabic{equation}}

For clarity,  we first consider  \eqref{extproblem} with $g(y)= \hat g_l \sin(ly),$ and then apply the principle of superposition to obtain \eqref{erroreqnUUp}.  Note that  by \eqref{uxyexact},    the exact solution of \eqref{extproblem}  is
\begin{equation}\label{exactU21}
U_l(x,y)=  \hat g_l\, e^{\ri \hat k_{l} x}\sin(l y),
\end{equation}
and the PML-solution of   \eqref{PMLeq1}-\eqref{tranmissionGuide}  is
\begin{equation}\label{exactUp23}
U_{{\rm p},l}(x,y)= \frac{e^{\ri \hat k_{l} (S_d-S)}- e^{-\ri \hat k_{l} (S_d-S)}}{e^{\ri \hat k_{l}  S_d}- e^{-\ri \hat k_{l} S_d}}\,\hat g_l \sin(l y).
\end{equation}
In fact, one  verifies directly that  \eqref{exactUp23} satisfies the PML-equation    \eqref{PMLeq1} and all conditions  in \eqref{Upeqna}-\eqref{tranmissionGuide}.  Since  $S(x)=x$ for $x\in (0,L),$ we find
\begin{equation}\label{exactUp23U}
U_{{\rm p},l}(x,y)= \frac{e^{\ri \hat k_{l} (S_d-2x)}- e^{-\ri \hat k_{l} S_d}}{e^{\ri \hat k_{l}  S_d}- e^{-\ri \hat k_{l} S_d}} \,U_l(x,y)=(1-R_l(x))\,U_l(x,y), \;\;\; \forall\, (x,y)\in \Omega,
\end{equation}
where  the representation of $R_l(x)$ in \eqref{erroreqnUUp} can be obtained  straightforwardly.    Thanks to the identity  \eqref{exactUp23U},  we derive from  the principle of superposition, \eqref{uxyexact} and \eqref{exactU21} that the PML-solution in $\Omega$ is given by
\begin{equation}\label{supersolu}
U_{\rm p}(x,y)=\sum_{l=1}^\infty U_{{\rm p},l}(x,y)=U(x,y)- \sum_{l=1}^\infty \hat g_l\, R_l(x)\, e^{\ri  \hat k_l x}\sin (ly),
\end{equation}
which yields  \eqref{erroreqnUUp}.

It remains to derive the bounds in \eqref{URerrors}.   One verifies readily  that  for $z=\alpha+\beta\,\ri$ with $\alpha, \beta\in {\mathbb R},$
\begin{equation}\label{funda-inequ}
\begin{split}
& |1-e^{\alpha}|\le |1-e^z|\le 1+e^{\alpha},\quad  |\sin z|=\Big|\frac{e^{\ri z}-e^{-\ri z}} {2\ri}   \Big| =\frac{e^{-\beta}} 2 |1-e^{-2\ri z} |.
\end{split}
\end{equation}
Then we have
\begin{equation}\label{eimk}
|1-e^{2 {\rm Im}\{\hat k_l S_d\} }|\le  |{1-e^{- 2\ri \hat k_{l} S_d}}|\le 1+e^{2 {\rm Im}\{\hat k_l S_d\} }.
\end{equation}
Therefore,   (i) for $k>l $ (note $\hat k_l=\sqrt{k^2-l^2}$), we obtain \eqref{URerrors} from \eqref{funda-inequ}-\eqref{eimk} immediately.

On the other hand,   (ii) for $k<l $ (note  $\hat k_l=\ri |\hat k_l|$), the lower and upper  bounds in  \eqref{URerrors2} are a direct consequence of \eqref{eimk}.

 If $k$ is a positive integer,  we find that for the mode $l=k,$  \eqref{exactU21} is still valid (i.e., $U_l=\hat g_l \sin (ly)$), while \eqref{exactUp23} becomes
\begin{equation}\label{newUp}
U_{{\rm p},l}(x,y)=\Big(1-\frac{S(x)} {S_d}\Big)\, \hat g_l\, \sin (ly).
\end{equation}
Thus, for the mode $l=k,$  $R_{l}(x)$ should be replaced by $R_l(x)=-S(x)/S_d$ in the identity \eqref{erroreqnUUp}.


\section{Proof of Theorem \ref{thm:PAL-eqn} }\label{AppendixB}
\renewcommand{\theequation}{B.\arabic{equation}}

\begin{proof}  Without loss of generality, we start with a  general nonsingular Cartesian coordinate transformation: $\tilde x=X(x,y),\; \tilde y=Y(x,y),$  with  Jacobian and Jacobian matrix  given by
\begin{equation}\label{XY2xy}
 \bs J=\frac{\partial (x,y)}{\partial(\tilde x,\tilde y)}=\frac{1} {{\rm det}(\bs J)}
\begin{pmatrix}
Y_y & -X_y  \\[1pt]
-Y_x & X_x  \\[1pt]
\end{pmatrix},\quad {\rm det}(\bs J)=X_xY_y-X_yY_x\not=0.
\end{equation}
 It is known from the standard text book 
  that \eqref{helmhotlz}  can be  transformed into
\begin{equation}\label{Huform}
\mathcal H[U]=\frac 1 n \big\{\nabla\cdot (\bs C\, \nabla U)+k^2 n\, U\big\},
\end{equation}
where $U(x,y)=\tilde U(\tilde x,\tilde y)$ and
\begin{equation}\label{p3frame0}
{\bs C}=
\begin{pmatrix}
C_{11}& C_{12}\\
C_{12}& C_{22}
\end{pmatrix}
=\frac{{\bs J}\,\bs J^t}  {{\rm det}(\bs J)},\quad n=\frac 1{{\rm det}(\bs J)}.
\end{equation}
To  represent $\bs J$ and its determinant in polar coordinates,
 we rewrite the above nonsingular transformation as $\tilde r=R(r,\theta),\, \tilde \theta= \Theta (r,\theta).$ 
Then by the chain rule, we have
\begin{equation*}
\frac{\partial (x,y)}{\partial(\tilde x,\tilde y)} \frac{\partial(\tilde x,\tilde y)}{\partial  (x,y)}= \frac{\partial(\tilde x,\tilde y)}{\partial  (\tilde r,\tilde \theta)}  \frac{\partial(\tilde r,\tilde \theta)}{\partial  (\tilde x,\tilde y)}= \frac{\partial( r, \theta)}{\partial  ( x, y)}  \frac{\partial( x, y)}{\partial  ( r, \theta)} =\bs I_2.
\end{equation*}
As a result,  the matrix $\bs J$ can be computed by
\begin{equation}\label{J2d1}
\begin{split}
\bs J&=\frac{\partial (x,y)}{\partial(\tilde x,\tilde y)}=\bigg(\frac{\partial(\tilde x,\tilde y)}{\partial  (x,y)}\bigg)^{-1}=\bigg(\frac{\partial(\tilde x,\tilde y)}{\partial  (\tilde r,\tilde \theta)}  \frac{\partial(\tilde r,\tilde \theta)}{\partial  ( r, \theta)}   \frac{\partial( r, \theta)}{\partial  ( x, y)}     \bigg)^{-1}  \\
&= \frac{\partial( x, y)}{\partial  ( r, \theta)}   \bigg(  \frac{\partial(\tilde r,\tilde \theta)}{\partial  ( r, \theta)} \bigg)^{-1} \frac{\partial(\tilde r,\tilde \theta)}{\partial  (\tilde x,\tilde y)}.
\end{split}
\end{equation}
Straightforward  calculation leads to
\begin{equation}\label{xyrt}
 \frac{\partial( x, y)}{\partial  ( r, \theta)}=\bs T(\theta)\begin{pmatrix}
 1& 0 \\
 0 & r
 \end{pmatrix}, \; \quad \frac{\partial(\tilde r,\tilde \theta)}{\partial  (\tilde x,\tilde y)}=
 \begin{pmatrix}
 1 & 0\\
 0   & 1/\tilde r
 \end{pmatrix}\bs T^t(\tilde \theta),
\end{equation}
and
\begin{equation}\label{tildert}
 \frac{\partial(\tilde r,\tilde \theta)}{\partial  ( r, \theta)}=\begin{pmatrix}
 R_r & R_{\theta}\\
  \Theta_r & \Theta_{\theta}
 \end{pmatrix},\quad\;\;   \Big(  \frac{\partial(\tilde r,\tilde \theta)}{\partial  ( r, \theta)} \Big)^{-1}=\frac{1}{R_r  \Theta_{\theta}- R_{\theta} \Theta_r}\begin{pmatrix}
 \Theta_{\theta} & -R_{\theta}\\
 - \Theta_r & R_r
 \end{pmatrix}.
\end{equation}
Inserting \eqref{xyrt}-\eqref{tildert} into \eqref{J2d1},  and using  the property:  ${\rm det}(\bs T(\theta))={\rm det}(\bs T(\tilde \theta))=1,$ we find
\begin{equation}\label{RT2rt}
{\rm det}(\bs J)=\frac{r}{R(R_r \Theta_{\theta}- R_{\theta} \Theta_r) }, \quad  \bs J= {{\rm det}(\bs J)}\,{\bs T}(\theta)
\begin{pmatrix}
R \Theta_{\theta}/r & - R_{\theta}/r  \\[1pt]
-R  \Theta_r & R_r  \\[1pt]
\end{pmatrix}
{\bs T}^t(\Theta).
\end{equation}

We now apply the above general formulas to \eqref{asBeqn}, that is,
 $\tilde r=R=S(r,\theta),\; \tilde \theta=\Theta=\theta,$
 so    $R_r=S_r,  R_\theta=S_{\theta}, \Theta_r=0, \Theta_\theta=1.$ Then we  obtain immediately from \eqref{RT2rt} that
\begin{equation}\label{invdetJ}
{\rm det}({\bs J})=\frac{r}{SS_r},\quad {\bs J}={\bs T}(\theta)   \begin{pmatrix}
\frac{1}{S_r} & -\frac{S_{\theta}}{S S_r}\\[4pt]
0 & \frac{r}{S}
\end{pmatrix} {\bs T}^t(\theta).
\end{equation}
Then,  ${\bs C}$ and $n$ in \eqref{Cn2dpolarform} can be derived directly from  \eqref{p3frame0} and \eqref{invdetJ}.
This ends the proof.
\end{proof}

\end{appendix}



\end{document}